\documentclass[aos]{imsart}

\RequirePackage[OT1]{fontenc}
\usepackage{bbm}
\usepackage[numbers]{natbib}
\usepackage{amsthm,amsmath}
\usepackage{graphicx,ifthen}
\usepackage{grffile}
\usepackage{latexsym}
\usepackage{mathrsfs}
\usepackage{amssymb}
\usepackage{hyperref}
\usepackage{xcolor}
\usepackage{caption, subcaption}

\startlocaldefs

\newtheoremstyle{example}{\topsep}{\topsep}%
     {}
     {}
     {\rmfamily}
     {}
     {\newline}
     {\thmname{#1}\thmnumber{ #2}\thmnote{ #3}}

   \theoremstyle{example}

\def\min{\text{min}}
\newcommand{\indic}{\mathbb{I}}

\numberwithin{equation}{section}
\theoremstyle{plain}
\newtheorem{thm}{Theorem}[section]

\newtheorem{prop}{Proposition}[section]
\newtheorem{lem}{Lemma}[section]
\newtheorem{rem}{Remark}[section]

\newtheorem{cor}{Corollary}[section]

\newcommand{\Lower}[2]{\smash{\lower #1 \hbox{#2}}}
\newcommand{\ben}{\begin{enumerate}}
\newcommand{\een}{\end{enumerate}}
\newcommand{\bi}{\begin{itemize}}
\newcommand{\ei}{\end{itemize}}

\endlocaldefs

\begin{document}

\begin{frontmatter}
\title{Posterior distributions for Hierarchical Spike and Slab Indian Buffet processes} \runtitle{HIBP}

\begin{aug}
\author{\fnms{Lancelot F.} \snm{James}\thanksref{t1}\ead[label=e1]{lancelot@ust.hk;abhinav.pandey@ust.connect.hk}},
\author{\fnms{Juho} \snm{Lee}\thanksref{t2}\ead[label=e2]{juholee@kaist.ac.kr}},
\and
\author{\fnms{Abhinav} \snm{Pandey}\thanksref{t1}\ead[label=e3]{abhinav.pandey@ust.connect.hk}}
\affiliation{Hong Kong University of Science and Technology\\Korea Advanced Institute of Science and Technology}
\thankstext{t1}{Supported in part by the
grants RGC-GRF 16300217 and T31-604/18-N of the HKSAR.}
\thankstext{t2}{
Supported by Institute of Information \& communications Technology Planning \& Evaluation (IITP) grant funded by the Korea government (MSIT) (No.2019-0-00075, Artificial Intelligence Graduate School Program (KAIST).)
}

\runauthor{James, Lee, and Pandey}

\address{Lancelt F. James and Abhinav Pandey\\
The Hong Kong University of Science and
Technology, \\Department of Information Systems, Business Statistics\\ and Operations Management,\\
Clear Water Bay, Kowloon, Hong Kong.\\ 
\printead{e1}\\}

\address{Juho Lee\\
The graduate school of AI,\\KAIST\\
Room 2223, N5 bldg\\
Yooseong-gu, Daejeon\\ 
Republic of Korea\\
\printead{e2}\\}

\end{aug}

\begin{abstract}Bayesian nonparametric hierarchical priors are highly effective in providing flexible models for latent data structures exhibiting sharing of information between and across groups. 
Most prominent is the Hierarchical Dirichlet Process (HDP), and its subsequent variants, which model latent clustering between and across groups. The HDP, may be viewed as a more flexible extension of Latent Dirichlet Allocation models~(LDA), and has been applied to, for example, topic modelling, natural language processing, and datasets arising in health-care. We focus on analogous latent feature allocation models, where the data structures correspond to multisets or unbounded sparse matrices. The fundamental development in this regard is the Hierarchical Indian Buffet process (HIBP), which utilizes a hierarchy of Beta processes over $J$ groups, where each group generates binary random matrices, reflecting within group sharing of features, according to beta-Bernoulli IBP priors. To encompass HIBP versions of non-Bernoulli extensions of the IBP, we introduce hierarchical versions of general spike and slab IBP. We provide explicit novel descriptions of the marginal, posterior and predictive distributions of the HIBP and its generalizations which allow for exact sampling and simpler practical implementation. We highlight common structural properties of these processes and establish relationships to existing IBP type and related models arising in the literature. Examples of potential applications may involve topic models, Poisson factorization models, random count matrix priors and neural network models
\end{abstract}


\end{frontmatter}
\section{Introduction}
The Indian Buffet process~(IBP) introduced by Griffiths and Ghahramani~\cite{GriffithsZ,Griffiths1}, represents an important amalgam of ideas from Bayesian non-parametric statistics~(BNP) and statistical machine learning, as it relates to the creation and  analysis of more flexible latent feature/factor models. In particular, this framework allows for sharing and learning of, for example, latent features and associated counts, which avoid issues of model selection arising in parametric latent factor models.
As described in~Thibaux and Jordan~\cite{Thibaux}, the IBP may be expressed in terms of $Z_{1},\ldots,Z_{M}|\mu$ conditionally iid Bernoulli processes, where $\mu$ is a Beta process~\cite{hjort}, which is otherwise a special case of a completely random measure (CRM) with countable jumps in $[0,1]$ and atoms(corresponding to features) drawn from a finite measure $B_{0}.$ The marginal distribution of $Z_{1},\ldots,Z_{M}$ produces sparse random binary matrices with $M$ rows and a potentially infinite set of columns. A sequential scheme to generate this, referred to as the Indian Buffet Process, is facilitated by an explicit description of the predictive distribution of $Z_{M+1}|Z_{1},\ldots,Z_{M},$ which in turn may be constructed from the posterior distribution of $\mu|Z_{1},\ldots,Z_{M}.$ A thorough understanding of these components then allows for their practical utility in a wide range of factor models where observed data structures, say $\mathbf{W},$ conditional on $(Z_{1},\ldots,Z_{M})=\mathbf{z},$ may have the form $\Phi(\mathbf{W}|\mathbf{z}).$ 
The IBP has inspired the development and investigation of generalizations which induce non-binary entries, with the bulk of these corresponding to Poisson and Negative Binomial count models. Here we list some applications of the IBP and these extensions~\cite{Broderick2,Caron2012,Croy,KnowlesThesis,KnowlesAAS, MasoeroBiometrika,Scheinthesis,Titsias,
ZhouCarin2015,Zhou1,ZhouPadilla2016}. 
These are examples of \textit{trait allocation models} as defined in~\cite{CampbellTrait}. In particular, they provide the following description ``Trait allocations apply when data may
belong to more than one group (a trait), and may exhibit non-negative integer levels of belonging in each group."
See~\cite{BroderickWilson} for relevant analysis related to such models. \cite{James2017} provides a unified framework for the posterior analysis of IBP type models, which produce sparse random matrices with arbitrary valued entries, by introducing what we call here,  \textit{spike and slab} IBP priors, based on random variables having mass at zero, but otherwise producing non-zero entries corresponding to any random variable. Furthermore, important to the present discussion~\cite[Section 5]{James2017} also introduces and provides posterior analysis for multivariate IBP priors. A recent work of \cite{AyedCaron} shows that their model for overlapping community detection may be expressed as a case of~\cite[Section 5]{James2017}. Furthermore, it is interesting to note that \cite[Section 6]{AyedCaron} describe the potential usage of conditionally Gamma and Normally distributed entries. In addition see~\cite{Basbug, Ragadeep,Scheinthesis,ZhouPGBN} for potential ideas in regards to matrices with general spike   and slab entries. 

Our focus in this work is to provide explicit novel descriptions of the marginal, posterior and predictive distributions of hierarchical versions of general spike and slab IBP, which allow for exact sampling and simpler practical implementation of these processes. This includes descriptions for the fundamental development in this regard, which is the Hierarchical Indian Buffet Process (HIBP)~\citep{Thibaux}. As an example, one can use our Corollary~\ref{bergenpredict} to describe explicitly the HIBP analogue of the IBP sequential scheme, which is otherwise unavailable in the literature. The HIBP utilizes a hierarchy of Beta processes over $J$ groups to induce sharing across group specific Bernoulli IBP, each consisting of $M_{j}$ processes for $j=1,\ldots,J.$ Specifically this is done
by setting $\mu_{1},\ldots,\mu_{J}|B_{0}$ to be conditionally independent Beta processes with common base measure $B_{0}$ and then specifying $B_{0}$ to also be a random Beta process. The authors discuss applications to topic models, but there are many possibilities, for example applications to neural nets and imaging as described in~\cite{BoChen,GuptaRestrict,Kessler, ZhoudHIBP}. Here we use general (conditonal) CRM specifications for $\mu_{j}$ and $B_{0},$ and otherwise apply these to the constructions in~\cite[Section 2]{James2017}. It is important for us to point out that while our results are novel in every case, the paper of~\cite{Masoerotrait} extends the HIBP framework to trait allocation models which corresponds to the important case of generalized HIBP with non-negative integer valued entries.  \cite{Masoerotrait} was done prior to this work and consists of two results~\cite[Theorems 1,2]{Masoerotrait}, which most closely parallels our forthcoming results in Propositions~\ref{marginalslabHIBP}, \ref{postBgivenZq}, and \ref{propmupostrep},
however the descriptions are quite different. 

The HIBP and its generalizations to spike and slab models, draws their conception from, and  are otherwise the analogue of BNP hierarchical priors for latent clustering models. The seminal work is the Hierarchical Dirichlet process(HDP)~\cite{HDP} which may be viewed as  a flexible extension of Latent Dirichlet allocation models~\cite{BleiLDA}, with for example, applications to topic models. We now briefly describe some elements of that construction. As the name suggests, the HDP specifies discrete random probability measures $P_{1},\ldots,P_{J}|P_{0}$ to be conditionally independent Dirichlet processes~\cite{Ferg1973}, with common base measure $P_{0}.$ (Hereafter we shall use the notation $[n]=\{1,\ldots,n\}$ and write $i\in [n].$)
Let us for a moment assume that $P_{0}$ is non-atomic, then in the presence of latent data $Y_{j,1},\ldots, Y_{j,M_{j}}|P_{j},P_{0}\overset{iid}\sim P_{j},$ the marginal process produces $K_{j,M_{j}}=k_{j}\leq M_{j}$ unique values, say $(Y^{*}_{j,l},l\in[k_{j}])\overset{iid}\sim P_{0}$ and creates clusters $C_{j,l}=\{i:Y_{j,i}=Y^{*}_{j,l}\}$ forming a random partition of $[M_{j}],$ corresponding to a Chinese Restaurant process~\cite{Pit96,Pit06}. 
A common and important application is to mixture models, see~\cite{Lo1984},  where observables $W_{j,1},\ldots,W_{j,M_{j}}|(Y_{j,i},i\in [M_{j}])$ are conditionally independent with respective densities, or pmf, $k(W_{j,i}|Y_{j,i})$, for $i\in [M_{j}].$ \cite{IJ2001,IJ2003} show that these latent clustering mechanisms, exhibiting sharing within a group, and related stick-breaking constructions, carry over to quite general choices for $P_{j}.$ In any event, if $P_{0}$ is non-atomic, there is no sharing of information (atoms) across groups. Sharing across groups is achieved in~\cite{HDP} by specifying $P_{0}$ to also be a Dirichlet process.  Subsequent work, see \cite{Goldwater,TehPY}, extended this framework to Pitman-Yor processes(which are described in~\cite{IJ2001,Pit96}). More recently, posterior analysis of a more general class of hierarchical  $P_{1},\ldots,P_{J}$ has beeen done in~\cite{CamerLP1}. See, additionally, \cite{Argiento,BuntineHPYtwitter,Wood} for some examples of applications. Also recently, the general L\'evy moving average/smooth hazard rate framework in~\cite[Section 4]{James2005}, which generalizes the gamma process formulation in~\cite{DykstraLaud,LoWeng89,IJ2004},  has been extended to a hierarchical setting, suitable for data sets corresponding to multiplicative intensity models, in~\cite{CamerLP2}, by utilizing kernel smoothed hierarchical CRM, of the form $\int_{0}^{\infty}k(t|s)\mu_{j}(ds),$ as dependent random hazard rates. Where 
$\mu_{1},\ldots,\mu_{J}|B_{0}$ and $B_{0}$ have the same construction as in our models.  See also~\cite{Broderick1,CampbellTrait,RagaBlei} for relevant general discussions.
\subsection{Outline}
We now present some key highlights of this work. We first note that much of our analysis is facilitated by the use of  the results on multivariate type IBP in~\cite[Section 5]{James2017}, coupled with key representations that we derive.
Section 2.1 describes a general version of the Bernoulli based HIBP, where we note that $B_{0}$ based on a generalized gamma process may  be used. In fact, as described in Sections 2.4, 2.6 and Section 3,  one of our key findings, highlighted in Proposition~\ref{PropBernoulliPoissonsharing}  and Proposition~\ref{multisample}, is that there is a common multivariate IBP Poisson structure, with similarities to~\cite{Caron2012, Lo1982, Titsias}, that essentially dictates the generation and sharing of features across groups no matter the spike and slab distribution used. Hence this choice of  prior for $B_{0}$ is easily applied in every spike and slab case. As such, we shall highlight calculations with respect to this choice of prior for $B_{0}$ throughout. Section 2.2. presents a general description of spike and slab HIBP priors based on the framework in~\cite[Section 2]{James2017}. Section 2.3-2.5 provide various representations that allow one to make connections with existing IBP type  models. In particular, besides the developments mentioned in Section 2.4, Section 2.5.1. describes relationships to models discussed in \cite{ZhouCarin2015,Zhou1, ZhouPadilla2016}. Section 2.6, Proposition~\ref{bigmargBgen} and Corollary~ \ref{corjointB}\label{bigmargBgen}, shows that our hierarchical priors may be expressed as multivariate IBP based on entries corresponding to general compound Poisson processes. 
It is noteworthy that this structure suggests natural computational routines which may be seen as generalizations of methods employed in~\cite{ZhouCarin2015,Zhou1, ZhouPadilla2016} based on the Logarithmic series distribution compound Poisson representation of negative binomial distributions in~\cite{Quenouille}. Furthermore, the general compound Poisson process factorization framework in~\cite{Basbug}, should offer indications of usage for our priors in various models.  Based on the representations in Section 2.6, Section 3 provides explicit results that are amenable to sampling in parallel to what is known for standard spike and slab IBP, we also show connections to ECPF as described in~\cite{ZhouFoF}. We mention here Proposition~\ref{multisample}, Theorem~\ref{thmHIBPmarginalgen}, and explicit prediction rules derived in Section 3.5, which are all directly pertinent to sampling. In particular, similar to what is mentioned and done in~\cite{ZhouPadilla2016}, our Theorem~\ref{thmHIBPmarginalgen} allows one to generate quite intricate sparse random matrices with arbitrary entries at once, whilst our results in Proposition~\ref{postgenpredict} for the prediction distributions allows for the adding of additional rows and columns sequentially. Sections 3.6 and 3.7 provide explicit details in the case of Poisson and Bernoulli based HIBP. However, we may easily mix various choices of spike and slab distributions. In Section 4, we provide some simple simulations.

\subsection{Indian buffet process}
We first discuss more details of  the basic Indian Buffet process as devised in~\cite{GriffithsZ,Griffiths1}. The Indian Buffet process sequential scheme is cast in terms of $M$ customers sequentially selecting dishes from an Indian Bufftet, where each dish represents features or attributes of that customer, the mechanism allows for the selection of dishes previously chosen by other customers as well as a mechanism to choose new dishes is described as follows for the simplest case. For a parameter $\theta>0,$ the first customers selects a $\mathrm{Poisson}(\theta)$ number of dishes drawn from a non-atomic finite measure  $B_{0},$ given the selection process of the first $M$ customers the ($M+1$) st customer will choose new dishes according to a $\mathrm{Poisson}(\theta/(M+1))$ variable and chooses each previously chosen dish according to Bernoulli random variables, where the most popular dishes have the highest chance of being chosen. 
As shown in \cite{Thibaux}, this process may be generated by expressing the features of each customer in terms of conditionally independent Bernoulli processes say 
$Z^{(1)},\ldots,Z^{(M)}$  defined for each $i$ as,
$$
Z^{(i)}=\sum_{k=1}^{\infty}b^{(i)}_{k}\delta_{\omega_{k}}
$$
where $(b^{(i)}_{k})$ for $i=1,\ldots,M$ are conditionally independent such that for each $(i,k)$ $b^{(i)}_{k}\sim\mathrm {Bernoulli}(p_{k}),$ where the $(p_{k},k=1,\ldots,):=(p_{k})$ are points of a Poisson random measure with mean measure a L\'evy measure of the form $\rho(s)=\theta s^{-1}\mathbb{I}_{\{0\leq p\leq 1\}}.$ The $(b^{(i)}_{k})$ form priors over equivalence classes of binary (random) matrices and are associated with a potentially infinite collection of features $(\omega_{k})$ that are drawn from a non-atomic finite measure $B_{0}(\cdot)$ over a Polish space $\Omega.$ In particular, since $B_{0}({\Omega})$ is finite, the
 $(\omega_{k})$ are iid $\bar{B}_{0}(dx)=B_{0}(dx)/B_{0}({\Omega}).$
The assignment of the $(p_{k})$ is expressed through a completely random measure $\mu$ with jumps in $[0,1],$ expressed as 
$\sum_{k=1}^{\infty}p_{k}\delta_{\omega_{k}}.$ For the current choice of $\rho(p)=\theta p^{-1}\mathbb{I}_{\{0\leq p\leq 1\}},$ $\mu$ is a simple homogeneous case of Beta processes introduced in~\cite{hjort}, see also~\cite{KimY}, with law denoted as 
$\mu\sim \mathrm{BP}(\theta,B_{0}),$ corresponding $Z^{(1)},\ldots,Z^{(M)}|\mu$ are iid Bernoulli processes with common law denoted as $\mathrm{BerP}(\mu)$, The IBP sequential scheme is obtained from the conditional distribution of $Z^{(M+1)}|Z^{(1)},\ldots,Z^{(M)}$ which is facilitated by a description of the posterior distribution of $\mu|Z^{(1)},\ldots,Z^{(M)}.$ The IBP has been extended to more general choices for $\mu,$ including the richer class of stable-Beta, see~\cite{TehG}, see also~\cite{KimLee}, with jumps $(p_{k})$ specified by a L\'evy density, 
\begin{equation}
\rho_{\alpha,\beta}(p)=\theta p^{-\alpha-1}(1-p)^{\beta
+\alpha-1}\indic_{\{0<p<1\}},
\label{stablebetauni}
\end{equation}
for $\beta>-\alpha$ and $0\leq \alpha<1.$ Here we will say $\mu$ is a stable-Beta process with parameters $(\alpha,\beta;\theta),$ with distributional notation $\mu\sim \mathrm{sBP}((\alpha,\beta),\theta,B_{0})$
One may use any $\mathrm{CRM}$ $\mu$ provided its jumps are in $[0,1].$ In general, 
we use the notation $\mu\sim \mathrm{CRM}(\rho, B)$ to denote a general $\mathrm{CRM},$ with jumps in $(0,\infty)$ over $\Omega$ where the L\'evy density $\rho(s)$ satisfies 
$\int_{0}^{\infty}\min(s,1)\rho(s)ds<\infty,$ and $B(d\omega)$ is a finite measure over $\Omega.$ 

\begin{rem} As in~\cite{James2002,James2005,JamesNTR,JLP2,James2017}, we exploit the representations $\mu(d\omega)=\int_{0}^{\infty}sN(ds,d\omega),$ where $N$ is a Poisson random measure with mean measure $\mathbb{E}(N(ds,dw)]=\nu(ds,d\omega)=\rho(s)B(d\omega),$ we write $N\sim \mathrm{PRM}(\nu).$ 
\end{rem} 
\section{Bernoulli process HIBP and extensions to multiple Hierarchical spike and slab IBP }
\subsection{The  Beta process HIBP and extensions in the Bernoulli case}\label{Bernoulliclassic}
We first sketch out the HIBP as described in the setting of  topic models in~\cite{Thibaux}. Consider a corpora consisting of $j=1,\ldots,J$ types of documents, where there are $M_{j}$ documents of each type $j\in [J]$. The documents are modelled as follows;
for $j\in [J],$ model documents of each type $j$ as conditionally independent Bernoulli processes as follows,
\begin{equation}
Z^{(1)}_{j},\ldots, Z^{(M_{j})}_{j}|\mu_{j}\overset{iid}\sim \mathrm{BeP}(\mu_{j})
\label{Bernoullilevel1}
\end{equation}
$\mu_{1},\ldots,\mu_{J}$ represent priors over the selection of words of each type and are modelled as conditionally independent Beta processes given a common measure $B_{0}$ over the feature space. More precisely $\mu_{1},\ldots,\mu_{J}|B_{0}$ are conditionally independent such that for $j\in [J],$
\begin{equation}
\mu_{j}|B_{0}\sim \mathrm{BP}(\theta_{j},B_{0}).
\label{givenB1}
\end{equation}

If $B_{0}$ is non-atomic then there are sharing of words within each document type $j$, according to usual Bernoulli IBP mechanisms but not sharing of information across document types. Similar to the HDP, in order to allow for sharing across document types \cite{Thibaux} specify a Beta process prior for $B_{0},$ that is
\begin{equation}
B_{0}\sim \mathrm{BP}(\theta_{0},G_{0}). 
\label{BetaB}
\end{equation}

While the choice of $B_{0}$ as  Beta process might seem a natural analogue of the use of layers of Dirichlet processes  in the HDP, we note that since $B_{0}$ is a finite measure it need not have jumps restricted to $[0,1],$ hence as a slight, but important, extension to the HIBP framework we can choose $B_{0}=\sum_{k=1}^{\infty}\tau_{k}\delta_{Y_{k}}$ to be a general $\mathrm{CRM}(\tau_{0},G_{0})$ where $\tau_{0}(\lambda)$ is now a general  L\'evy density on $(0,\infty),$ with Laplace exponent, for $\kappa>0,$
\begin{equation}
\tilde{\psi}_{0}(\kappa)=\int_{0}^{\infty}(1-{\mbox e}^{-\lambda\kappa})\tau_{0}(\lambda)d\lambda.
\label{BLaplaceexponent}
\end{equation}
The choice of $B_{0}$ as a generalized gamma process will be a featured example in this work due to its relevant flexible distributional properties and its  tractability, see for instance~\cite{Caron2012,CaronFox, James2002,James2005,JamesNTR,James2017,ZhouFoF}. Specifically we say, that $B_{0} ,$ is a generalized gamma process with law denoted as $
\mathrm{GG}(\alpha,\zeta;\theta_{0}G_{0})$ if 
\begin{equation}
\tau_{0}(\lambda):=\tilde{\tau}_{\alpha}(\lambda|\zeta,\theta_{0})=\frac{\theta_{0}}{\Gamma(1-\alpha)}\lambda^{-\alpha-1}{\mbox e}^{-\zeta \lambda}{\mbox { for }}0<\lambda<\infty,
\label{GGLevy}
\end{equation}
for $\theta_{0}>0,$ and the ranges $0<\alpha<1,$ $\zeta\ge 0,$ or $\alpha\leq 0$ 
and $\zeta>0.$ When $\alpha=0$ this is the case of the gamma process. That is to say $\mathrm{GG}(0,\zeta;\theta_{0})$ corresponds to a Gamma process with shape $\theta_{0}$ and scale $1/\zeta.$ When $\alpha<0,$ this results in a class of gamma compound Poisson processes. 
In particular, as a special case of~(\ref{BLaplaceexponent}), $\tilde{\psi}_{0}(\kappa)=\theta_{0}\tilde{\psi}_{\alpha,\zeta}(\kappa)$ where
\begin{equation}
\label{rfunction}
\tilde{\psi}_{\alpha,\zeta}(\kappa)=\left\{\begin{array}{ll}
 \frac{1}{\alpha}[{(\kappa+\zeta)}^{\alpha}-{\zeta}^{\alpha}], &   {\mbox{ if }} 0<\alpha<1, \zeta\ge 0\\ 
{\log(1+\kappa/\zeta)},& {\mbox{ if }} \alpha=0, \zeta>0\\
   \frac{1}{\delta}[{\zeta}^{-\delta}-{(\kappa+\zeta)}^{-\delta}], &   {\mbox{ if }} \alpha=-\delta<0,\zeta>0\\
                     \end{array}\right.
\end{equation}
In addition, as further natural extensions of the classic HIBP, we may replace the Beta process specifications of 
$\mu_{1},\ldots,\mu_{J}|B_{0}$ in~(\ref{givenB1}) with the case where the $\mu_{j}|B_{0}$ are conditionally independent $\mathrm{CRM}(\rho_{j},B_{0})$ with the constraint that the L\'evy density $\rho_{j}$ dictate jumps of $\mu_{j}$ restricted to $[0,1],$ for $j\in [J].$ The work of~\cite{TehG}, and subsequent applications in the IBP case, sugggest that the choice of $\mu_{j}|B_{0}\overset{ind}\sim\mathrm{sBP}((\alpha_{j},\beta_{j}),\theta_{j},B_{0}),$ $j\in[J],$ would be ideal choices both in terms of tractability, and model flexibility in that it allows fitting $((\alpha_{j},\beta_{j}),j\in[J])$ to various types of data structures which cannot be captured by the standard Beta process. 

\subsection{Mixed hierarchical generalized spike and slab Indian Buffet processes}
We now show how to construct and combine Hierarchical versions of processes based on natural extensions of the $\mathrm{IBP}$ that have appeared in the literature.  We use the unified framework following~\cite[Secton 2.1]{James2017}, for each $j\in[J],$ let $A_{j}$ denote a random variable such that given a parameter,$s_{j},$ the distribution of $A_{j}|s_{j}$ is given by the spike and slab probability measure,
\begin{equation}
G_{A_{j}}(da|s_{j})=[1-\pi_{A_{j}}(s_{j})]\delta_{0}(da)+\pi_{A_{j}}(s_{j})
\tilde{G}_{A'_{j}}(da|s_{j})
\label{spikeslabdist1}
\end{equation}
where $\tilde{G}_{A'_{j}}(da|s_{j})$ corresponds to a proper distribution~(slab distribution) of a random variable $A'_{j}$ which does not take mass at $0,$  and 
$\mathbb{P}(A_{j}=0|s_{j})=1-\pi_{A_{j}}(s_{j})>0$ is the spike. Equivalently, $A_{j}\overset{d}=b_{j}A'_{j},$  where $b_{j}\overset{d}=\mathbb{I}_{\{A_{j}\neq 0\}}|s_{j}\sim \mathrm{Bernoulli}(\pi_{{A}_{j}}(s_{j})),$ and it follows that (\ref{spikeslabdist1}) can be expressed as
\begin{equation}
G_{A_{j}}(da|s_{j})={[1-\pi_{A_{j}}(s_{j})]}^{1-b(a)}{[\pi_{A_{j}}(s_{j})
\tilde{G}_{A'_{j}}(da|s_{j})]}^{b(a)}
\label{spikeslabdist2}
\end{equation}
for $b(a):=\mathbb{I}_{\{a\neq 0\}}.$ Now, given $B_{0}=\sum_{k=1}^{\infty}\tau_{k}\delta_{Y_{k}},$ let, for each fixed $j\in[J],$  $((A^{(i)}_{j,k}, i\in [M_{j}]),s_{j,k},\omega_{j,k})),$ denote points of a Poisson random measure with mean intensity 
\begin{equation}
\left[\prod_{i=1}^{M_{j}}G_{A_{j}}(da^{(i)}_{j}|s_{j})\right]\rho_{j}(s_{j})B_{0}(d\omega),
\label{PoissonMeanjoint}
\end{equation}
where $\mu_{j}=\sum_{k=1}^{\infty}s_{j,k}\delta_{\omega_{j,k}}|B_{0}$ is $\mathrm{CRM}(\rho_{j},B_{0})$ where now $\rho_{j}$ is a 
L\'evy density on $(0,\infty),$  for $j\in[J].$  Now as in~\cite[Section 2 eq.(2.1), Section 3 eq(3.1)]{James2017}, define thinnings of $\mu_{j},$ by 
$\rho_{j,i-1}(s)=[1-\pi_{A_{j}}(s)]^{i-1}\rho_{j}(s)$ for $i=1,2\ldots,$ where $\pi_{A_{j}}(s)$  must satisfy
\begin{equation}
\gamma_{j,i}=\int_{0}^{\infty}\pi_{A_{j}}(s)\rho_{j,i-1}(s)ds<\infty
\label{gamij}
\end{equation}  
with, see \cite[Appendix, eqs.(A.2),(A.4)]{James2017},
\begin{equation}
\psi_{j}(M_{j})=\sum_{i=1}^{M_{j}}\gamma_{i,j}=\int_{0}^{\infty}(1-{[1-\pi_{A_{j}}(s)]}^{M_{j}})\rho_{j}(s)ds.
\label{genpsi}
\end{equation}

Then we can construct a mixed collection of hierarchical generalized spike and slab Indian Buffet processes as follows
\begin{equation}
((Z^{(i)}_{j},i\in [M_{j}]),j\in[J])=((\sum_{k=1}^{\infty}A^{(i)}_{j,k}\delta_{\omega_{j,k}},i\in [M_{j}]),j\in[J])
\label{mixedHIBP}
\end{equation}
Similar to \cite{James2017}, we say that  $((Z^{(i)}_{j},i\in [M_{j}]),j\in[J])|\mu_{1},\ldots,\mu_{J},B_{0}$ are conditionally independent across $(j\in[J])$ with 
$(Z^{(i)}_{j},i\in [M_{j}])|\mu_{j},B_{0}$ conditionally iid $\mathrm{IBP}(G_{A_{j}}|\mu_{j}),$ interchangeably we shall also say such processes are $\mathrm{IBP}(A_{j}|\mu_{j}).$ 

As prominent examples, set $A_{1}|p\sim \mathrm{Bernoulli}(p)$ with conditional pmf $p^{a}(1-p)^{1-a}$ for $a\in \{0,1\}$ and hence the slab is $A'_{1}=1$ and spike is $1-\pi_{A_{1}}(p)=(1-p).$ In this case $(Z^{(i)}_{1})$ are conditionally iid 
$\mathrm{BeP}(\mu_{1}).$ Set $A_{2}|s\sim \mathrm{Poisson}(rs)$ with pmf $(rs)^{a}{\mbox e}^{-rs}/a!,$ $a=0,1,2,\ldots,$ 
with spike $1-\pi_{A_{2}}(s)={\mbox e}^{-rs},$ and slab $A'_{2}|s$ corresponding to a zero truncated Poisson distribution, say $A'_{2}|s\sim \mathrm{tPoissson}(rs),$ with pmf
\begin{equation}
\mathbb{P}(A'_{2}=a|s)\frac{{(rs)}^{a}{\mbox e}^{-rs}}{a!(1-{\mbox e}^{-rs})}{\mbox { for }}a=1,2,\ldots 
\end{equation}
We can say that $(Z^{(i)}_{2})$ are conditionally iid 
$\mathrm{IBP}(\mathrm{Poisson}(r)|\mu_{2})$ or simply $\mathrm{PoiP}(r,\mu_{j}).$
In the case where $A_{3}|p$ is Negative-Binomial $(r,p),$  denoted as $\mathrm{NB}(r,p),$ with $
\mathbb{P}(A_{3}=a|p)={a+r-1\choose a}p^{a}{(1-p)}^{r}, a=0,1,\ldots.
$
It follows that $\pi_{A_{3}}(p)=1-(1-p)^{r},$ and $A'_{3}$ has a discrete slab distribution  $\tilde{G}_{A'_{3}},$ specified for $a=1,2,\ldots$
by
\begin{equation}
\mathbb{P}(A'_{3}=a|p)=\frac{{a+r-1\choose a}p^{a}{(1-p)}^{r}}
{1-(1-p)^{r}}.
\label{NBslab}
\end{equation}
In general $A'_{j}$ may have continuous distributions, such as Normal or Gamma variable. 
\begin{rem}We shall also view the processes in~(\ref{mixedHIBP}) as multivariate IBP in the sense of \cite[Section 5]{James2017}, for instance, $(Z^{(i)}_{j},i\in [M_{j}])|\mu_{j},B_{0}$ is conditionally a multivariate $\mathrm{IBP}((A^{(1)}_{j},\ldots,
A^{(M_{j})}_{j})|\mu_{j}).$ See \cite{AyedCaron,ZhouPadilla2016} for applications of related constructions for a single $j.$
\end{rem}
\begin{rem}We note that results for extensions to inhomogeneous $\rho_{j}(s|\omega)$ or hierarchies of the general multivariate IBP in~\cite[Section 5]{James2017}, will follow from the arguments similar to that used here coupled with the results in~\cite{James2017}. 
\end{rem}
\subsection{Hierarchical representations}\label{Hrep}
We next describe initial alternative representations. 
\begin{prop}\label{HIBPsumrep}For each fixed $(j,k),$ let
 $\mathbf{s}_{j,k}:=((s_{j,k,l})),$ such that given $B_{0},$ $\mathbf{s}_{j,k}$ are the points of a $\mathrm{PRM}$ with mean $\tau_{k}\rho_{j},$ and are conditionally independent across $j,k.$ Furthermore, define independent subordinators $(\sigma_{j}(t),t\ge 0)$ with $-\log \mathbb{E}[{\mbox e}^{-\sigma_{j}(t)}]=t\int_{0}^{1}(1-{\mbox e}^{-s})\rho_{j}(s)ds,$ and let for each $j,$ $((\sigma_{j,k}(t)))$ denote iid copies of $\sigma_{j}.$ Then there are the following representations
\begin{enumerate}
\item[(i)]
$\mu_{j}\overset{d}=\sum_{k=1}^{\infty}[\sum_{l=1}^{\infty}s_{j,k,l}]\delta_{Y_{k}}\overset{d}=\sum_{k=1}^{\infty}\sigma_{j,k}(\tau_{k})\delta_{Y_{k}},$
\item[(ii)]$Z^{(i)}_{j}\overset{d}=\sum_{k=1}^{\infty}\left[\sum_{l=1}^{\infty}A^{(i)}_{j,k,l}\right]\delta_{Y_{k}}$ where $A^{(i)}_{j,k,l}|s_{j,k,l}\overset{ind}\sim G_{A_{j}}(da|s_{j,k,l}).$
\end{enumerate}
\end{prop}
\begin{proof}Expanding $B_{0}=\sum_{k=1}^{\infty}\tau_{k}\delta_{Y_{k}},$ in~(\ref{PoissonMeanjoint}), leads to the mean measures $\left[\prod_{i=1}^{M_{j}}G_{A_{j}}(da^{(i)}_{j}|s_{j})\right]\tau_{k}\rho_{j}(s_{j})\delta_{Y_{k}}(d\omega),$ for  $k=1,2,\ldots,$  and each fixed $j\in[J].$ The result can then be verified by looking at the respective Laplace functionals given $B_{0}.$ 
\end{proof}

For fixed $\lambda,$ $\sigma_{j}(\lambda)$ correspond to random variables with non-negative infinitely divisible distributions. In particular, if $\rho_{j}$ is a finite measure, then $\sigma_{j}(\lambda)$ correspond to a compound Poisson distribution, otherwise it corresponds to continuous infinitely divisible variables on $(0,\infty).$  Notice that even if $s_{j,k,l}$ take values in $[0,1],$ the infinite sums in the form of $\sigma_{j,k}(\lambda)$ can take values larger than $1$ with positive probability. As an example if $\mu_{j}|B_{0}\sim \mathrm{BP}(\theta_{j},B_{0}).$ as in (\ref{givenB1}), $\sigma_{j,k}(\lambda)$ has a generalized Dickman distribution~\cite{Penrose}. 
This leads to the following result, which is a variant of known results. 
\begin{cor}Denote the distribution  of $\sigma_{j}(\lambda),$ for fixed $\lambda,$ as $f_{j}(v|\lambda):=\mathbb{P}(\sigma_{j}(\lambda)\in dv)/dv.$ Then the joint un-conditional distribution of $\mu_{0}:=(\mu_{j},j\in [J]),$ corresponds to a sub-class of multivariate CRM, as described in~\cite[Section 5]{James2017}, constructed from a Poisson random measure with mean intensity
$\rho_{0}(v_{1},\ldots,v_{J})G_{0}(dy),$ where 
$
\rho_{0}(v_{1},\ldots,v_{J})=\int_{0}^{\infty}\left[\prod_{j=1}^{J}f_{j}(v_{j}|\lambda)\right]\tau_{0}(\lambda)d\lambda
$ 
is the joint L\'evy density of $\mu_{0}.$ 
\end{cor}

\begin{rem}
Dependent $\mathrm{CRM}$ defined in the same hierarchical manner as $(\mu_{j},j\in[J])$ are employed in \cite{CamerLP2}. See~\cite[Section 2]{CamerLP2} for other equivalent distributional characterizations using joint Laplace functionals. 
\end{rem}
  
\subsection{Relationship between Bernoulli HIBP and Poisson IBP models}In Proposition~\ref{HIBPsumrep}, if $A_{j}|p$ is $\mathrm{Bernoulli}(p)$ then  $A^{(i)}_{j,k,l}|s_{j,k,l}\overset{ind}\sim  \mathrm{Bernoulli}(s_{j,k,l})$ and $\sum_{i=1}^{M_{j}}A^{(i)}_{j,k,l}\overset{d}=B_{j,k,l},$ where $B_{j,k,l}\sim \mathrm{Binomial}(M_{j},s_{j,k,l}).$
Since, it is well known for a standard Bernoulli process that its sum over $\Omega$ is Poisson distributed, it follows that marginally $\sum_{l=1}^{\infty}A^{(i)}_{j,k,l},$ are conditional on $B_{0},$ Poisson random variables.  We now focus on the across group behavior of the Bernoulli based HIBP and show that given $B_{0}$ the univariate process equates in distribution to across and within  group  parameter specific variants of the Poisson IBP in~\cite{Titsias}, and otherwise has similarities to the structure of the processes in~\cite{Caron2012} . This across group Poisson structure will play a fundamental role in the characterization of all mixed spike and slab HIBP models. In this section $\gamma_{j,i}=\int_{0}^{1}p(1-p)^{i-1}\rho_{j}(p)dp,$ with 
$\psi_{j}(M_{j})=\sum_{i=1}^{M_{j}}\gamma_{j,i}=\int_{0}^{1}(1-{(1-p)}^{M_{j}})\rho_{j}(p)dp$ however see Remark~
\ref{genAbernoulli} below.

Suppose for a moment that $B_{0}$ has no atoms, then in  the case of the standard Bernoulli IBP, viewed sequentially $(Z^{(i)}_{j},i\in[M_{j}])|B_{0}$ consists of Poisson numbers of (new/unique) dishes selected by each new  customer and otherwise points shared among the previous customers.  Using known characterizations of the predictive and marginal distributions~\cite{GriffithsZ,Griffiths1,Thibaux}, as expressed in~\cite[Propositions 3.2 and 3.3]{James2017} for $A_{j}|p\sim \mathrm{Bernoulli}(p),$ the $M_{j}$ conditionally independent processes corresponding to dishes/points newly picked by each customer may be represented as 
\begin{equation}
(\tilde{Z}_{j,i},i\in [M_{j}])\overset{d}=(\sum_{\ell=1}^{\xi^{(i)}_{j}}\delta_{\omega^{(i)}_{j,\ell}}, i\in [M_{j}])
\label{pickedprocess}
\end{equation}
with $Z^{(1)}_{j}\overset{d}=\tilde{Z}_{j,1},$  
$(\xi^{(i)}_{j}, i\in [M_{j}])|B_{0}\overset{ind}\sim \mathrm{Poisson}(\gamma_{j,i}B_{0}(\Omega))$ variables, for $i\in [M_{j}],$ and where $(\omega^{(i)}_{j,\ell},\ell\in[\xi^{(i)}_{j}])|B_{0}$ are conditionally iid draws from $\bar{B}_{0}(d\omega)=B_{0}(d\omega)/B_{0}(\Omega),$ independent across $(i\in[M_{j}]),$ (hence if $B_{0}$ is non-atomic the processes in (\ref{pickedprocess}) share no common dishes/features). $\xi_{j}\overset{d}=\sum_{i=1}^{M_{j}}\xi^{(i)}_{j}\sim \mathrm{Poisson}(\psi_{j}(M_{j})B_{0}(\Omega))$ represents the total number of new dishes picked by the $M_{j}$ customers 
and 
\begin{equation}
\sum_{i=1}^{M_{j}}\tilde{Z}_{j,i}\overset{d}=(\sum_{\ell=1}^{\xi_{j}}\delta_{\tilde{\omega}_{j,\ell}})
\label{pickedprocessM}
\end{equation}
where $(\tilde{\omega}_{j,\ell})|B_{0}$ are iid $\bar{B}_{0},$ and for $\xi_{j}=d_{j},$ $\sum_{i=1}^{M_{j}}\tilde{Z}_{j,i}|B_{0}$ has distribution,
\begin{equation}
{[\psi_{j}(M_{j})]}^{d_{j}}{\mbox e}^{-\psi_{j}(M_{j})B_{0}(\Omega)}\prod_{\ell=1}^{d_{j}}B_{0}(d\tilde{\omega}_{j,\ell}).
\label{margjbigacross}
\end{equation}

Naturally when $B_{0}=\sum_{k=1}^{\infty}\tau_{k}\delta_{Y_{k}},$ the points picked in~(\ref{pickedprocess}) are no longer distinct. However, if we write $((s^{(i)}_{j,\ell},\omega^{(i)}_{j,\ell}),i\in [M_{j}])$ to be the concomittant pairs of jumps and atoms selected from $(s_{j,k},\omega_{j,k}),$  where $\mu_{j}=\sum_{k=1}^{\infty}s_{j,k}\delta_{\omega_{j,k}},$  by each customer, these pairs are distinct, since 
$(s^{(i)}_{j,\ell})$ are distinct.
Throughout the rest of this manuscript, let $((\mathscr{P}_{j,k}(\lambda),j\in[J]))$ and 
$((\mathscr{P}^{(i)}_{j,k}(\lambda),i\in [M_{j]}))$ denote independent collections of Poisson variables with mean $\lambda.$ 

\begin{prop}\label{PropBernoulliPoissonsharing}
Consider the HIBP in the Bernoulli case where $A_{j}|p\sim \mathrm{Bernoulli}(p),$ and $\mu_{j}|B_{0},$ are independent $\mathrm{CRM}(\rho_{j},B_{0})$ for $j\in [J],$ and $B_{0}=\sum_{k=1}^{\infty}\tau_{k}\delta_{Y_{k}}\sim \mathrm{CRM}(\tau_{0},G_{0}).$ Then the conditional distribution of the processes, defined in~(\ref{pickedprocess}), given $B_{0}$ are conditionally independent across $i\in [M_{j}]$ and $j\in [J],$ with 
\begin{equation}
((\tilde{Z}_{j,i},i\in[M_{j}]),j\in[J])\overset{d}=((\sum_{k=1}^{\infty}\mathscr{P}^{(i)}_{j,k}(\gamma_{j,i}\tau_{k})\delta_{Y_{k}},i\in [M_{j}]),j\in [J])
\label{BerPoisson}
\end{equation}
corresponding to within and across group individual parameter specified extensions of the Poisson IBP~\cite{Titsias} or multi-graph variations of~\cite{Caron2012}. Additionally
\begin{enumerate}
\item[(i)]$((\sum_{i=1}^{M_{j}}\tilde{Z}_{j,i},j\in[J])\overset{d}=((\sum_{k=1}^{\infty}\mathscr{P}_{j,k}(\psi_{j}(M_{j})\tau_{k})\delta_{Y_{k}},j\in [J])$
\item[(ii)] For each $(j,k),$
$(\mathscr{P}^{(i)}_{j,k}(\gamma_{j,i}\tau_{k}),i\in [M_{j}])|B_{0}$ has the joint probability mass function 
\begin{equation}
\frac{n_{j,k}!}{\prod_{i=1}^{M_{j}}n_{j,i,k}!}
\prod_{i=1}^{M_{j}}\pi^{n_{j,j,k}}_{j,i}
\times\frac{{(\psi_{j}(M_{j})\tau_{k})}^{n_{j,k}}{\mbox e}^{-\psi_{j}(M_{j})\tau_{k}}}{n_{j,k}!}
\label{BerPoisson1pmf}
\end{equation}
for $\sum_{i=1}^{M_{j}}n_{j,i,k}=n_{j,k},$ and $n_{j,k}\ge 0,$ and $\pi_{j,i}=\gamma_{j,i}/\psi_{j}(M_{j}).$
\item[(iii)]For each $k,$ $(\mathscr{P}_{j,k}(\psi_{j}(M_{j})\tau_{k}),j\in [J])|B_{0}$ has the joint probability mass function 
\begin{equation}
\frac{n_{k}!}{\prod_{j=1}^{J}n_{j,k}!}
\prod_{j=1}^{J}\pi^{n_{j,k}}_{j}
\times\frac{{([\sum_{j=1}^{J}\psi_{j}(M_{j})]\tau_{k})}^{n_{k}}{\mbox e}^{-[\sum_{j=1}^{J}\psi_{j}(M_{j})]\tau_{k}}}{n_{k}!}
\label{BerPoisson2pmf}
\end{equation}
for $\sum_{j=1}^{J}n_{j,k}=n_{k},$ and $n_{k}\ge 0,$ and $\pi_{j}=\psi_{j}(M_{j})/\sum_{l=1}^{J}\psi_{l}(M_{l}).$
\end{enumerate}
\end{prop}
\begin{proof}
The result is indicated by utilizing the distributional representations in~(\ref{pickedprocess}) and then using the discrete representation of $B_{0}=\sum_{k=1}^{\infty}\tau_{k}\delta_{Y_{k}},$ which indicates
that for measurable functions $f_{j,i}$ over $\Omega$, the Laplace exponent $-\log(\mathbb{E}[{\mbox e}^{-\tilde{Z}_{j,i}(f_{j,i})}|B_{0}])=\gamma_{j,i}B_{0}(\Omega)\int_{\Omega}(1-{\mbox e}^{-f_{j,i}(\omega)})\bar{B}_{0}(d\omega),$ which is equvalent to
$
\sum_{k=1}^{\infty}\gamma_{j,i}\tau_{k}(1-e^{-f_{j,i}(Y_{k})}).
$
\end{proof}
\begin{rem}
 In the notation of~\cite[Propositions 3.2,3.3]{James2017} the conditional distribution of $(\tilde{Z}_{j,i},i\in [M_{j}])|B_{0},$ are $\mathrm{IBP}(A_{j},\rho_{j,i-1}B_{0})$ for $i\in[M_{j}]$ and $A_{j}|p\sim \mathrm{Bernoulli}(p).$
\end{rem}

\begin{rem}\label{genAbernoulli}Setting $b_{j}=\mathbb{I}_{\{A_{j}\neq 0\}}|s\sim \mathrm{Bernoulli}(\pi_{A_{j}}(s)),$ describes the common Bernoulli process mechanism which dictates the absence or presence of a feature in a general spike and slab process. The result in Proposition \ref{PropBernoulliPoissonsharing}  hold for this parametrization of $b_{j},$ with the representations in 
(\ref{gamij}) and (\ref{genpsi}).
\end{rem}
\begin{rem}
As communicated to us by Jim Pitman, the distributional representations~(\ref{BerPoisson1pmf})
 and (\ref{BerPoisson2pmf}), correspond to a classic relationship between Multinomial and independent Poisson  variables, attributed to Fisher~\cite{FisherPoisson} and independently Soper~\cite{Soper}. 
These types of representations figure prominently in 
\cite[Corollary 3]{ZhouCarin2015}, \cite[Lemma 4.1]{Zhou1}, \cite{ZhouPadilla2016}.
\end{rem}

\subsection{Hierarchical Poisson IBP and mixed Poisson processes}
We now explore Proposition~\ref{HIBPsumrep} in the quite interesting case where $A_{j}|\lambda_{j}$ are $\mathrm{Poisson}(r_{j}\lambda_{j})$ variables for $j\in[J],$ which corresponds to hierarchical versions of the Poisson IBP models of \cite{Titsias}. The next result shows that this particular hierarchical framework leads to processes of multi-group IBP processes based on infinitely divisible mixed Poisson distributions. Furthermore, it provides a clear important connection to the work of \cite{ZhouCarin2015,Zhou1, ZhouPadilla2016} on random count matrices. 
\begin{prop}\label{mixedPoisson}Suppose that $A_{j}|\lambda_{j}$ is $\mathrm{Poisson}(r_{j}\lambda_{j})$ for $j\in [J],$ yielding hierarchical versions of the Poisson IBP of~\cite{Titsias}. Then given $\mu_{1},\ldots,\mu_{J},B_{0},$ 
$
Z^{(i)}_{j}\overset{d}=\sum_{k=1}^{\infty}\mathscr{P}^{(i)}_{j,k}(r_{j}s_{j,k})\delta_{\omega_{j,k}}
$
and $\sum_{i=1}^{M_{j}}Z^{(i)}_{j}\overset{d}=\sum_{k=1}^{\infty}\mathscr{P}_{j,k}(M_{j}r_{j}s_{j,k})\delta_{\omega_{j,k}}$ for where $\mu_{j}=\sum_{k=1}^{\infty}s_{j,k}\delta_{\omega_{j,k}}|B_{0}$ is $\mathrm{CRM}(\rho_{j},B_{0}).$ Then $((Z^{(i)}_{j},i\in [M_{j}]),j\in[J])|B_{0}$ is conditionally independent across $j\in [J],$ representable as 
\begin{equation}
((Z^{(i)}_{j},i\in [M_{j}]),j\in[J])\overset{d}=
((\sum_{k=1}^{\infty}\mathscr{P}^{(i)}_{j,k}(r_{j}\sigma_{j,k}(\tau_{k}))\delta_{Y_{k}},i\in [M_{j}]), j\in [J]),
\label{FullPoissHIBP}
\end{equation}
where $\mathscr{P}_{j,k}(M_{j}r_{j}\sigma_{j,k}(\tau_{k}))=\sum_{i=1}^{M_{j}}\mathscr{P}^{(i)}_{j,k}(r_{j}\sigma_{j,k}(\tau_{k})).$ This is deduced from the following distributional representations corresponding to the setting in Proposition~\ref{HIBPsumrep}. 
\begin{enumerate}
\item[(i)]$((\sum_{l=1}^{\infty}A^{(i)}_{j,k,l},i\in [M_{j}]))\overset{d}=((\mathscr{P}^{(i)}_{j,k}(r_{j}\sigma_{j,k}(\tau_{k})),i\in [M_{j}]))$
\item[(ii)]The distribution of 
$(\mathscr{P}^{(i)}_{j,k}(r_{j}\sigma_{j,k}(\tau_{k})),i\in [M_{j}])|\mathscr{P}_{j,k}(M_{j}r_{j}\sigma_{j,k}(\tau_{k}))=n_{j,k}$ is  $\mathrm{Multinomial}(n_{j,k},(\pi_{j,i},i\in[M_{j}])),$ for $\pi_{j,i}=1/{M_{j}}$ for $i\in[M_{j}].$  
\item[(iii)]
$(\mathscr{P}^{(i)}_{j,k}(r_{j}\sigma_{j,k}(\tau_{k})),i\in [M_{j}])|B_{0}$ has the joint probability mass function 
\begin{equation}
\frac{n_{j,k}!}{\prod_{i=1}^{M_{j}}n_{j,i,k}!}
\prod_{i=1}^{M_{j}}\pi^{n_{j,i,k}}_{j,i}
\times\frac{\mathbb{E}\left[{(M_{j}r_{j}\sigma_{j,k}(\tau_{k}))}^{n_{j,k}}{\mbox e}^{-M_{j}r_{j}\sigma_{j,k}(\tau_{k})}|\tau_{k}\right]}{n_{j,k}!}
\label{mixedPoisson1pmf}
\end{equation}
for $\sum_{i=1}^{M_{j}}n_{j,i,k}=n_{j,k},$ and $n_{j,k}\ge 0,$ and where in this case $\prod_{i=1}^{M_{j}}\pi^{n_{j,j,k}}_{j,i}={M_{j}}^{-n_{j,k}}.$  
\item[(iv)]The distribution of $((\mathscr{P}_{j,k}(M_{j}r_{j}\sigma_{j,k}(\tau_{k})),j\in [J]))$ given $B_{0}$ may be expressed
\begin{equation}
\prod_{j=1}^{J}\frac{\mathbb{E}\left[{(M_{j}r_{j}\sigma_{j,k}(\tau_{k}))}^{n_{j,k}}{\mbox e}^{-M_{j}r_{j}\sigma_{j,k}(\tau_{k})}|\tau_{k}\right]}{n_{j,k}!}.
\label{mixedPoisson2pmf}
\end{equation}
\end{enumerate}
\end{prop}
\begin{proof}It suffices to verify statement [(i)] as the others are straightforward. In this case, $A^{(i)}_{j,k,l}|s_{j,k,l}$ is equivalent in distribution to a Poisson variable $\mathscr{P}^{(i)}_{j,k,l}(r_{j}s_{j,k,l}),$ which are otherwise conditionally independent across $(j,k,l).$ The result follows from
$\sum_{l=1}^{\infty}\mathscr{P}^{(i)}_{j,k,l}(r_{j}s_{j,k,l})\overset{d}=\mathscr{P}^{(i)}_{j,k}(r_{j}\sum_{l=1}^{\infty}s_{j,k,l})$
\end{proof}
\begin{rem}Unlike the Bernoulli case in Proposition~\ref{PropBernoulliPoissonsharing}, Proposition~\ref{mixedPoisson} provides a complete description of the conditional distribution given $B_{0}$ of the entire Poisson based HIBP. However we shall not rely on these representations for practical implementation.
\end{rem}
Next we discuss some strong relations to the work of \cite{ZhouCarin2015,Zhou1, ZhouPadilla2016}
\subsubsection{Connection to priors for random count matrices}   
Quite interestingly, with reference to~(\ref{GGLevy}), if $\mu_{1},\ldots,\mu_{J}|B_{0}\overset{ind}\sim \mathrm{GG}(0,\zeta_{j},\theta_{j}),$ corresponding to independent Gamma processes, it follows that $\sigma_{j,k}(\tau_{k})|\tau_{k}$ are  independent $\mathrm{Gamma}(\theta_{j}\tau_{k},\zeta_{j})$ variables, for $j\in [J].$ Then the distribution in~(\ref{mixedPoisson2pmf}),corresponds to components having  Negative-Binomial distributions with parameters $(\theta_{j}\tau_{k},q_{j}),$ for $q_{j}=M_{j}r_{j}/(\zeta_{j}+M_{j}r_{j}),$ for $j\in [J],$ write such variables as  $(\mathscr{N}_{j,k}(\theta_{j}\tau_{k},q_{j}))$ 
In addition, for any choice of $B_{0}$, each process $(Z^{(i)}_{j}, i\in [M_{j})|B_{0}$  coincides with conditionally independent processes that generate  gamma-Poisson or Negative Binomial process random count matrices as described in~\cite[Section 1.2.3]{ZhouPadilla2016}. 
When $B_{0}\sim \mathrm{GG}(0,\zeta,\theta_{0})$ the processes $(\sum_{i=1}^{M_{j}}Z^{(i)}_{j},j\in[J])|B_{0}$ have the representation,
$$
(\sum_{k=1}^{\infty}\mathscr{P}_{j,k}(M_{j}r_{j}\sigma_{j,k}(\tau_{k}))\delta_{Y_{k}},j\in [J])
\overset{d}=
(\sum_{k=1}^{\infty}\mathscr{N}_{j,k}(\theta_{j}\tau_{k},q_{j})\delta_{Y_{k}},j\in [J]),
$$
modelling across group sharing for $j\in [J],$ then coincides with the generation of Gamma Negative Binomial  random count matrix processes in ~\cite[Section 2.2]{ZhouPadilla2016}, see also~\cite[Section 5.1]{ZhouCarin2015}. 
So in this case the hierarchical Poisson IBP generates simultaneously both these types of count matrices as it applies to within and across group sharing respectively. Proposition~\ref{mixedPoisson} shows this extends to processes based on quite general infinitely divisible mixed Poisson distributions. See the forthcoming Section~\ref{PoissonCase} for details on how to practically implement these processes, which are based on the developments we begin to discuss next.

\subsection{General joint representations given $B_{0}$}
Proposition~\ref{mixedPoisson} gives a complete description of the Poisson based HIBP conditional distribution given $B_{0}.$ Here we provide a description for all processes in~(\ref{mixedHIBP}). We note that these representations also provide important alternative descriptions, to Proposition~\ref{mixedPoisson}, in the Poisson HIBP case. We first obtain a preliminary description,

\begin{lem}\label{LemmaprelimjointB}
Consider the general specifications in~(\ref{mixedHIBP}). Then the conditional distribution of  $((Z^{(i)}_{j},i\in [M_{j}]),j\in[J])|B_{0}$ is such that it is conditionally independent across $j\in [J]$ satisfying for each $j,$
\begin{equation}
(Z^{(i)}_{j},i\in [M_{j}])\overset{d}=(\sum_{\ell=1}^{\xi_{j}}
\hat{A}^{(i)}_{j,\ell}
\delta_{{\tilde{\omega}}_{j,\ell}},i\in [M_{j}])
\label{Jointmarginalpartitionrep}
\end{equation}
where $\xi_{j}|B_{0}\sim \mathrm{Poisson}(\psi_{j}(M_{j})B_{0}(\Omega)),$ $(\tilde{\omega}_{j,\ell})|B_{0}$ are  iid $\bar{B}_{0},$ and  $((\hat{A}^{(i)}_{j,\ell},i\in[M_{j}]),\ell\in[\xi_{j}])$ are $\xi_{j}$ iid vectors with common joint distribution, with arguments, for each $\ell\in[\xi_{j}],$ $\mathbf{a}_{j,\ell}=(a^{(1)}_{j,\ell},  \ldots,a^{(M_{j})}_{j,\ell}),$
\begin{equation}
\mathsf{S}_{j}(\mathbf{a}_{j,\ell}|\rho_{j})=\frac{\int_{0}^{\infty}\left[\prod_{i=1}^{M_{j}}G_{A_{j}}(da^{(i)}_{j,\ell}|s)\right]\rho_{j}(s)ds}{\psi_{j}(M_{j})}\mathbb{I}_{\{\mathbf{a}_{j,\ell}\notin\mathbf{0}\}}
\label{generaljointtrunc}
\end{equation} 
$(\hat{A}^{(i)}_{j,\ell},i\in[M_{j}])|\tilde{H}_{j,\ell}=s$ has joint distribution $\mathbb{I}_{\{\mathbf{a}_{j}\notin\mathbf{0}\}}\prod_{i=1}^{M_{j}}G_{A_{j}}(da^{
(i)}_{j}|s)/(1-{[1-\pi_{A_{j}}(s)]}^{M_{j}}),$ 
meaning at least one component is non-zero, for  
$\tilde{H}_{j,\ell}$ with density $(1-{[1-\pi_{A_{j}}(s)]}^{M_{j}})\rho_{j}(s)/\psi_{j}(M_{j}).$
\end{lem}
\begin{proof}
The result follows from a direct application of \cite[Proposition 5.2]{James2017} where we treat 
$(Z^{(i)}_{j},i\in [M_{j}])$ as a single vector valued IBP with $A_{j,0}|s:=(A^{(1)}_{j},\ldots,A^{(M_{j})}_{j})|s$ having distribution $\prod_{i=1}^{M_{j}}G_{A_{j}}(da^{(i)}_{j}|s)$ 
\end{proof}

The next result gives a description of the marginal distributions of $(Z^{(i)}_{j},i\in [M_{j}])|B_{0},$ and the jumps of each $\mu_{j}$ that are paired with the $(\tilde{\omega}_{j,\ell}).$ It is merely a rephrasing of \cite[Proposition 3.1]{James2017}.
\begin{lem}\label{Lemmaconditionaljoint}
Consider the descriptions in Lemma~\ref{LemmaprelimjointB}. Then for each j, given $\xi_{j}=d_{j}$  and $(\tilde{\omega}_{j,\ell}),$ the distribution of  $(Z^{(i)}_{j},i\in [M_{j}])|B_{0},$ can be written as,
\begin{equation}
\left[\prod_{\ell=1}^{d_{j}}\mathsf{S}_{j}(\mathbf{a}_{j,\ell}|\rho_{j})\right]{[\psi_{j}(M_{j})]}^{d_{j}}{\mbox e}^{-\psi_{j}(M_{j})B_{0}(\Omega)}\prod_{\ell=1}^{d_{j}}B_{0}(d\tilde{\omega}_{j,\ell})
\label{margjbiggen}
\end{equation}
Consider the pairs $((S_{j,\ell},\tilde{\omega}_{j,\ell}),\ell \in [d_{j}]),$ where the $(S_{j,\ell})$ are the unobserved jumps of $\mu_{j}$ paired with the selected atoms $(\tilde{\omega}_{j,\ell}).$ Then given $(Z^{(i)}_{j},i\in [M_{j}]),B_{0},$ the  $(S_{j,\ell})$ are conditionally independent where $S_{j,\ell}$ has density $\mathbb{P}(S_{j,\ell}\in ds)/ds$ equivalent to, 
\begin{equation}
\vartheta(s|\mathbf{a}_{j,\ell},\rho_{j})\propto\left[\prod_{i=1}^{M_{j}}G_{A_{j}}(da^{(i)}_{j,\ell}|s)\right]\rho_{j}(s).
\label{jumpdensity1}
\end{equation}
\end{lem}
The next result provides an important representation in terms of compound Poisson random vectors attached to each atom $Y_{k}$ and also highlights the general role of the vectors $(N_{j,k},j\in[J])\overset{d}=(\mathscr{P}_{j,k}(\psi_{j}(M_{j})\tau_{k}),j\in[J]).$

\begin{prop}\label{bigmargBgen}Let $(\psi_{j}(M_{j}),j\in[J])$ be specified as in~(\ref{genpsi}). 
Then the joint distribution of $((Z^{(i)}_{j},i\in [M_{j}]),j\in [J])|B_{0},$ specified in~(\ref{mixedHIBP}), is equivalent to that of the random processes
\begin{equation}
((\sum_{k=1}^{\infty}[\sum_{l=1}^{N_{j,k}}\hat{A}^{(i)}_{j,k,l}]\delta_{Y_{k}},i\in [M_{j}]),j\in [J])
\label{gencompoundPoissonrep}
\end{equation}
where, given $B_{0},$
$
(N_{j,k},j\in[J])\overset{d}=(\mathscr{P}_{j,k}(\psi_{j}(M_{j})\tau_{k}),j\in[J])
$
has the distribution in Proposition~\ref{PropBernoulliPoissonsharing}~eq. (\ref{BerPoisson1pmf}) with parameters specified by~(\ref{genpsi}). Given $N_{j,k}=n_{j,k}>0,$ and each fixed $(j,k),$ the collection of vectors
$((\hat{A}^{(i)}_{j,k,l}, i\in [M_{j}]), l\in [n_{j,k}])$ are $n_{j,k}$ iid vectors with common distribution ~(\ref{generaljointtrunc}).
\end{prop}
\begin{proof}The result follows from Lemma~\ref{LemmaprelimjointB}, and arguments similar to the proof of Proposition~\ref{PropBernoulliPoissonsharing}.
\end{proof}
We now present a rephrasing of the results in 
Lemma~\ref{LemmaprelimjointB} and Proposition~\ref{bigmargBgen}, which shows that $((Z^{(i)}_{j},i\in [M_{j}]),j\in[J])|B_{0}$ are conditionally independent multivariate IBP with vector valued masses corresponding to compound Poisson variables, hence infinitely divisible variables, with positive mass at zero, regardless of the choice of $A_{j},$ attached to each possible dish/feature $(Y_{k})$
\begin{cor}\label{corjointB} Lemma~\ref{LemmaprelimjointB}, (\ref{Jointmarginalpartitionrep}), shows that 
$((Z^{(i)}_{j},i\in [M_{j}]),j\in[J])|B_{0}$
are, using notation in~\cite[Proposition 5.2]{James2017}, conditionally independent vector valued $\mathrm{IBP}((A^{(1)}_{j},\ldots,A_{j}^{(M_{j})}),\rho_{j}B_{0}),$ across $j\in[J].$
Proposition~\ref{bigmargBgen} shows that given $B_{0}=\sum_{k=1}^{\infty}\tau_{k}\delta_{Y_{k}},$ they are equivalently multivariate IBP in the sense of~\cite[Section 5]{James2017} consisting of  $J$ conditionally independent processes, corresponding to $J$ customers, with distributions
\begin{equation}
(Z^{(i)}_{j},i\in [M_{j}])|B_{0}\overset{ind}\sim
\mathrm{IBP}((\sum_{l=1}^{N_{j}}\hat{A}^{(1)}_{j,l},\ldots,\sum_{l=1}^{N_{j}}\hat{A}^{(M_{j})}_{j,l})|B_{0})
\end{equation}
where $(N_{j},j\in [J])|\lambda \overset{d}=(\mathscr{P}_{j}(\psi_{j}(M_{j})\lambda),j\in[J]).$ 
$(\hat{A}^{(1)}_{j,l},\ldots,\hat{A}^{(M_{j}}_{j,l})$
 are independent of this, and for each fixed $j,$ are iid vectors across $l,$ with common distribution $\mathsf{S}_{j}(\mathbf{a}_{j,l}|\rho_{j})$ in~(\ref{generaljointtrunc}), not depending on $\lambda,$ and hence $B_{0}.$

\end{cor}

\section{Distributions of $((Z^{(i)}_{j},i\in [M_{j}]),j\in [J])$}
We now focus on representations of the unconditioned process which will aid in practically sampling these processes. First, we obtain some properties of the generalized Poisson IBP process, based on $(N_{j,k},j\in[J])\overset{d}=(\mathscr{P}_{j,k}(\psi_{j}(M_{j})\tau_{k}),j\in[J]),$ that appears in [Proposition~\ref{PropBernoulliPoissonsharing}, (i)],Proposition~\ref{bigmargBgen}, and Corollary~\ref{corjointB}. That is to say, in every mixed spike and slab HIBP regardless of choice of $A_{j}.$

\begin{prop}\label{multisample}
 Given $B_{0}=\sum_{k=1}^{\infty}\tau_{k}\delta_{Y_{k}}\sim CRM(\tau_{0},G_{0}),$ consider the random process
\begin{equation}
(\sum_{k=1}^{\infty}N_{j,k}\delta_{Y_{k}},j\in [J])\overset{d}=(\sum_{k=1}^{\infty}\mathscr{P}_{j,k}(\psi_{j}(M_{j})\tau_{k})\delta_{Y_{k}},j\in [J]),
\label{Poissonsharing}
\end{equation}
 with parameters specified as in~(\ref{genpsi}). Let $\xi(\varphi)$ be a $\mathrm{Poisson}(\varphi)$ variable with 
$\varphi:=G_{0}(\Omega)\tilde{\psi}_{0}(\sum_{j=1}^{J}\psi_{j}(M_{j}))$ then the unconditional distribution of the processes in~(\ref{Poissonsharing}) can be expressed as
\begin{equation}
(\sum_{k=1}^{\infty}N_{j,k}\delta_{Y_{k}},j\in [J])\overset{d}=(\sum_{k=1}^{\xi(\varphi)}X_{j,k}\delta_{\tilde{Y}_{k}},j\in[J])
\end{equation}
where $(\tilde{Y}_{k})$ are iid $\bar{G}_{0},$ $X_{0,k}:=(X_{1,k},\ldots X_{J,k})$ are iid across $k,$ and independent of $\xi(\varphi),$ and there are iid pairs $((X_{0,k},H_{k}))$ satisfying the following properties, 
\begin{enumerate}
\item[(i)]$X_{0,k}|H_{k}=\lambda$ has the joint probability mass function 
\begin{equation}
\frac{n_{k}!}{\prod_{j=1}^{J}n_{j,k}!}
\prod_{j=1}^{J}\pi^{n_{j,k}}_{j}
\times\frac{{([\sum_{j=1}^{J}\psi_{j}(M_{j})]\lambda)}^{n_{k}}{\mbox e}^{-[\sum_{j=1}^{J}\psi_{j}(M_{j})]\lambda}}{n_{k}!(1-{\mbox e}^{-\lambda[\sum_{j=1}\psi_{j}(M_{j})]})}
\label{truncatedMPoisson}
\end{equation}
for $\sum_{j=1}^{J}n_{j,k}=n_{k},$ and $n_{k}>0,$ and $\pi_{j}=\psi_{j}(M_{j})/\sum_{l=1}^{J}\psi_{l}(M_{l}),$
and 
\item[(ii)]$\mathbb{P}(H_{k}\in d\lambda)/d\lambda=(1-{\mbox e}^{-\lambda[\sum_{j=1}\psi_{j}(M_{j})]})\tau_{0}(\lambda)/\tilde{\psi}_{0}(\sum_{j=1}^{J}\psi_{j}(M_{j})).$
\item[(iii)]$X_{0,k}|\sum_{j=1}^{J}X_{j,k}=n_{k}$ has a Multinomial distribution with parameters $(n_{k},\pi_{1},\ldots,\pi_{J})$
\item[(iv)]$\tilde{X}_{k}:=\sum_{j=1}^{J}X_{j,k},$ has the distribution for $n_{k}=1,2,\ldots,$
\begin{equation}
\tilde{p}(n_{k}|\kappa,\tau_{0})=\frac{\kappa^{n_{k}}\int_{0}^{\infty}\lambda^{n_{k}}
{\mbox e}^{-\lambda\kappa}\tau_{0}(\lambda)d\lambda}{n_{k}!\tilde{\psi}_{0}(\kappa)}
\label{mixedtruncatedPoissonpmf}
\end{equation}
for $\kappa=\sum_{j=1}^{J}\psi_{j}(M_{j}).$
\item[(v)]Denote the distribution of $X_{0,k}$ as $\mathrm{MtP}((\psi_{j}(M_{j}),j\in [J]),\tau_{0}),$ and of $\tilde{X}_{k}$ as $\mathrm{MtP}(\sum_{j=1}^{J}\psi_{j}(M_{j}),\tau_{0}).$
\end{enumerate}
\end{prop}
\begin{proof}
The result follows from an application of \cite[Proposition 5.2]{James2017} where we view the processes (\ref{Poissonsharing}) as a single  multivariate IBP, with 
$A_{0,k}:=(N_{1,k},\ldots,N_{J,k})|\tau_{k}$  having distribution as in (\ref{mixedPoisson2pmf}). It then follows that the spike is $1-\pi_{A_{0}}(\lambda)=\mathbb{P}((N_{j,k}=0,j\in [J])|\tau_{k}=\lambda)={\mbox e}^{-\lambda[\sum_{j=1}^{J}\psi_{j}(M_{j})]}$ and the slab corresponds to the zero truncated product of independent Poisson variables in~(\ref{truncatedMPoisson}).
\end{proof}
\begin{rem}Proposition~\ref{multisample} is a multivariate extension of the result for the Poisson IBP in~\cite{Titsias}, \cite[Section 4.2]{James2017} or in~\cite[Theorem 1]{ZhouFoF} in relation to frequency of frequency distributions(FoF). As noted in \cite[Remark 4.1]{James2017} the infinite sequences $((\tilde{X_{k}},H_{k}))$ arise much earlier in~\cite{Pit97}.
\end{rem}
As indicated in Proposition~\ref{multisample}, in order to sample  $X_{0,k}\sim\mathrm{MtP}((\psi_{j}(M_{j}),j\in [J]),\tau_{0}),$ it suffices to sample $\tilde{X}_{k}=n_{k}$ from $\mathrm{MtP}(\sum_{j=1}^{J}\psi_{j}(M_{j}),\tau_{0}),$  and then a Multinomial distribution. We next describe distributions of $\tilde{X}_{k},$ in the generalized gamma case that has appeared elsewhere in the literature. 
\subsubsection{The generalized gamma case}
\label{Gengamma}
Now, as in ~\cite[Section 4.2.1, eq. (4.5)]{James2017} and \cite[Section 3]{ZhouFoF} we give the calculations for the distribution of $\tilde{X}_{k}$  when $B_{0}$ is $\mathrm{GG}(\alpha,\zeta;\theta_{0}G_{0}),$ specified by
(\ref{GGLevy}).
In this case for $m=1,2,\ldots,$ we can set 
$\tilde{p}(m|\kappa,\tau_{0}),$ defined in~(\ref{mixedtruncatedPoissonpmf}), equal to $p_{\alpha}(m|\kappa,\zeta)$ where
\begin{equation}
\label{Xfunction}
p_{\alpha}(m|\kappa,\zeta)=\left\{\begin{array}{ll}
\frac{\kappa^{m}{(\kappa+\zeta)}^{\alpha-m}}
{[{(\kappa+\zeta)}^{\alpha}-{\zeta}^{\alpha}]}\frac{\alpha \Gamma(m-\alpha)}{m!\Gamma(1-\alpha)}, &   {\mbox{ if }} 0<\alpha<1, \zeta\ge 0\\ 
\frac{\kappa^{m}{(\kappa+\zeta)}^{-m}}
{m\log(1+\kappa/\zeta)},& {\mbox{ if }} \alpha=0, \zeta>0\\
\frac{\kappa^{m}{(
\kappa+\zeta)}^{-(\delta+m)}}
{[{\zeta}^{-\delta}-{(\kappa+\zeta)}^{-\delta}]}\frac{\Gamma(m+\delta)}{m!\Gamma(\delta)}, &   {\mbox{ if }} \alpha=-\delta <0,\zeta>0\\
                     \end{array}\right.
\end{equation}
Hence in this case $\tilde{X}_{k}\sim\mathrm{MtP}(\sum_{j=1}^{J}\psi_{j}(M_{j}),\tau_{0}),$ has pmf $p_{\alpha}(m|\sum_{j=1}^{J}\psi_{j}(M_{j}),\zeta).$ In addition $\tilde{\psi}_{0}(\sum_{j=1}^{J}\psi_{j}(M_{j}))=\theta_{0}\tilde{\psi}_{\alpha,\zeta}(\sum_{j=1}^{J}\psi_{j}(M_{j}))),$ as in (\ref{rfunction}). 
Setting $\alpha=0,$ gives the Logarithmic series distribution~\cite{Fisher} and $\sum_{k=1}^{\xi(\varphi)}\tilde{X}_{k}$ coincides with the compound Poisson representation of a Negative Binomial distribution due to~\cite{Quenouille}.  This fact plays a key role in the practical implementation of Negative Binomial random count matrices and related quantities in the works of~\cite{ZhouCarin2015,Zhou1, ZhouPadilla2016}. The choice of $\zeta=0,$ equates to $\tilde{X_{k}}$ having Sibuya's distribution and gives the compound Poisson representation of a discrete-Stable variable, see~\cite{DevroyeTryp}. See \cite[Sections 3,4,5]{ZhouFoF} for applications of the general case in~(\ref{Xfunction}) within the FoF context.
\subsection{Results for sampling $((Z^{(i)}_{j},i\in [M_{j}]),j\in [J])$}
Next we describe one of our primary results which has implications for sampling and practical implementation. 
\begin{thm}\label{thmHIBPmarginalgen}
Let $((X_{0,k},\tilde{Y}_{k}),k\in [\xi(\varphi)]),\xi(\varphi))$ be as described in Proposition~\ref{multisample}. The marginal distribution of $((Z^{(i)}_{j},i\in [M_{j}]),j\in [J]),$ specified in~(\ref{mixedHIBP}), satisfies
$$
((Z^{(i)}_{j},i\in [M_{j}]),j\in [J])\overset{d}=((\sum_{k=1}^{\xi(\varphi)}[\sum_{l=1}^{X_{j,k}}\hat{A}^{(i)}_{j,k,l}]\delta_{\tilde{Y}_{k}},i\in [M_{j}]),j\in [J]),
$$
where $(\hat{A}^{(i)}_{j,k,l})$ are otherwise as specified in Proposition~\ref{bigmargBgen} and (\ref{generaljointtrunc}). Hence
\begin{equation}
(\sum_{i=1}^{M_{j}}Z^{(i)}_{j},j\in [J])\overset{d}=(\sum_{k=1}^{\xi(\varphi)}[\sum_{l=1}^{X_{j,k}}[\sum_{i=1}^{M_{j}}\hat{A}^{(i)}_{j,k,l}]]\delta_{\tilde{Y}_{k}},j\in [J])
\label{sumprocess}
\end{equation}
\end{thm}
\begin{proof}
This result can also be deduced as a special case of \cite[Proposition 5.2]{James2017}, however there are some details to note. The key is to identify the spike and slab distributional decomposition of the components in Proposition~\ref{bigmargBgen}. The result follows by noting for each $k,$ the vector
$
((\sum_{l=1}^{N_{j,k}}\hat{A}^{(i)}_{j,k,l}, i\in [M_{j}]), j\in [J])
$
has all components equal to zero if and only if $N_{j,k}=0$ for $j\in [J].$ The result then follows from the details in the proof of Proposition~\ref{multisample}.
 \end{proof}

\subsection{Marginal distributions and multi-group ECPFs}
Using slightly different notation than in \cite[Corollary 2]{ZhouFoF}, including an extra parameter $\kappa>0,$  for $\gamma_{0}:=G_{0}(\Omega),$ a general L\'evy density $\tau_{0}$, and positive cell counts $\mathbf{n}_{r}:=(n_{1},\ldots,n_{r})$, where $\sum_{k=1}^{r}n_{k}=n,$ define the function
\begin{equation}
\varrho(\mathbf{n}_{r};n|\gamma_{0},\tau_{0},\kappa)=\frac{\gamma^{r}_{0}}{n!}{\mbox e}^{-\gamma_{0}\int_{0}^{\infty}(1-{\mbox e}^{-\lambda \kappa})\tau_{0}(\lambda)d\lambda}
\prod_{k=1}^{r}\varpi(n_{k}|\kappa,\tau_{0}),
\label{ECPF}
\end{equation}
where 
$
\varpi(n_{k}|\kappa,\tau_{0})=n_{k}!\tilde{\psi}_{0}(\kappa)\tilde{p}(n_{k}|\kappa,\tau_{0})=\int_{0}^{\infty}\kappa^{n_{k}}\lambda^{n_{k}}{\mbox e}^{-\lambda\kappa}\tau_{0}(\lambda)d\lambda,
$
which is equivalent to an exchangeable cluster probability function ($\mathrm{ECPF}$) in \cite{ZhouFoF}. Now, from Theorem~\ref{thmHIBPmarginalgen} in terms of observed values of $(\hat{A}^{(i)}_{j,k,l}),$ there is the correspondence $(\mathbf{a}_{j,\ell}, \ell\in[d_{j}])=
(\mathbf{a}_{j,k,l}, l\in[n_{j,k}],k\in[r]),$ for $\mathbf{a}_{j,k,l}=(a^{(i)}_{j,k,l},i\in [M_{j}]).$
Hence, using the convention $\prod_{l=1}^{n_{j,k}}c_{l}=1$ for $n_{j,k}=0,$  we can write,
\begin{equation}
\mathbf{S}_{[J]}(\mathbf{a})=\prod_{j=1}^{J}\left[\prod_{\ell=1}^{d_{j}}\mathsf{S}_{j}(\mathbf{a}_{j,\ell}|\rho_{j})\right]=\prod_{j=1}^{J}\prod_{k=1}^{r}\prod_{l=1}^{n_{j,k}}\mathsf{S}_{j}(\mathbf{a}_{j,k,l}|\rho_{j}),
\label{Alikelihood}
\end{equation}
which is the likelihood of $(((\hat{A}^{(1)}_{j,k,l},\ldots,\hat{A}^{(M_{j})}_{j,k,l}),l\in[n_{j,k}],k\in[r]),j\in J),$ otherwise the same as in the respective non-hierarchical cases.
\begin{prop}\label{marginalslabHIBP}The marginal distribution of $((Z^{(i)}_{j},i\in [M_{j}]),j\in [J])$ can be expressed as, 
\begin{equation}
\mathbf{S}_{[J]}(\mathbf{a})\times{n!\prod_{j=1}^{J}\frac{\pi^{d_{j}}_{j}}{
\prod_{k=1}^{r}n_{j,k}!}}
\times \varrho(\mathbf{n}_{r};n|\gamma_{0},\tau_{0},\sum_{j=1}^{J}\psi_{j}(M_{j}))\times \prod_{k=1}^{r}
\bar{G}_{0}(d\tilde{Y}_{k})
\label{multiECPF1}
\end{equation}
where $((X_{j,k}=n_{j,k},j\in [J],k\in[r]),\xi(\varphi)=r),$ $\sum_{j=1}^{J}n_{j,k}=n_{k}>0,$ $\sum_{k=1}^{r}n_{j,k}=d_{j}$ and $\sum_{j=1}^{J}d_{j}=n,$ and  $(\hat{A}^{(i)}_{j,k,l},i\in [M_{j}])=\mathbf{a}_{j,k,l}.$ 
In addition, given $((Z^{(i)}_{j},i\in [M_{j}]),j\in [J]),$
\begin{enumerate}
\item[(i)]the corresponding jumps $(S_{j,k,l},l\in[n_{j,k}],k\in[r])\overset{d}=(S_{j,\ell},\ell\in [d_{j}]),$ of $\mu_{j},$ are conditionally independent with density 
$\vartheta(s|\mathbf{a}_{j,k,l},\rho_{j})$
described in~(\ref{jumpdensity1}).
\item[(ii)]$((\tilde{L}_{k},\tilde{Y}_{k}),k\in[r])$ are the pairs of unobserved jumps and selected atoms of $B_{0},$ where the $(\tilde{L}_{k},k\in [r])$ are conditionally independent with density $\mathbb{P}(\tilde{L}_{k}\in dt)/dt$ equal to, for 
$\kappa=\sum_{j=1}^{J}\psi_{j}(M_{j}),$
\begin{equation}
\eta(t|n_{k},\kappa,\tau_{0})=\frac{t^{n_{k}}{\mbox e}^{-t\sum_{j=1}^{J}\psi_{j}(M_{j})}\tau_{0}(t)}{\int_{0}^{\infty}v^{n_{k}}{\mbox e}^{-v\sum_{j=1}^{J}\psi_{j}(M_{j})}\tau_{0}(v)dv}.
\label{JumpsDist1}
\end{equation}
\end{enumerate}
\end{prop}
\begin{proof}The form of the joint distribution of $((Z^{(i)}_{j},i\in [M_{j}]),j\in [J])|B_{0}$ is given by taking the product over $j\in[J]$ of (\ref{margjbiggen}) in Lemma~\ref{Lemmaconditionaljoint}.
It follows from (\ref{margjbigacross})
and Proposition~\ref{PropBernoulliPoissonsharing} that, mod $\mathbf{S}_{[J]}(\mathbf{a}),$ the resulting expression corresponds to the distribution of $(\sum_{k=1}^{\infty}N_{j,k}\delta_{Y_{k}},j\in [J])|B_{0}$ in Proposition~\ref{multisample}. Using the descriptions in that proof, the expression in~(\ref{multiECPF1}) may then be obtained by an application of~\cite[eqs. (5.1) and (5.2)]{James2017} with $M=1$ and $m_{\ell}=1.$ Alternatively one may use the Poisson calculus methods in \cite{James2002,James2005}. Statement [(i)] follows from Lemma~\ref{Lemmaconditionaljoint} and the correspondence in Theorem~\ref{thmHIBPmarginalgen}. Statement [(ii)] can be read from~\cite[Proposition 5.1]{James2017}.
\end{proof}
Proposition~\ref{marginalslabHIBP} describes the joint marginal distribution in terms of the components distributions that can be sampled, the following presents an equivalent form of (\ref{multiECPF1})
\begin{cor}\label{altmargslab}Consider the specifications in Proposition~
\ref{marginalslabHIBP}, and set $\kappa=\sum_{j=1}^{J}\psi_{j}(M_{j}),$ then the marginal distribution of $((Z^{(i)}_{j},i\in [M_{j}]),j\in [J])$ expressed as (\ref{multiECPF1}), is equivalent to, 
\begin{equation}
\frac{\Delta(\mathbf{a})n!\varrho(\mathbf{n}_{r};n|\gamma_{0},\tau_{0},\sum_{j=1}^{J}\psi_{j}(M_{j}))}{{[\sum_{j=1}^{J}\psi_{j}(M_{j})]}^{n}\prod_{j=1}^{J}\prod_{k=1}^{r}n_{j,k}!}
\prod_{k=1}^{r}
\bar{G}_{0}(d\tilde{Y}_{k})
\label{multiECPF2}
\end{equation}
where $\Delta(\mathbf{a})=\prod_{j=1}^{J}\prod_{\ell=1}^{d_{j}}\int_{0}^{\infty}\left[\prod_{i=1}^{M_{j}}G_{A_{j}}(da^{(i)}_{j,\ell}|s)\right]\rho_{j}(s)ds.$ 
\end{cor}
\subsubsection{Calculations for $B_{0}\sim \mathrm{GG}(\alpha,\zeta;\theta_{0}G_{0})$}
Suppose that $B_{0}\sim \mathrm{GG}(\alpha,\zeta;\theta_{0}G_{0}),$ then for Proposition~\ref{marginalslabHIBP} and Corollary~\ref{altmargslab},
\begin{equation}
n!\varrho(\mathbf{n}_{r};n|\gamma_{0},\tau_{0},\kappa)={\mbox e}^{-\theta_{0}\gamma_{0}\psi_{\alpha,\zeta}(\kappa)}\theta^{r}_{0}\gamma^{r}_{0}{(\frac{\kappa}{\kappa+\zeta})}^{n}\prod_{k=1}^{r}\frac{\Gamma(n_{k}-\alpha)}{\Gamma(1-\alpha)}
\label{GGBlikelihood}
\end{equation}
Set $p=\kappa/(\kappa+\zeta)$ in \cite[Section 2.2 and Section 3, eq. (13)]{ZhouFoF} to recover equivalent expressions.
\subsubsection{Multi-group ECPF}
As a Corollary, the next result, which may be compared with~\cite[Corollary 2]{ZhouFoF}, describes a multi-group version of an $\mathrm{ECPF}.$ 

\begin{cor}\label{propMulti1}Set $\sum_{j=1}^{J}d_{j}=n,$ and let $K(n)=r\in [n]$ denote the number of distinct points $(\tilde{Y}_{1},\ldots,\tilde{Y}_{r})$ drawn from $n$ samples from $\bar{B}_{0}.$ 
\begin{enumerate}
\item[(i)]Then  the joint distribution of random variables $((\tilde{\omega}_{j,\ell},\ell\in[\xi_{j}], \xi_{j}=d_{j}),j\in[J]),$ described in Lemma~\ref{LemmaprelimjointB}~(\ref{Jointmarginalpartitionrep}), can be expressed as
\begin{equation}
\frac{n!\prod_{j=1}^{J}\pi^{d_{j}}_{j}}{\prod_{j=1}^{J}d_{j}!}\times \varrho(\mathbf{n}_{r};n|\gamma_{0},\tau_{0},\sum_{j=1}^{J}\psi_{j}(M_{j}))\times \prod_{k=1}^{r}\bar{G}_{0}(d\tilde{Y}_{k})
\label{multiECPFfirst}
\end{equation}
where $\gamma_{0}=G_{0}(\Omega).$
\item[(ii)]The distribution of the variables in $[(i)],$
is equivalent to 
$$
\frac{r!\prod_{k=1}^{r}\prod_{j=1}^{J}n_{j,k}!}{\prod_{j=1}^{J}d_{j}!}
$$
times the distribution of $((X_{j,k}=n_{j,k},j\in [J]),\tilde{Y}_{k}),k\in [\xi(\varphi)), \xi(\varphi)=r).$ 
\end{enumerate}
\end{cor}
\begin{rem}See \cite[Proposition 4 and eq. (5)]{JLP2}, and also \cite[Section 2.3]{DiBenedettopartition}, for other interpretations of the ECPF in~(\ref{ECPF}). Pitman~\cite{PitmanPoissonMix} provides additional interpretations, and points out these correspondences to earlier work of Fisher and McCloskey on species sampling models. See also~\cite{HJL,JamesStick}.
\end{rem}
\subsection{Posterior distributions for $B_{0}$}
We now describe the posterior distribution of $B_{0}.$ We note that the result only depends on the choice of $A_{j}$
through $\psi_{j}(M_{j}).$
\begin{prop}\label{postBgivenZq}
Define the L\'evy density $\tau_{0,J}(\lambda)={\mbox e}^{-\lambda\sum_{j=1}^{J}\psi_{j}(M_{j})}\tau_{0}(\lambda).$ Then the posterior distribution of $B_{0}| ((Z^{(i)}_{j},i\in [M_{j}]),j\in [J])$ is equivalent
to the distribution of the random measure
\begin{equation}
\tilde{B}_{0,J}+\sum_{k=1}^{r}\tilde{L}_{k}\delta_{\tilde{Y}_{k}}
\label{postBrep}
\end{equation}
where $\tilde{B}_{0,J}=\sum_{k=1}^{\infty}\tau'_{k}\delta_{Y'_{k}}\sim \mathrm{CRM}(\tau_{0,J},G_{0}),$ and independent of this, $(\tilde{L}_{k},k\in [r])$ are as in~(\ref{JumpsDist1}).
\end{prop}
\begin{proof}
The distribution of $B_{0}|((Z^{(i)}_{j},i\in [M_{j}]),j\in [J])$ is the same as $B_{0}|(\sum_{k=1}^{\infty}N_{j,k}\delta_{Y_{k}},j\in [J]).$ Hence the result follows from  \cite[Proposition 5.1]{James2017}, with $M=1.$
\end{proof}
It follows that the posterior distribution of $B_{0}$ in Proposition~\ref{postBgivenZq}, is also equivalent to that of $B_{0}|\sum_{k=1}^{\infty}[\sum_{j=1}^{J}N_{j,k}]\delta_{Y_{k}},$ where, using the specifications in Proposition~\ref{multisample},
\begin{equation}
\sum_{k=1}^{\infty}[\sum_{j=1}^{J}N_{j,k}]\delta_{Y_{k}}\overset{d}=\sum_{k=1}^{\xi(\varphi)}\tilde{X}_{k}\delta_{\tilde{Y}_{k}}
\label{PoissonIBP}
\end{equation}
is a univariate Poisson $\mathrm{IBP}$ with $\sum_{j=1}^{J}N_{j,k}|\tau_{k}\sim \mathrm{Poisson}(\tau_{k}\sum_{j=1}^{J}\psi_{j}(M_{j})).$ Hence more specific details may be read from \cite[See section 4.2]{James2017}. The next explicit computations can be read from~\cite[Section 4.2.1]{James2017}.
\subsubsection{Descriptions for $B_{0}$ in generalized gamma case}
If $B_{0}\sim \mathrm{GG}(\alpha,\zeta;\theta_{0}G_{0}),$ then 
$\tilde{B}_{0,J}\sim\mathrm{GG}(\alpha,\zeta+\sum_{j=1}^{J}\psi_{j}(M_{j});\theta_{0}G_{0})$ and $\tilde{L}_{k}\overset{ind}\sim \mathrm{Gamma}(n_{k}-\alpha,\zeta+\sum_{j=1}^{J}\psi_{j}(M_{j})).$ In addition $\tau_{0,J}(\lambda)=\tilde{\tau}_{\alpha}(\lambda|\zeta+\sum_{j=1}^{J}\psi_{j}(M_{j}),\theta_{0}).$
\subsubsection{The case where $B_{0}$ is $\mathrm{sBP}((\alpha,\beta),\theta_{0},G_{0})$} In view of the standard specifications for $B_{0}$ as a Beta process for the Bernoulli HIBP in~Section~\ref{Bernoulliclassic}, ~(\ref{BetaB}), it is interesting to look at details for that case and the more flexible extension where $B_{0}$ is  $\mathrm{sBP}(\alpha,\theta_{0},G_{0}).$ Setting $\tau_{0}(p)=\rho_{\alpha,\beta}(p)$ in $\ref{stablebetauni},$ it follows that $\tilde{B}_{0,J},$ has L\'evy density 
$\tau_{0,J}(p)={\mbox e}^{-p\sum_{j=1}^{J}\psi_{j}(M_{j})}\rho_{\alpha,\beta}(p),$ and $\tilde{L_{k}}$ are independent with densities proportional to 
\begin{equation}
{\mbox e}^{-p\sum_{j=1}^{J}\psi_{j}(M_{j})} p^{n_{k}-\alpha-1}(1-p)^{\beta
+\alpha-1}\indic_{\{0<p<1\}},
\label{expBeta}
\end{equation} 
corresponding to exponentially tilted Beta random variables. See~\cite[Section 4.4.2]{James2005} for  inhomogeneous versions of these quantities.

\subsection{Posterior distributions of $(\mu_{j},j\in[J])$}
Proposition~\ref{postBgivenZq} leads next to descriptions of the posterior distributions of $(\mu_{j}, j\in [J])|((Z^{(i)}_{j},i\in [M_{j}]),j\in [J]).$ 

\begin{prop}\label{propmupostrep}
Consider the specifications in Proposition~
\ref{marginalslabHIBP} and Proposition~\ref{postBgivenZq}. Then the posterior distribution of $(\mu_{j}, j\in [J])|((Z^{(i)}_{j},i\in [M_{j}]),j\in [J]),$ is such that for each $j,$ the distribution of $\mu_{j}$ is equivalent to
\begin{equation}
\tilde{\mu}_{j,M_{j}}+\sum_{k=1}^{r}\left[\sum_{l=1}^{\infty}\tilde{s}_{j,k,l}\right]\delta_{\tilde{Y}_{k}}+\sum_{k=1}^{r}\left[\sum_{l=1}^{n_{j,k}}S_{j,k,l}\right]\delta_{\tilde{Y}_{k}}
\label{mureppost}
\end{equation}
where $\tilde{\mu}_{j,M_{j}}|\tilde{B}_{0,J}$ is $\mathrm{CRM}(\rho_{j,M_{j}},\tilde{B}_{0,J}),$ and  for each fixed $(j,k),$ the $\tilde{\mathbf{s}}_{j,k}=(\tilde{s}_{j,k,l})|\tilde{L}_{k}$
are points of a $\mathrm{PRM}$ with mean  $\tilde{L}_{k}\rho_{j,M_{j}}.$
Additionally, define independent subordinators $(\tilde{\sigma}_{j}(t),t\ge 0)$ with $-\log \mathbb{E}[{\mbox e}^{-\tilde{\sigma}_{j}(t)}]=t\int_{0}^{1}(1-{\mbox e}^{-s})\rho_{j,M_{j}}(s)ds,$ and let for each $j,$ $((\tilde{\sigma}_{j,k}(t)))$ and $((\varsigma_{j,k}(t)))$ denote iid copies of $\tilde{\sigma}_{j},$ then,
applying Proposition~\ref{HIBPsumrep},
\begin{equation}
\tilde{\mu}_{j,M_{j}}+\sum_{k=1}^{r}\left[\sum_{l=1}^{\infty}\tilde{s}_{j,k,l}\right]\delta_{\tilde{Y}_{k}}\overset{d}=\sum_{k=1}^{\infty}\tilde{\sigma}_{j,k}(\tau'_{k})\delta_{Y'_{k}}+\sum_{k=1}^{r}\varsigma_{j,k}(\tilde{L}_{k})\delta_{\tilde{Y}_{k}},
\label{mupartpost}
\end{equation}
where all variables in~(\ref{mupartpost}) are random except 
$((\tilde{Y}_{k},k\in[r]),r).$
\end{prop}
\begin{proof}
An application of \cite[Theorem 3.1]{James2017}
shows that the posterior distribution of 
$\mu_{j}|((Z^{(i)}_{j},i \in [M_{j}]),j\in[J]),B_{0}$ is equivalent to 
$\hat{\mu}_{j,M_{j}}+\sum_{\ell=1}^{d_{j}}S_{j,\ell}\delta_{\omega_{j,\ell}}$
where $\hat{\mu}_{j,M_{j}}|B_{0}$ is $\mathrm{CRM}(\rho_{j,M_{j}}, B_{0}).$ Using the posterior representation of $B_{0}$ in (\ref{postBrep}),it follows that the distribution of $\hat{\mu}_{j,M_{j}}|((Z^{(i)}_{j},i \in [M_{j}])$ corresponds to that of (\ref{mupartpost}).
The result is concluded by again noting that $(S_{j,\ell},\ell\in[d_{j}])\overset{d}=(S_{j,k,l},l\in[n_{j,k}],k\in [r]).$
\end{proof}

\subsection{Predictive distributions}
The descriptions in Proposition~\ref{postBgivenZq}
and Proposition~\ref{propmupostrep}, in particular Proposition~\ref{postBgivenZq}, leads to descriptions of the predictive distribution of various univariate and multivariate processes. Here, we describe predictive distributions in the univariate case. Multivariate extensions do not present extra difficulties. We will use the following univariate distribution defined for each $j$ and $M_{j}=0,1,2,\ldots$
\begin{equation}
\frac{\mathbb{I}_{\{{a}_{j}\neq 0\}}\int_{0}^{\infty}\left[G_{A_{j}}(da_{j}|s)\right]\rho_{j,M_{j}}(s)ds}{\psi_{j}(M_{j}+1)-\psi_{j}(M_{j})}
\label{generalunivariatetrunc}
\end{equation}

where $\gamma_{j,M_{j}+1}=\psi_{j}(M_{j}+1)-\psi_{j}(M_{j}).$ In the next result, we suppress notation on some variables which would otherwise indicate dependence on $M_{j}.$

\begin{prop}\label{postgenpredict}
Set $\tilde{\psi}_{0,J}(\kappa)=\int_{0}^{\infty}(1-{\mbox e}^{-\lambda \kappa})\tau_{0,J}(\lambda)d\lambda,$ and for each $j$, $\tilde{\xi}_{j}(\phi_{j})$ is a Poisson random variable with mean $\phi_{j}=G_{0}(\Omega)\tilde{\psi}_{0,J}(\gamma_{j,M_{j}+1}),$ $(\tilde{X}_{j,k})$ are iid 
$\mathrm{MtP}(\gamma_{j,M_{j}+1},\tau_{0,J})$ and $(Y_{j,k})$  are iid $\bar{G}_{0},$ and $(\tilde{Y}_{k},k\in[r])$ are fixed previously observed points. For any $j\in [J],$ the predictive distribution of $Z^{(M_{j}+1)}_{j}$ given $((Z^{(i)}_{\ell},i\in [M_{\ell}]),\ell\in [J]),$ has the representation
\begin{equation}
\hat{Z}_{j}+\sum_{k=1}^{r}\left[\sum_{l=1}^{n_{j,k}}A_{j,k,l}\right]\delta_{\tilde{Y}_{k}},
\label{predictAMj}
\end{equation}
where,
 $(A_{j,k,l}),$ are conditionally independent such that $A_{j,k,l}|S_{j,k,l}=s$ has distribution $G_{A_{j}}(da|s)$, and $(S_{j,k,l})$ are conditionally independent with density specified in Proposition~\ref{marginalslabHIBP}.
$\hat{Z}_{j}|((Z^{(i)}_{l},i \in [M_{j}]),l\in[J]),$ can be represented as
\begin{equation}
\sum_{k=1}^{\tilde{\xi}_{j}(\phi_{j})}\left[\sum_{l=1}^{\tilde{X}_{j,k}}
\tilde{A}_{j,k+r,l}\right]\delta_{Y_{j,k}}+
\sum_{k=1}^{r}\left[\sum_{l=1}^{\tilde{N}_{j,k}}
\tilde{A}_{j,k,l}\right]\delta_{\tilde{Y}_{k}}
\label{newpartpredictA}
\end{equation}
where $(\tilde{N}_{j,k},k\in[r])\overset{d}=(\mathscr{P}_{j,k}(\gamma_{j,M_{j}+1}\tilde{L}_{k}),k\in[r]),$
and 
 $((\tilde{A}_{j,t,l},t\in[r+\tilde{\xi}_{j}(\phi_{j})]))$ are collections of iid variables each with distribution (\ref{generalunivariatetrunc}). 

\end{prop}
\begin{proof}
Apply \cite[Proposition 3.2]{James2017} to obtain a description of the predictive distribution of ${Z}^{(M_{j}+1)}_{j}|((Z^{(i)}_{l},i \in [M_{j}]),l\in[J]),B_{0},$  equating to $\hat{Z}_{j}+\sum_{\ell=1}^{d_{j}}A_{j,\ell}\delta_{\omega_{j,\ell}}.$ Where it follows that $\sum_{\ell=1}^{d_{j}}A_{j,\ell}\delta_{\omega_{j,\ell}}\overset{d}=\sum_{k=1}^{r}\left[\sum_{l=1}^{n_{j,k}}A_{j,k,l}\right]\delta_{\tilde{Y}_{k}},$ and $\hat{Z}_{j}|((Z^{(i)}_{l},i \in [M_{j}]),l\in[J]),B_{0},$ has an $\mathrm{IBP}(A_{j},\rho_{j,M_{j}}B_{0})$ distribution which means that it can be represented as a compound Poisson process with $\mathrm{Poisson}(\gamma_{j,M_{j}+1}B_{0})$ sum of iid variables following the distribution in (\ref{generalunivariatetrunc}), and iid points drawn from $\bar{B}_{0},$ which otherwise is similar to a univariate version of the expressions in Lemma~\ref{LemmaprelimjointB}.  Use (\ref{postBrep}) to express 
$\hat{Z}_{j}$ as a sum of an $\mathrm{IBP}(A_{j},\rho_{j,M_{j}}\tilde{B}_{0,J})$ and $\mathrm{IBP}(A_{j},\rho_{j,M_{j}}\sum_{k=1}^{r}\tilde{L}_{k}\delta_{\tilde{Y}_{k}})$ process. The first term in (\ref{newpartpredictA})
follows by applying Theorem~\ref{thmHIBPmarginalgen} to the univariate $\mathrm{IBP}(A_{j},\rho_{j,M_{j}}\tilde{B}_{0,J})$ with appropriate adjustments. The other expression  follows  directly by expanding $\sum_{k=1}^{r}\tilde{L}_{k}\delta_{\tilde{Y}_{k}}.$ 
\end{proof}

The next result describes the prediction rule for a previously unseen document type, say $Z^{(1)}_{J+1}|\mu_{J+1},B_{0}\sim \mathrm{IBP}(A_{J+1}|\mu_{J+1}).$
\begin{cor}
The predictive distribution of  $Z^{(1)}_{J+1}|((Z^{(i)}_{l},i \in [M_{j}]),l\in[J]),$ equates to the processes in (\ref{newpartpredictA}) setting $j=J+1$ and $M_{J+1}+1=1.$ 
\end{cor}
\subsubsection{Prediction calculations for $B_{0}$ generalized gamma}\label{predictGG}
In the case $B_{0}$ is $\mathrm{GG}(\alpha,\zeta;\theta_{0}G_{0})$
Section~\ref{Gengamma}, (\ref{Xfunction}), indicates that the iid collection $(\tilde{X}_{j,k})$ have common pmf 
$p_{\alpha}(m|\gamma_{j,M_{j}+1},\zeta+\sum_{j=1}^{J}\psi_{j}(M_{j}),$ and $\tilde{N}_{j,k}\overset{ind}\sim
\mathrm{NB}(n_{k}-\alpha,q_{j}),$ where 
$$
q_{j}=\frac{\psi_{j}(M_{j}+1)-\psi_{j}(M_{j})}{\zeta+\sum_{l\neq j}^{J}\psi_{l}(M_{l})+\psi_{j}(M_{j}+1)},
$$
where for $j=J,$ $\sum_{l\neq j}^{J}\psi_{l}(M_{l})=\sum_{l=1}^{J-1}\psi_{l}(M_{l}).$ Additionally, from (\ref{rfunction}), $\psi_{0,J}(\kappa)=\theta_{0}
\tilde{\psi}_{\alpha,\tilde{\zeta}_{J}}(\kappa),$
where $\tilde{\zeta}_{J}=\zeta+\sum_{j=1}^{J}\psi_{j}(M_{j}).$

\subsection{Results for the Poisson HIBP case}\label{PoissonCase}
We now obtain calculations and results in the case of  the Poisson HIBP models in Proposition~\ref{mixedPoisson}. As a by-product, our results lead to tractable sampling schemes for general versions of models in~\cite{ZhouCarin2015,Zhou1, ZhouPadilla2016}. In the Poisson cases
\begin{equation}
\psi_{j}(M_{j})=\int_{0}^{\infty}(1-{\mbox e}^{-sM_{j}r_{j}})\rho_{j}(s)ds.
\label{psiPoisson}
\end{equation}
We first apply Proposition~\ref{bigmargBgen} to obtain alternative conditional representations.
\begin{cor}\label{PoissonCor1}Consider the general Poisson HIBP setting in Proposition~\ref{mixedPoisson}, with $A_{j}|s\sim \mathrm{Poisson}(r_{j}s)$, $j\in [J].$ Then from that result the conditional distribution of the processes $((Z^{(i)}_{j},i\in [M_{j}]),j\in[J])|B_{0}$, has the mixed Poisson representations  
$((\sum_{k=1}^{\infty}\mathscr{P}^{(i)}_{j,k}(r_{j}\sigma_{j,k}(\tau_{k}))\delta_{Y_{k}},i\in [M_{j}]), j\in [J])$ in (\ref{FullPoissHIBP}), and can equivalently be expressed in terms of compound Poisson random variables as follows, for each $j\in [J]$
\begin{equation}
(\sum_{k=1}^{\infty}\mathscr{P}^{(i)}_{j,k}(r_{j}\sigma_{j,k}(\tau_{k}))\delta_{Y_{k}},i\in [M_{j}])\overset{d}=(\sum_{k=1}^{\infty}[\sum_{l=1}^{N_{j,k}}{X}^{(i)}_{j,k,l}]\delta_{{Y}_{k}},i\in [M_{j}]),
\label{FullPoissHIBP}
\end{equation}
where the vectors $X_{0,j,k,l}:=(X^{(i)}_{j,k,l},i
\in [M_{j}])$
have distributions $\mathrm{MtP}((r_{j},i\in [M_{j}]),\rho_{j}),$ which can be  read from Proposition~\ref{multisample}, and are independent across $(j,k,l).$ 
Hence the conditional distribution of $(\sum_{i=1}^{M_{j}}Z^{(i)}_{j},j\in[J])|B_{0}$ has the representations
\begin{equation}
(\sum_{k=1}^{\infty}\mathscr{P}_{j,k}(M_{j}r_{j}\sigma_{j,k}(\tau_{k}))\delta_{Y_{k}}, j\in [J])\overset{d}=(\sum_{k=1}^{\infty}[\sum_{l=1}^{N_{j,k}}{\tilde{X}}_{j,k,l}]\delta_{{Y}_{k}},j\in [J]),
\label{FullPoissHIBP}
\end{equation}
where $\tilde{X}_{j,k,l}:=\sum_{i=1}^{M_{j}}{X}^{(i)}_{j,k,l}\sim \mathrm{MtP}(M_{j}r_{j},\rho_{j}),$ independent across $(j,k,l).$ For clarity $X_{0,j,k,l}|\tilde{X}_{j,k,l}=a_{j,k,l}$ are $\mathrm{Multinomial}(a_{j,k,l},(1/M_{j},i\in [M_{j}])),$ 
\end{cor}

We now apply Theorem~\ref{thmHIBPmarginalgen} to obtain the following important representations which allows for explicit sampling. In particular, the result provides a natural extension of the sampling approaches of~\cite{ZhouPadilla2016}, which involves the compound Poisson representation of the Negative Binomial distribution. See the forthcoming Section~\ref{gengammacount} for further details.  
\begin{cor}\label{cormixedPoissonsim}Consider the setting and specifications in~Corollary~\ref{PoissonCor1}. Then the unconditional distribution of the processes $((Z^{(i)}_{j},i\in [M_{j}]),j\in[J])$, can be expressed as,
\begin{equation}
((Z^{(i)}_{j},i\in [M_{j}]),j\in[J])\overset{d}=((\sum_{k=1}^{\xi(\varphi)}[\sum_{l=1}^{X_{j,k}}{X}^{(i)}_{j,k,l}]\delta_{\tilde{Y}_{k}},i\in [M_{j}]),j\in [J]),
\label{FullPoissHIBP}
\end{equation}
where $\xi(\varphi), (\tilde{Y}_{k})$ and $X_{0,k}=(X_{1,k},\ldots,X_{J,k})\sim \mathrm{MtP}((\psi_{j}(M_{j}),j\in [J]),\tau_{0}),$ are as in Proposition~\ref{multisample}. Hence the unconditional distribution of $(\sum_{i=1}^{M_{j}}Z^{(i)}_{j},j\in[J])$ has the representations
\begin{equation}
(\sum_{k=1}^{\infty}\mathscr{P}_{j,k}(M_{j}r_{j}\sigma_{j,k}(\tau_{k}))\delta_{Y_{k}}, j\in [J])\overset{d}=(\sum_{k=1}^{\xi(\varphi)}[\sum_{l=1}^{X_{j,k}}{\tilde{X}}_{j,k,l}]\delta_{\tilde{Y}_{k}},j\in [J]),
\label{FullPoissHIBPsum}
\end{equation}
where $\tilde{X}_{j,k,l}:=\sum_{i=1}^{M_{j}}{X}^{(i)}_{j,k,l}\sim \mathrm{MtP}(M_{j}r_{j},\rho_{j}),$ independent across $(j,k,l).$ 
\end{cor}
In terms of observations $X_{0,j,k,l}:=(\tilde{X}^{(i)}_{j,k,l},i\in [M_{j}])=(a^{(i)}_{j,k,l},i\in [M_{j}]),$ with $\tilde{X}_{j,k,l}=\sum_{i=1}^{M_{j}}a^{(i)}_{j,k,l}:=a_{j,k,l}>0.$  $(X_{1,k}=n_{1,k},\ldots,X_{J,k}=n_{J,k})$ with $\tilde{X}_{k}=\sum_{j=1}^{J}n_{j,k}=n_{k}>0.,$ leading to $((\sum_{l=1}^{n_{j,k}}a^{(i)}_{j,k,l},i\in [M_{j}]),j\in[J])$ and 
$(\sum_{l=1}^{n_{j,k}}a_{j,k,l},j\in[J])$ as counts for each $\tilde{Y}_{k},$ for $k\in[r],$ where $\xi(\phi)=r.$ In this case,  $S_{j,k,l}$ has density 
\begin{equation}
\vartheta(s|\mathbf{a}_{j,k,l},\rho_{j})=\eta(s|a_{j,k,l},M_{j}r_{j},\rho_{j})
\propto s^{a_{j,k,l}}{\mbox e}^{-M_{j}r_{j}s}\rho_{j}(s)
\label{Poissongenjump}
\end{equation}
and hence are in the same family of distributions as the jumps $(\tilde{L}_{k})$ of $B_{0}$ specified in~(\ref{JumpsDist1}). Furthermore, 
in regards to the calculation of the marginal distribution in  Proposition~\ref{marginalslabHIBP}, it follows that now $\mathbf{S}_{[J]}(\mathbf{a})$ in (\ref{Alikelihood}) has components,

\begin{equation}
\mathsf{S}_{j}(\mathbf{a}_{j,k,l}|\rho_{j})=\frac{a_{j,k,l}!}{\prod_{i=1}^{M_{j}}a^{(i)}_{j,k,l}!}
{(\frac{1}{M_{j}})}^{a_{j,k,l}}
\times\tilde{p}(a_{j,k,l}|M_{j}r_{j},\rho_{j})
\label{jointMjPoisson}
\end{equation}
corresponding to a $\mathrm{MtP}((r_{j},i\in [M_{j}]),\rho_{j}),$ distribution. We now apply Proposition~\ref{postgenpredict} to obtain descriptions of the prediction rule.
\begin{cor}\label{corPoissonpredict}
Consider the general specifications in Proposition~
\ref{postgenpredict}. Then the Poisson case of $A_{j}|s\sim Poisson(r_{j}s)$ for each $j$ yields prediction rules of the form
\begin{equation}
\hat{Z}_{j}+\sum_{k=1}^{r}\left[\mathscr{P}_{j,k}(r_{j}\sum_{l=1}^{n_{j,k}}S_{j,k,l})\right]\delta_{\tilde{Y}_{k}},
\label{PoissonpredictAMj}
\end{equation}
where,
$\hat{Z}_{j}|((Z^{(i)}_{l},i \in [M_{j}]),l\in[J]),$ can be represented as
\begin{equation}
\sum_{k=1}^{\tilde{\xi}_{j}(\phi_{j})}\left[\sum_{l=1}^{\tilde{X}_{j,k}}
\tilde{X}_{j,k+r,l}\right]\delta_{Y_{j,k}}+
\sum_{k=1}^{r}\left[\sum_{l=1}^{\tilde{N}_{j,k}}
\tilde{X}_{j,k,l}\right]\delta_{\tilde{Y}_{k}}
\label{PoissonnewpartpredictA}
\end{equation}
where now 
 $((\tilde{X}_{j,t,l},t\in[r+\tilde{\xi}_{j}(\phi_{j})]))$ are collections of iid variables each with distribution 
$\mathrm{MtP}(r_{j},\rho_{j,M_{j}}),$ for $\rho_{j,M_{j}}(s)={\mbox e}^{-M_{j}r_{j}s}\rho_{j}(s).$ The case of $Z^{(1)}_{J+1}|((Z^{(i)}_{l},i \in [M_{j}]),l\in[J]),$
is equivalent to (\ref{PoissonnewpartpredictA}) with $\gamma_{J+1,1} $ in place of $\gamma_{j,M_{j}+1},$ and
 $((\tilde{X}_{j,t,l},t\in[r+\tilde{\xi}_{j}(\phi_{j})]))$ iid
$\mathrm{MtP}(r_{J+1},\rho_{J+1}).$
\end{cor}

\subsubsection{Poisson generalized gamma cases}\label{gengammacount}
As in~Section~\ref{Gengamma}
we suppose that $B_{0}$ is $\mathrm{GG}(\alpha,\zeta;\theta_{0}G_{0}),$ and we  choose 
$\mu_{j}|B_{0}\overset{ind}\sim \mathrm{GG}(\alpha_{j},\zeta_{j};\theta_{j}B_{0}),$ for $j\in[J].$ Then,
$$
\psi_{j}(M_{j})=\theta_{j}\tilde{\psi}_{\alpha_{j},\zeta_{j}}(M_{j}r_{j}),
$$
as in (\ref{rfunction}), for $j\in [J],$   and
$\tilde{X}_{j,k,l}:=\sum_{i=1}^{M_{j}}{X}^{(i)}_{j,k,l}\sim \mathrm{MtP}(M_{j}r_{j},\rho_{j}),$ has pmf $p_{\alpha_{j}}(m|M_{j}r_{j},\zeta_{j})$ 
for $j\in [J],$  as in~(\ref{Xfunction}).
When $\alpha=0,$ and $\alpha_{j}=0$ for $j\in[J],$ the processes in $(\ref{FullPoissHIBPsum})$ coincide with the generation of Gamma-Negative Binomial random count matrices as described in~\cite[Section 2.2]{ZhouPadilla2016}. In fact the sampling algorithms proposed in \cite{ZhouPadilla2016}, in those 
cases, equate with the right hand expression in 
(\ref{FullPoissHIBPsum}), where as derived and exploited in~\cite{ZhouCarin2015, ZhouPadilla2016}, there is a closed form expression for the probability mass function of 
$\sum_{l=1}^{X_{j,k}}{\tilde{X}}_{j,k,l}$ when the $({\tilde{X}}_{j,k,l})$ have iid Logarithmic series distributions. See ~\cite[Supplement]{ZhouFoF} for some relevant details for $\alpha_{j}\neq 0.$ However, sampling in this more general case, which allows for modelling more general data structures, is made straightforward by (\ref{FullPoissHIBPsum}). This is of course one sub-component of the more complex structure in 
Corollary~\ref{cormixedPoissonsim}. Additionally, for fixed $j,$ when $\alpha_{j}=0,$ $(X_{0,j,k,l}:=(\tilde{X}^{(i)}_{j,k,l},i\in [M_{j}]),l\in [n_{j,k}],k\in[r])$ 
with $X_{j,k}=n_{j,k}$ and $\xi(\varphi)=r,$ generate random count matrices in~\cite[eq. (6)]{ZhouPadilla2016}, with $M_{j}$ in place of $J.$ 
\begin{cor} Consider the specifications in this section. Then the marginal distribution of $((Z^{(i)}_{j},i\in [M_{j}]),j\in [J])$ is explicit and follows the descriptions in Proposition~\ref{marginalslabHIBP} and Corollary~\ref{altmargslab} by, (\ref {GGBlikelihood}) and 
$$
\Delta(\mathbf{a})=\prod_{j=1}^{J}\frac{\theta^{d_{j}}_{j}r^{a_{j,\cdot}}_{j}}{{(r_{j}M_{j}+\zeta_{j})}^{a_{j,\cdot}-\alpha_{j}d_{j}}}\prod_{k=1}^{r}\prod_{l=1}^{n_{j,k}}\frac{\Gamma(a_{j,k,l}-\alpha_{j})}{\Gamma(1-\alpha_{j})\prod_{i=1}^{M_{j}}a^{(i)}_{j,k,l}!}
$$
where $\sum_{k=1}^{r}\sum_{l=1}^{n_{j,k}}a_{j,k,l}=a_{j,\cdot},$ and $a_{j,k,l}\in\{1,2,\ldots,\}$ for $n_{j,k}\neq 0.$
\end{cor}

In terms of the prediction rule in Corollary~\ref{corPoissonpredict}, $S_{j,k,l}\overset{ind}\sim\mathrm{Gamma}(a_{j,k,l}-\alpha_{j},M_{j}r_{j}+\zeta_{j}),$ and hence  
$\sum_{l=1}^{n_{j,k}}S_{j,k,l}\overset{ind}\sim \mathrm{Gamma}(a_{j,k}-n_{j,k}\alpha_{j},M_{j}r_{j}+\zeta_{j}),$  
for $a_{j,k}:=\sum_{l=1}^{n_{j,k}}a_{j,k,l}.$ Hence
$\mathscr{P}_{j,k}(r_{j}\sum_{l=1}^{n_{j,k}}S_{j,k,l})\sim \mathrm{NB}(a_{j,k}-n_{j,k}\alpha_{j},p_{j}),$ for $p_{j}=r_{j}/((M_{j}+1)r_{j}+\zeta_{j}),$ and $((\tilde{X}_{j,t,l},t\in[r+\tilde{\xi}_{j}(\phi_{j})]))$ are collections of iid variables each with distribution 
 $p_{\alpha_{j}}(m|r_{j},\zeta_{j}+M_{j}r_{j}).$ 
\subsection{Results for Bernoulli HIBP}\label{HIBPresults}
We now return to the Bernoulli case, where we will encounter calculations certainly well known in the literature, see for example~\cite[Section 4.1]{James2017}. Here, 
$(A^{(1)}_{j},\ldots,A^{(M_{j})}_{j})|p$ have joint probability mass function $p^{m_{j}}{(1-p)}^{M_{j}-m_{j}}$, where $m_{j}=\sum_{i=1}^{M_{j}}a^{(i)}_{j},$ for $a^{(i)}_{j}\in{\{0,1\}},$ $i\in[M_{j}].$ Hence $\sum_{i=1}^{M_{j}}A^{(i)}_{j}|p\sim \mathrm{Binomial}(M_{j},p),$ and it follows that the distribution of  $(\hat{A}^{(1)}_{j,k,l},\ldots,\hat{A}^{(M_{j})}_{j,k,l})$ is, for $m_{j,k,l}:=\sum_{i=1}^{M_{j}}a^{(i)}_{j,k,l}\in [M_{j}],$
\begin{equation}
\mathsf{S}_{j}(\mathbf{a}_{j,k,l}|\rho_{j})=b_{j}(m_{j,k,l}|\rho_{j})
=\frac{\int_{0}^{1}p^{m_{j,k,l}}{(1-p)}^{M_{j}-m_{j,k,l}}\rho_{j}(p)dp}{\psi_{j}(M_{j})}
\label{truncatedMBernoulli}
\end{equation}
This leads to an interesting relationship to multivariate Hypergeometric distributions.
\begin{cor}\label{jointtruncbernoulli}
The vector $(\hat{A}^{(1)}_{j,k,l},\ldots,\hat{A}^{(M_{j})}_{j,k,l})$ with distribution~(\ref{truncatedMBernoulli}) may be described as follows. The $\sum_{i=1}^{M_{j}}\hat{A}^{(i)}_{j,k,l}$ has a mixed zero truncated Binomial distribution with probability mass function 
$$
\mathbb{P}(\sum_{i=1}^{M_{j}}\hat{A}^{(i)}_{j,k,l}=m)=\binom{M_{j}}{m}b_{j}(m|\rho_{j})
$$
for $m\in[M_{j}].$
Hence 
$(\hat{A}^{(1)}_{j,k,l},\ldots,\hat{A}^{(M_{j})}_{j,k,l})|\sum_{i=1}^{M_{j}}\hat{A}^{(i)}_{j,k,l}=m
$ has joint probability mass function $1/\binom{M_{j}}{m},$ for $\sum_{i=1}^{M_{j}}a^{(i)}_{j,k,l}=m.$ This corresponds to a simple multivariate Hypergeometric distribution.
\end{cor}

The jumps $S_{j,k,l}$ have density proportional to $p^{m_{j,k,l}}{(1-p)}^{M_{j}-m_{j,k,l}}\rho_{j}(p).$ If $\mu_{j}|B_{0}\overset{ind}\sim\mathrm{sBP}((\alpha_{j},\beta_{j}),\theta_{j},B_{0}),$ then 
$S_{j,k,l}\overset{ind}\sim \mathrm{Beta}(m_{j,k,l}-\alpha_{j},M_{j}-m_{j,k,l}+\beta_{j}+\alpha_{j}).$ In the simplest case of the basic Bernoulli HIBP, $\mu_{j}|B_{0}\overset{ind}\sim\mathrm{BP}(\theta_{j},B_{0}),$ one has $\psi_{j}(M_{j})=\theta_{j}\sum_{i=1}^{M_{j}}1/i,$ 
and
$
\mathbb{P}(\sum_{i=1}^{M_{j}}\hat{A}^{(i)}_{j,k,l}=m)=\frac{1/m}{\sum_{k=1}^{M_{j}}1/k}
$ for $m\in [M_{j}].$ In the case where  $\mu_{j}|B_{0}\overset{ind}\sim\mathrm{sBP}((\alpha_{j},\beta_{j}),\theta_{j},B_{0}),$ 
$$
\psi_{j}(M_{j})=\theta_{j}\sum_{i=1}^{M_{j}}\frac{\Gamma(1-\alpha_{j})\Gamma(\beta_{j}+\alpha_{j}+i-1)}{\Gamma(\beta_{j}+i)}.
$$
We now give an explicit description of the marginal distribution.
\begin{cor}\label{Bernoullimarg} Consider the specifications in this section where $\mu_{j}|B_{0}\overset{ind}\sim\mathrm{sBP}((\alpha_{j},\beta_{j}),\theta_{j},B_{0}),$ and $B_{0}\sim \mathrm{GG}(\alpha,\zeta; \theta_{0}G_{0}).$ Then the marginal distribution, in the Bernoulli case, of $((Z^{(i)}_{j},i\in [M_{j}]),j\in [J])$ is explicit and follows the descriptions in Proposition~\ref{marginalslabHIBP} and Corollary~\ref{altmargslab} by, (\ref {GGBlikelihood}) and 
$$
\Delta(\mathbf{a})=\prod_{j=1}^{J}{\theta^{d_{j}}_{j}}\prod_{k=1}^{r}\prod_{l=1}^{n_{j,k}}\frac{\Gamma(m_{j,k,l}-\alpha_{j})\Gamma(M_{j}-m_{j,k,l}+\alpha_{j}+\beta_{j})}{\Gamma(M_{j}+\beta_{j})}
$$
\end{cor}
\subsubsection{The Bernoulli HIBP prediction rule}
We now specialize Proposition~\ref{postgenpredict} to arrive at a fairly simple description of the predictive distribution in the Bernoulli HIBP case.
\begin{cor}\label{bergenpredict}
Consider the specifications in Proposition~\ref{postgenpredict} for $A_{j}|p\sim \mathrm{Bernoulli}(p)$
then for  any $j\in [J],$ the predictive distribution of $Z^{(M_{j}+1)}_{j}$ given $((Z^{(i)}_{\ell},i\in [M_{\ell}]),\ell\in [J]),$ has the representation
\begin{equation}
\sum_{k=1}^{\tilde{\xi}_{j}(\phi_{j})}\tilde{X}_{j,k}\delta_{Y_{j,k}}+
\sum_{k=1}^{r}\tilde{N}_{j,k}\delta_{\tilde{Y}_{k}}+\sum_{k=1}^{r}\left[\sum_{l=1}^{n_{j,k}}A_{j,k,l}\right]\delta_{\tilde{Y}_{k}},
\label{predictBerj}
\end{equation}
where,
 $(A_{j,k,l})\sim \mathrm{Bernoulli}(S_{j,k,l}).$ If $j=J+1,$ then only the two left most terms in (\ref{predictBerj}) are used, since $n_{J+1,k}=0,$ for $k\in[r].$ In the special cases where $B_{0}\sim \mathrm{GG}(\alpha,\zeta; \theta_{0}G_{0}),$ and $\mu_{j}|B_{0}\overset{ind}\sim\mathrm{sBP}((\alpha_{j},\beta_{j}),\theta_{j},B_{0}),$ 
 $(\tilde{X}_{j,k})$ and $\tilde{N}_{j,k}$ are as described in Section~\ref{predictGG} and   $A_{j,k,l}\overset{ind}\sim \mathrm{Bernoulli}(\frac{m_{j,k,l}-\alpha_{j}}{M_{j}+\beta_{j}}).$
\end{cor}

The expression in (\ref{predictBerj}) shows that the classic Bernoulli HIBP has a prediction rule like a Poisson IBP plus a component  that dictates sharing within group $j$ of previous dishes according to independent sums of Bernoulli distributions. It follows that if one samples only one process in each $j=1,2,\ldots,$ then the HIBP corresponds to a multi-factor Poisson $\mathrm{IBP},$ with relations to~\cite{Caron2012}.

\section{Simulation studies for Bernoulli HIBP} Our results provide explicit expressions suitable for practical implementation in a variety of settings. Naturally, we hope this encourages the  development of a variety of efficient computational procedures elsewhere.  As an illustration, we conduct a simulation study with a simple version of the Bernoulli GG-Beta HIBP model described in Sections~\ref{Bernoulliclassic} and \ref{HIBPresults}. That is to say a slight variation of the original HIBP of~\cite{Thibaux}. Here, specifically we let $A_{j}|p_{j}\sim \mathrm{Bernoulli}(p_{j}),$ and choose
\begin{align}
    B_0 \sim \mathrm{GG}(\alpha, \zeta; \theta_0 G_0), \quad \mu_j | B_0 \sim \mathrm{BP}(\theta_j, B_0) \text{ for } j=1,\dots, J.
\end{align}
We set base distribution $G_0$ as a uniform distribution $\mathrm{Unif}(0, 1)$.
For GG, we only consider the infinite-activty case with nonzero $\alpha$, that is, the parameters verifying $\theta_0 > 0$, $0 < \alpha < 1$ and $\zeta > 0$.
We further fixed $\zeta= 1$ for simplicity. For BP, we tied the parameters $\theta_1 = \dots = \theta_J = \theta$.

\subsection{Generation from the model}
Using Theorem~\ref{thmHIBPmarginalgen}, and Corollary~\ref{jointtruncbernoulli}, we can easily simulate feature assignments $((Z_j^{(i)}, i \in [M_j]), j\in[J])$. Figure~\ref{fig:sim_generation} shows some properties of the data generated from the model. Due to the characteristic of GG, the distribution of the frequencies of the feature assignment counts seems to exhibit power-law properties whose exponent depends on the parameter $\alpha$ (Figure~\ref{fig:sim_generation}, left). The parameter $\alpha$ also plays an important role in overall number of features and the diversity in the feature assignments (Figure~\ref{fig:sim_generation}, right).

\begin{figure}
    \centering
    \begin{subfigure}[b]{0.55\linewidth}
    \includegraphics[width=0.95\linewidth]{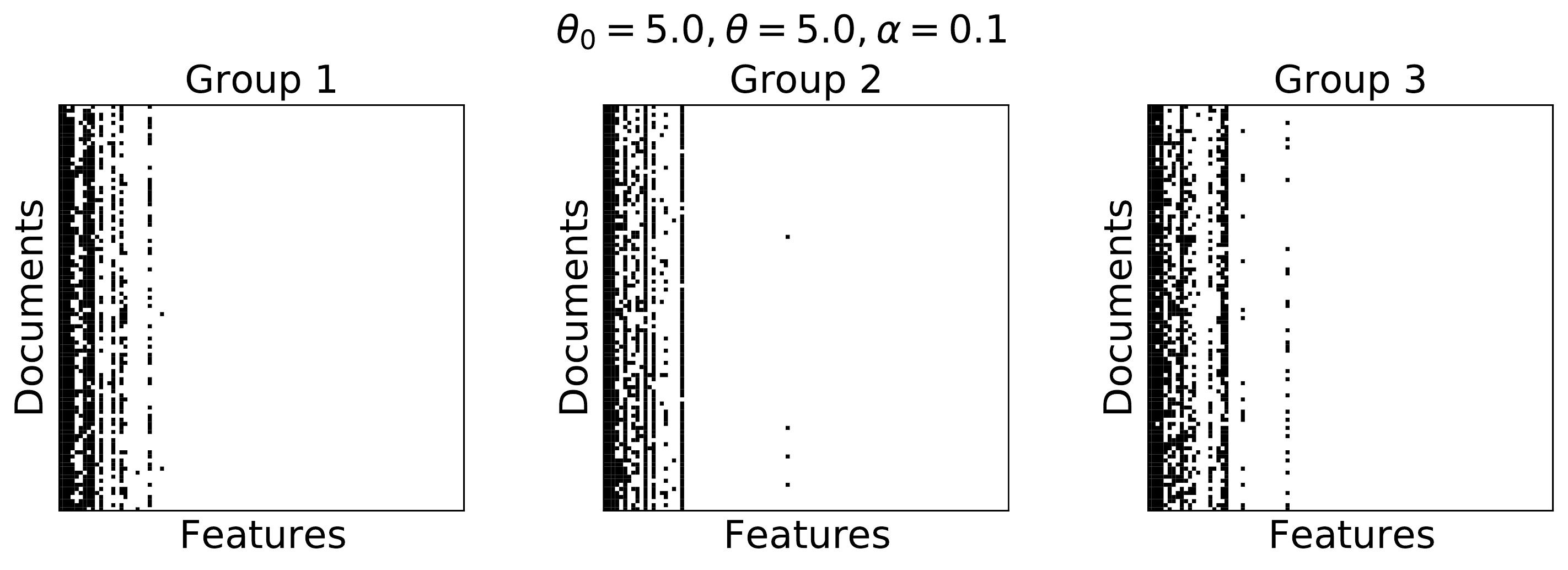}
     \includegraphics[width=0.95\linewidth]{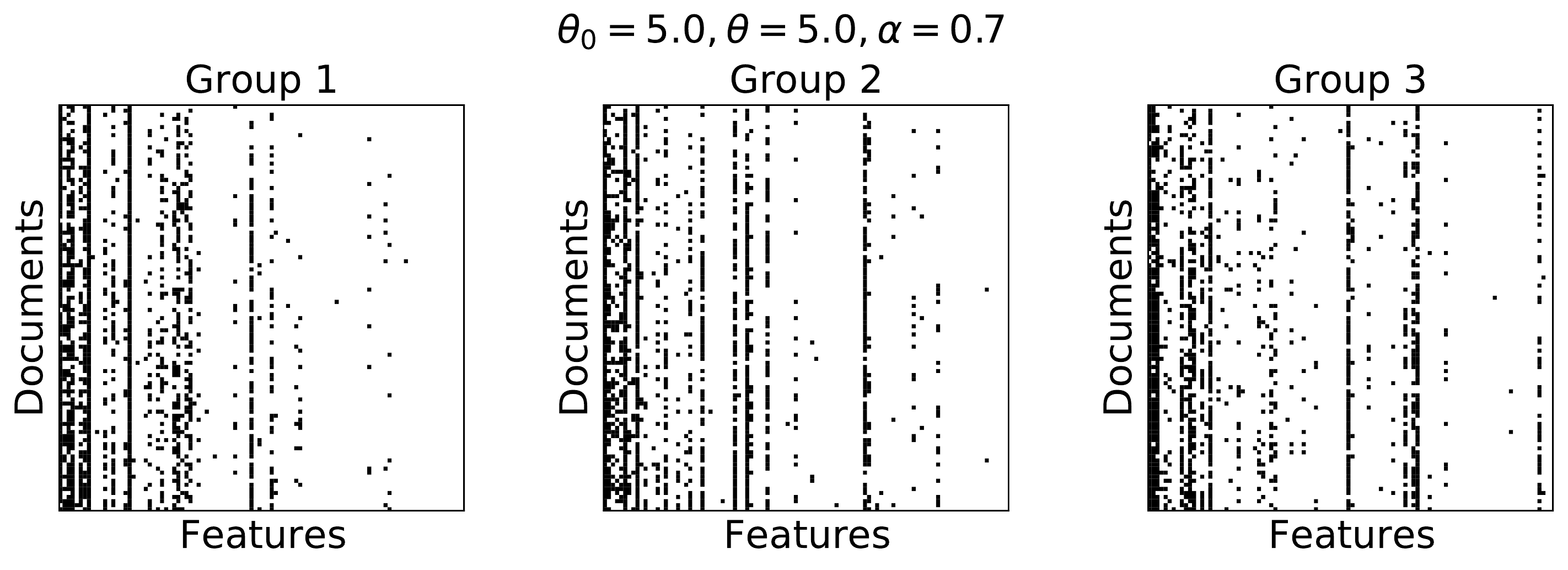}
     \end{subfigure}
      \begin{subfigure}[b]{0.40\linewidth}
    \includegraphics[width=0.95\linewidth]{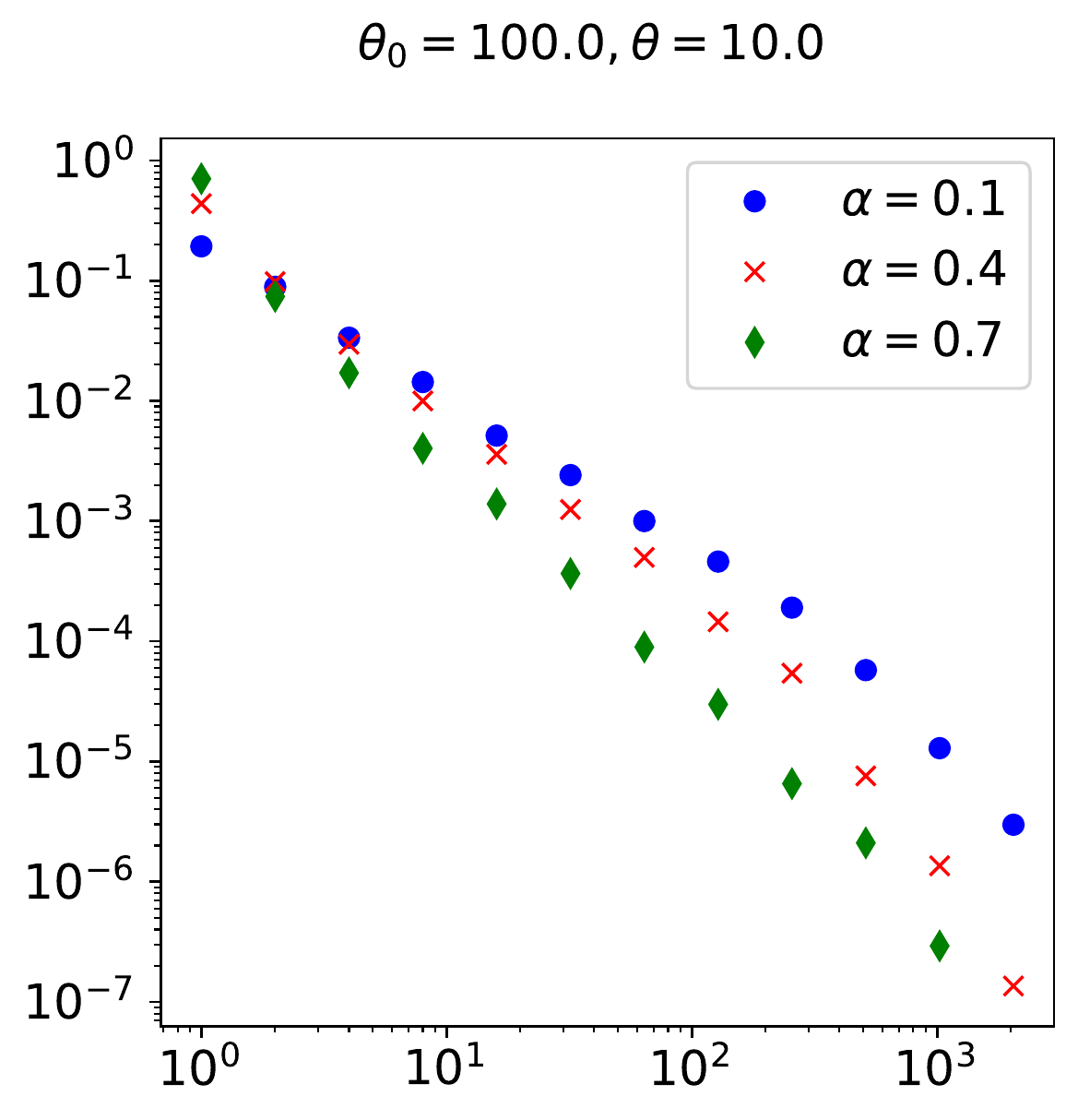}
    \end{subfigure}
    \caption{(Left) Plots of the feature assignment matrices generated from GG-Beta HIBP model. Black cells indicate at least one feature and white cells indicate no features. Top row shows the matrices for first three groups generated with $\theta_0=5.0, \theta=5.0, \alpha=0.1$. Bottom row shows the same plot with $\alpha=0.7$. (Right) Log-log plot of Frequency of Frequency (FoF) distributions of the counts generated from GG-Beta HIBP model with various $\alpha$ values.}
    \label{fig:sim_generation}
\end{figure}

\subsection{Posterior inference for the parameters}
Given an observed set of feature assignments $\mathbf{Z} := ((Z_j^{(i)}, i \in [M_j]), j \in [J])$, using Proposition \ref{marginalslabHIBP}, and more specifically Corollary~\ref{Bernoullimarg}, we can easily compute the marginal distribution $\mathbb{P}(\mathbf{Z}|\theta_0, \theta, \alpha)$, which, if we collect the terms related to the parameters $(\theta_0, \theta, \alpha)$, is proportional to,
\begin{equation}
\tilde\psi_0(\kappa)^{r} e^{-\tilde\psi_0(\kappa)} 
\prod_{k=1}^{r} 
\frac{\kappa^{n_{k}}(\kappa+\zeta)^{\alpha-n_{k}}}{(\kappa+\zeta)^\alpha-\zeta^\alpha}\frac{\alpha\Gamma(n_{k}-\alpha)}{\Gamma(1-\alpha)}
\prod_{j=1}^J \bigg(\frac{\psi_j(M_j)}{\kappa}\bigg)^{d_{j}},
\label{berECPF}
\end{equation}
where $n_{j,k}$ is the count of $k$th feature in the $j$th group, $n_{k} = \sum_{j=1}^J n_{j,k}$, $\psi_j(M_j) = \theta\sum_{i=1}^{M_j}1/i$, and $\kappa = \sum_{j=1}^J \psi_j(M_j)$. Hence, one can easily infer the parameters $(\theta_0, \theta, \alpha)$ using Markov-chain Monte-Carlo (MCMC) with some prior on $(\theta_0, \theta, \alpha)$. As a proof-of-concept, we generated two datasets from GG-Beta HIBP model with different parameter settings. The first one was generated with $\theta_0=2.0, \theta=2.0,$ and $\alpha=0.7$, and the second one was generated with $\theta_0=5.0, \theta=5.0$, and $\alpha=0.2$. Then we ran a simple random-walk Metropolis-Hastings sampler to see if we can recover the parameters used to generate data. For each dataset, we ran three independent chains where each chain was ran for 30,000 steps with 15,000 burn-in steps. Figure~\ref{fig:posterior_inference} shows that this simple algorithm properly infer the parameters, especially for the dataset with $\alpha=0.7$. For more details, refer to the Appendix~\ref{appendix:experiments}.

\begin{figure}
    \centering
        \includegraphics[width=0.3\linewidth]{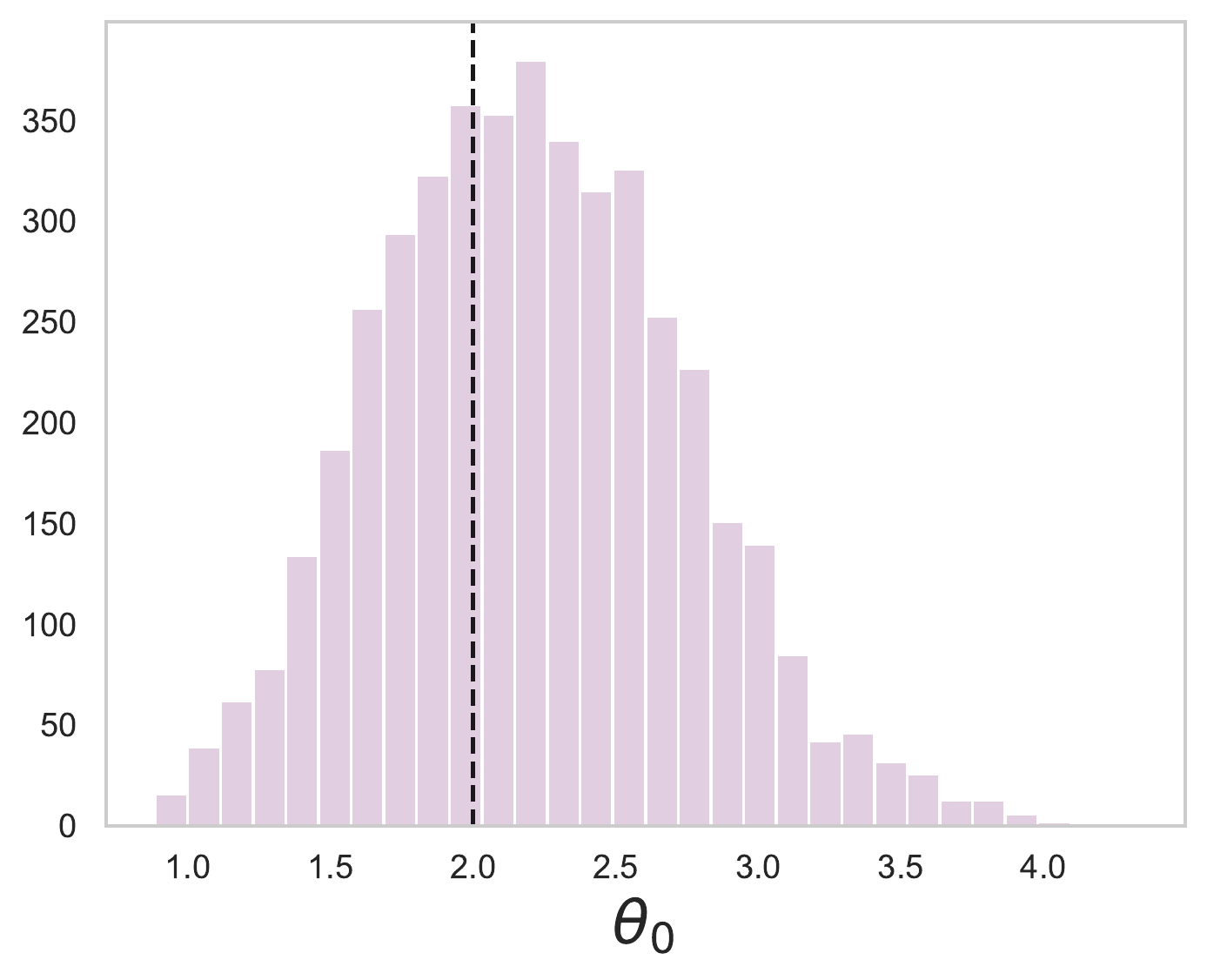}
    \includegraphics[width=0.3\linewidth]{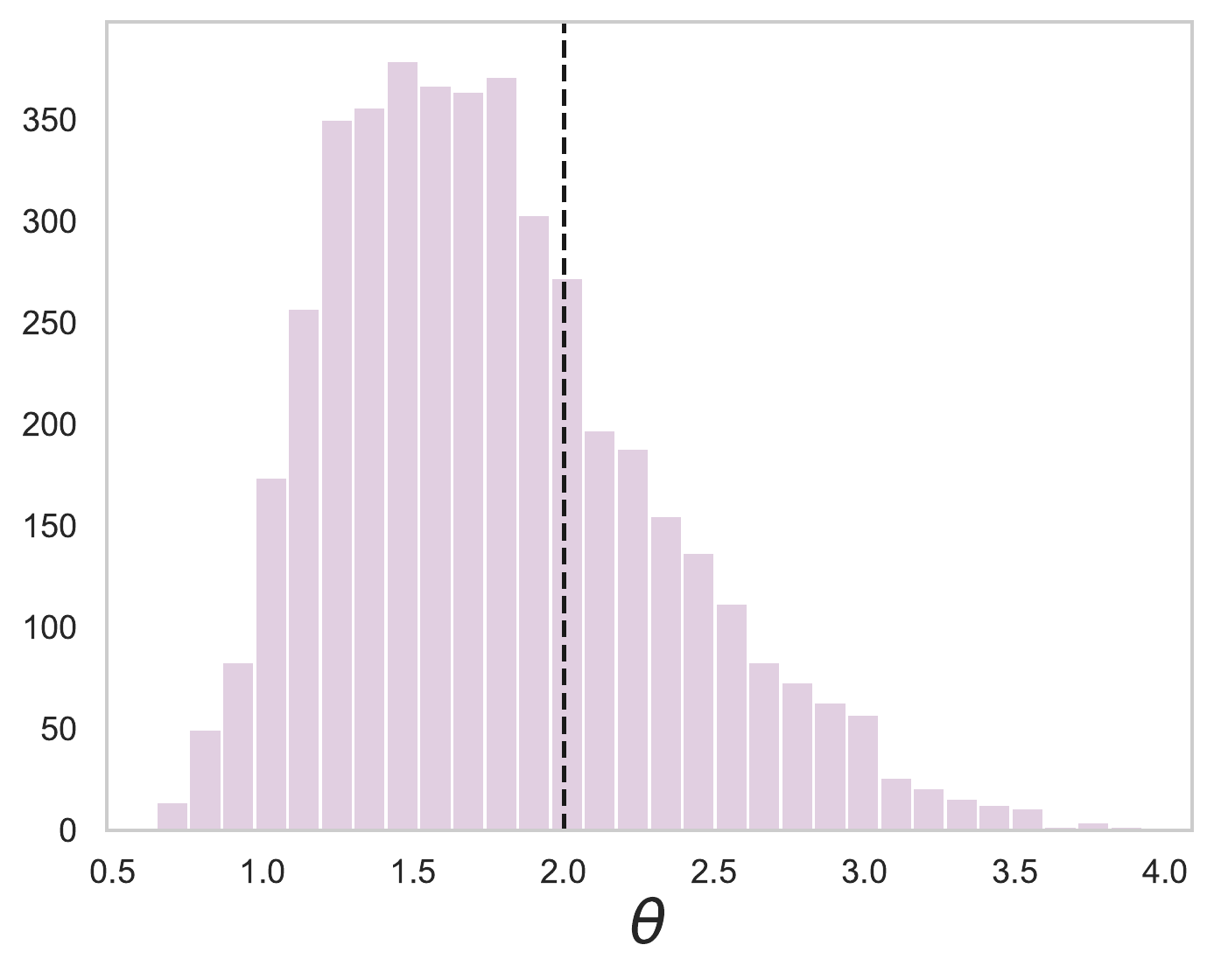}
    \includegraphics[width=0.3\linewidth]{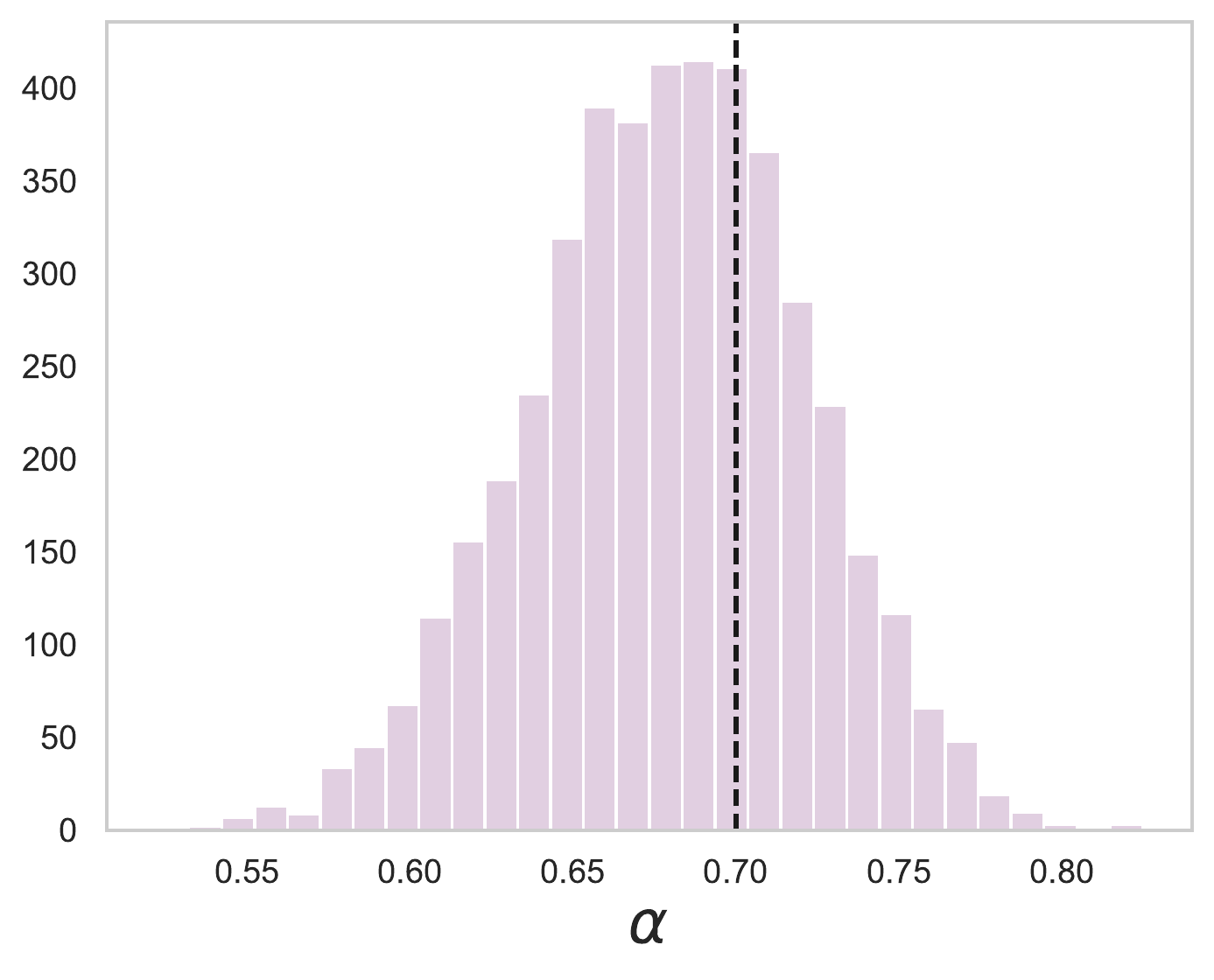}
    \includegraphics[width=0.3\linewidth]{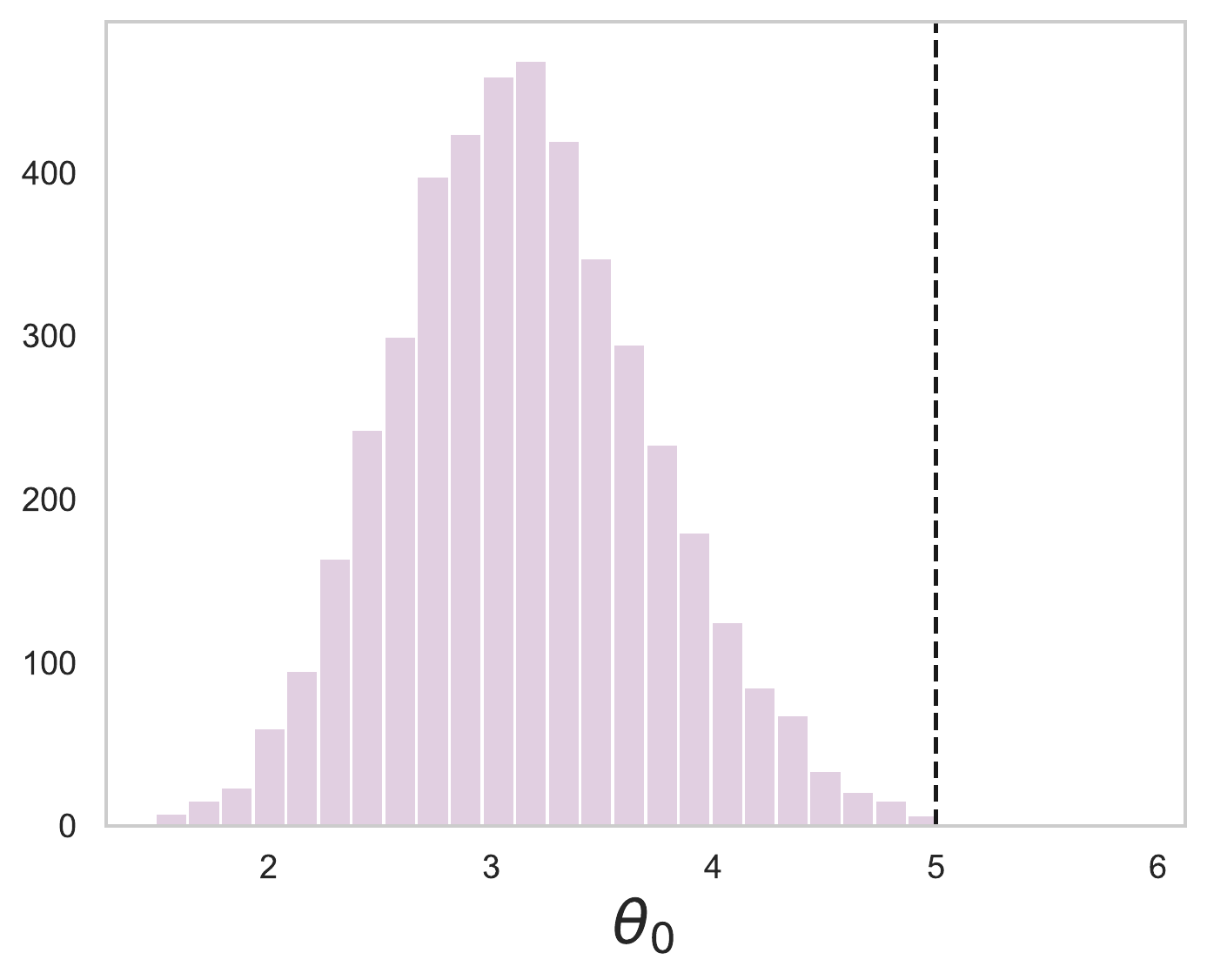}
    \includegraphics[width=0.3\linewidth]{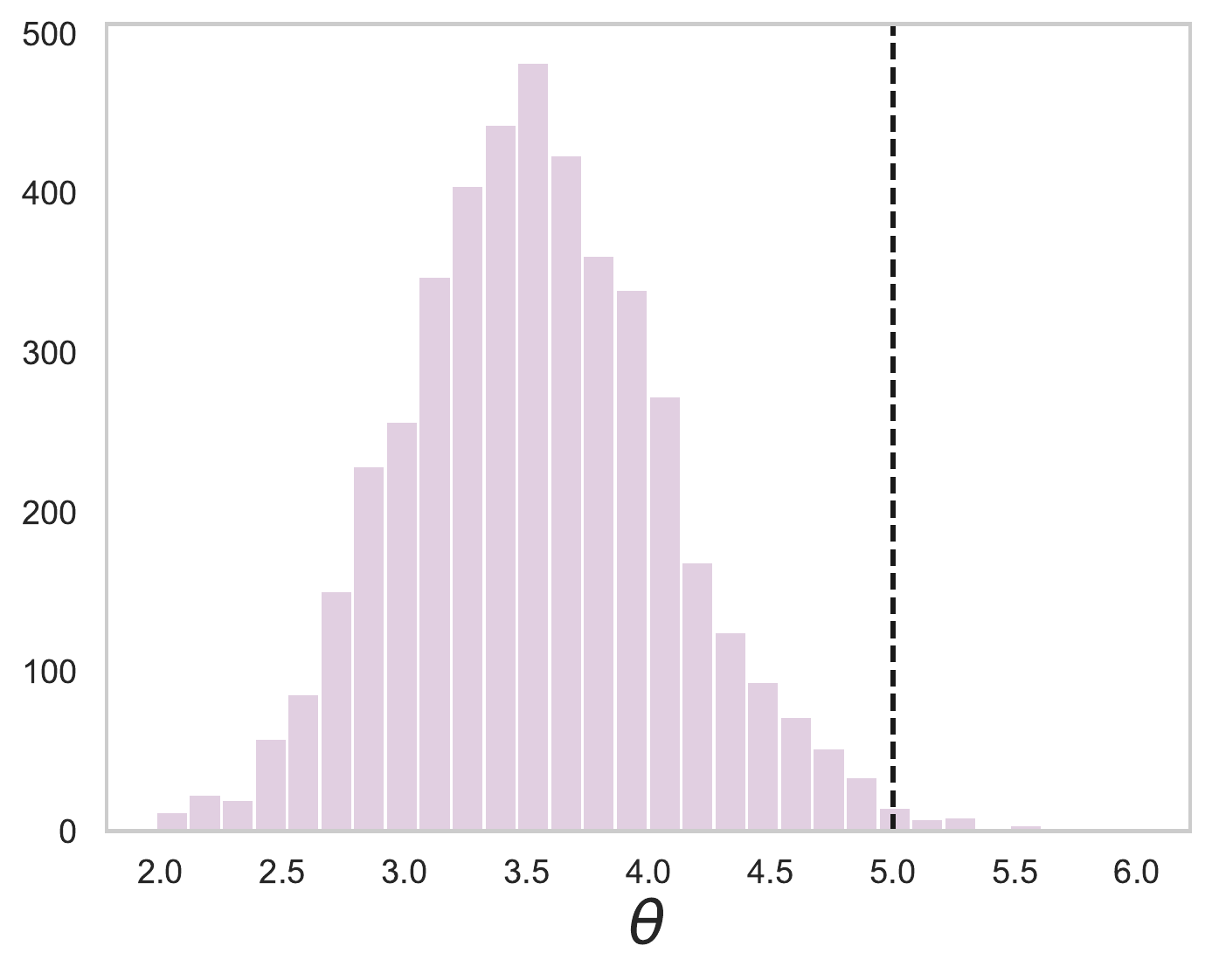}
    \includegraphics[width=0.3\linewidth]{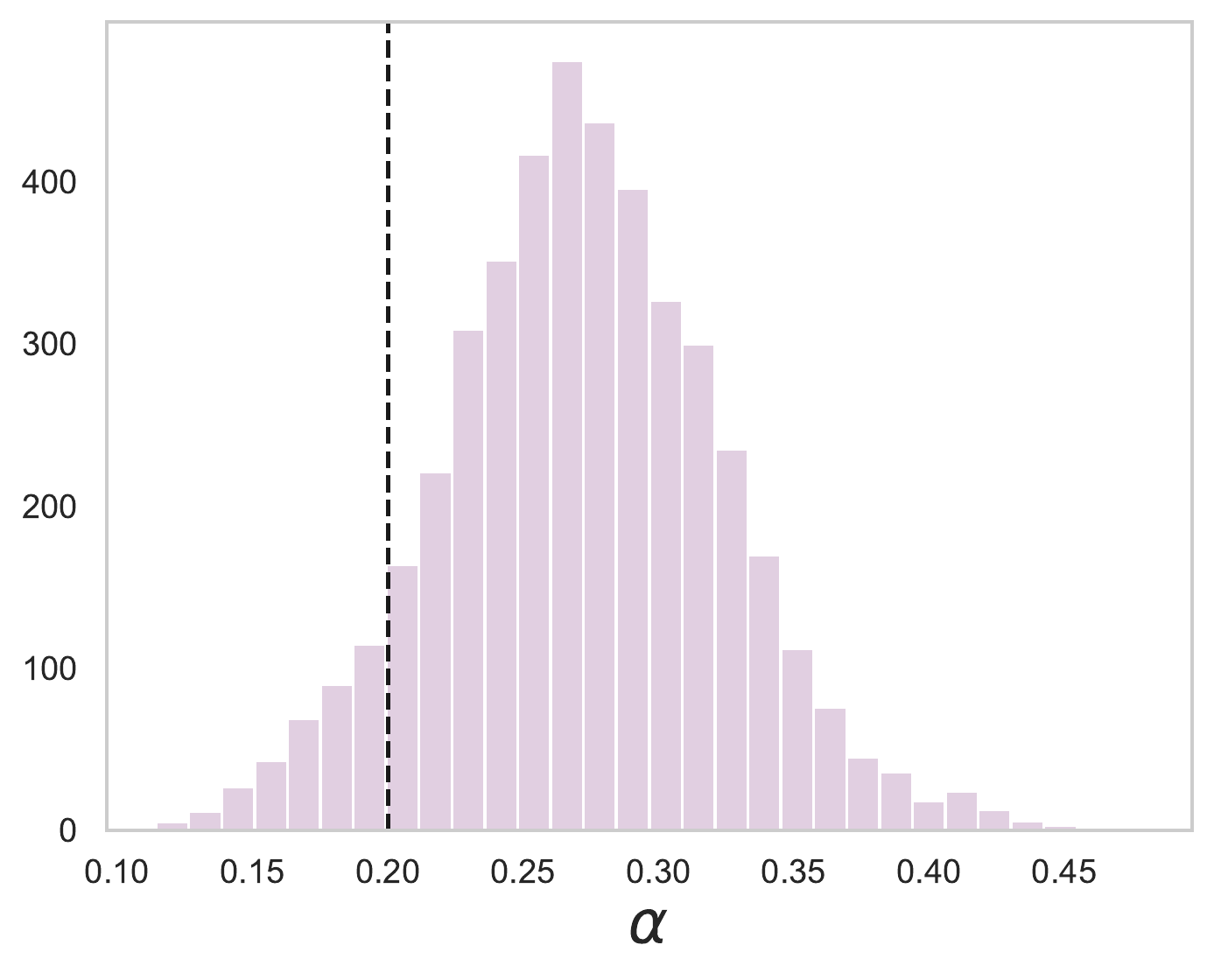}
    \caption{Posterior samples for parameters $\theta_0$ (left), $\theta$ (middle), and $\alpha$) (right). The top row for the data generated with $(\theta_0, \theta, \alpha) = (2.0, 2.0, 0.7)$, and the bottom row for the data generated with $(\theta_0, \theta, \alpha) = (5.0, 5.0, 0.2)$.}
    \label{fig:posterior_inference}
\end{figure}

\begin{rem} The expression in~(\ref{berECPF}) is proportional to a variant of the ECPF that appears in~
\cite[eq. (13)]{ZhouFoF}, hence one may consult~\cite[Section 4 and supplementary materials]{ZhouFoF} for further inference and computational ideas and procedures.   
\end{rem}
\begin{appendix}
\section{Details on the experiments}
\label{appendix:experiments}
We provide details on the posterior inference procedure described in the main text. As we stated in the main paper, we generated two datasets, one with $(\theta_0, \theta, \alpha) = (2.0, 2.0, 0.7)$ and another with $(\theta_0, \theta, \alpha) = (5.0, 5.0, 0.2)$. For each configuration, we generated documents of $J=10$ groups and $M_j = 5,000$ documents per each group, hence $50,000$ documents in total. We placed standard log-normal prior for $\theta_0$ and $\theta$, and $\mathrm{Beta}(1,1)$ (uniform) prior for $\alpha$. We used the standard random-walk Metropolis Hastings with variance $0.05$, and ran three independent chains for each dataset. For each chain, among 30,000 samples, we collected every 10th samples after discarding 15,000 burn-in samples. Figure~\ref{fig:data1_traces} and \ref{fig:data2_traces} show the log probability traces and posterior parameter traces.

\begin{figure}
    \centering
    \includegraphics[width=0.45\linewidth]{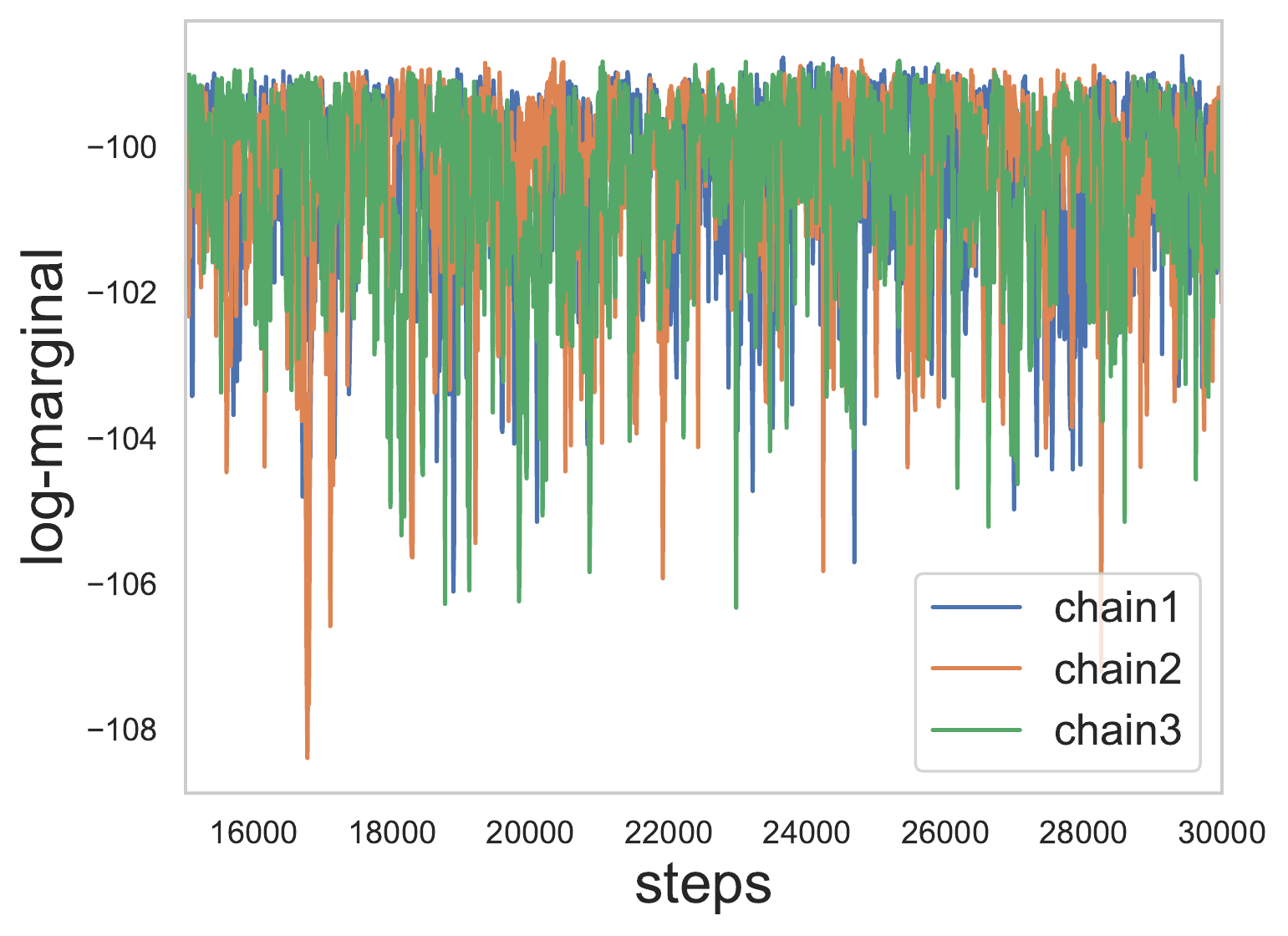}
        \includegraphics[width=0.45\linewidth]{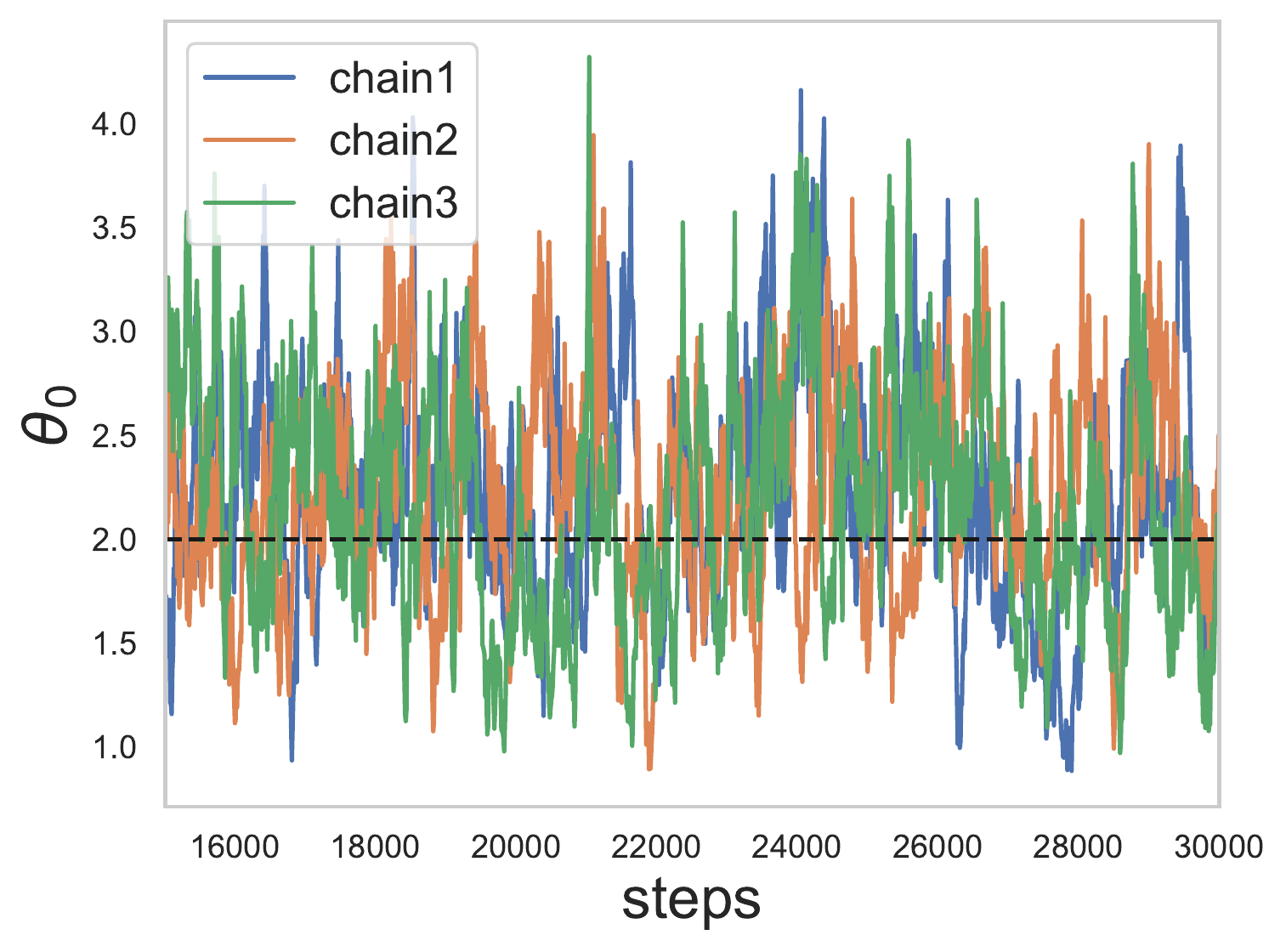}\\
            \includegraphics[width=0.45\linewidth]{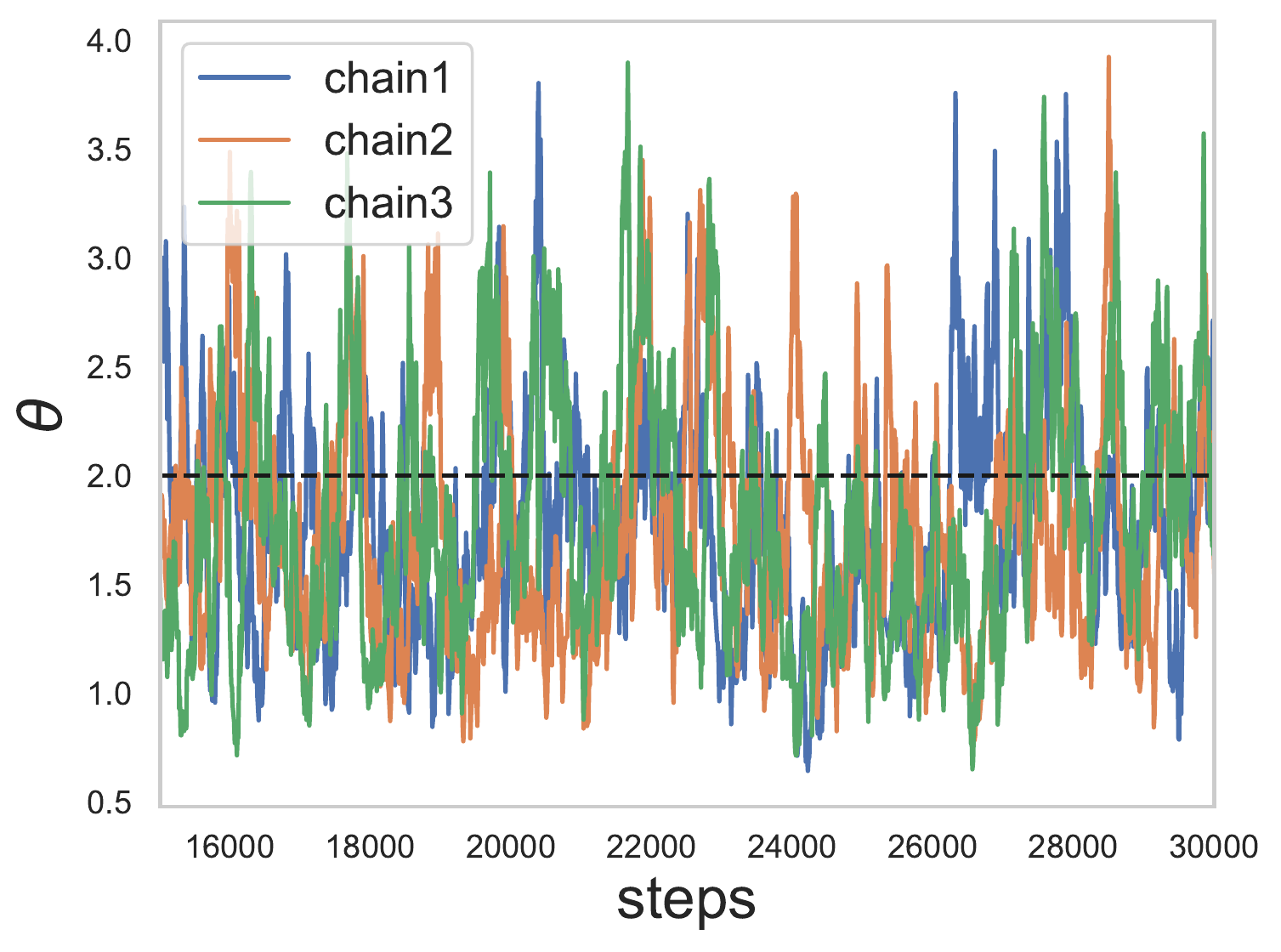}
                \includegraphics[width=0.45\linewidth]{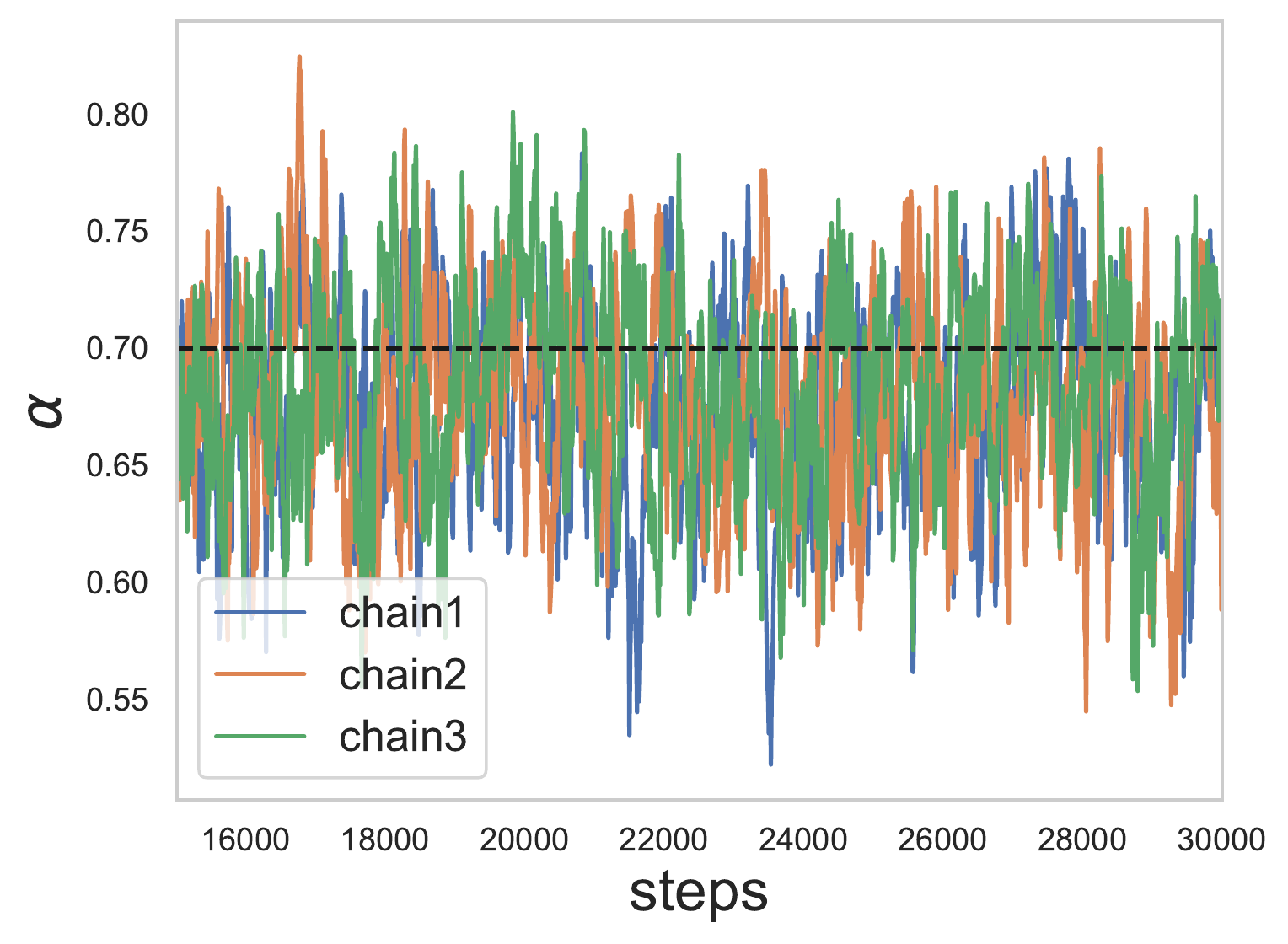}
    \caption{Posterior inference results for the dataset generated with $(\theta_0, \theta, \alpha) = (2.0, 2.0, 0.7)$. Log-marginal trace (upper left), sample traces for $\theta_0$ (upper right), $\theta$ (lower left), and $\alpha$ (lower right). Black dashed lines indicate true values.}
    \label{fig:data1_traces}
\end{figure}

\begin{figure}
    \centering
    \includegraphics[width=0.45\linewidth]{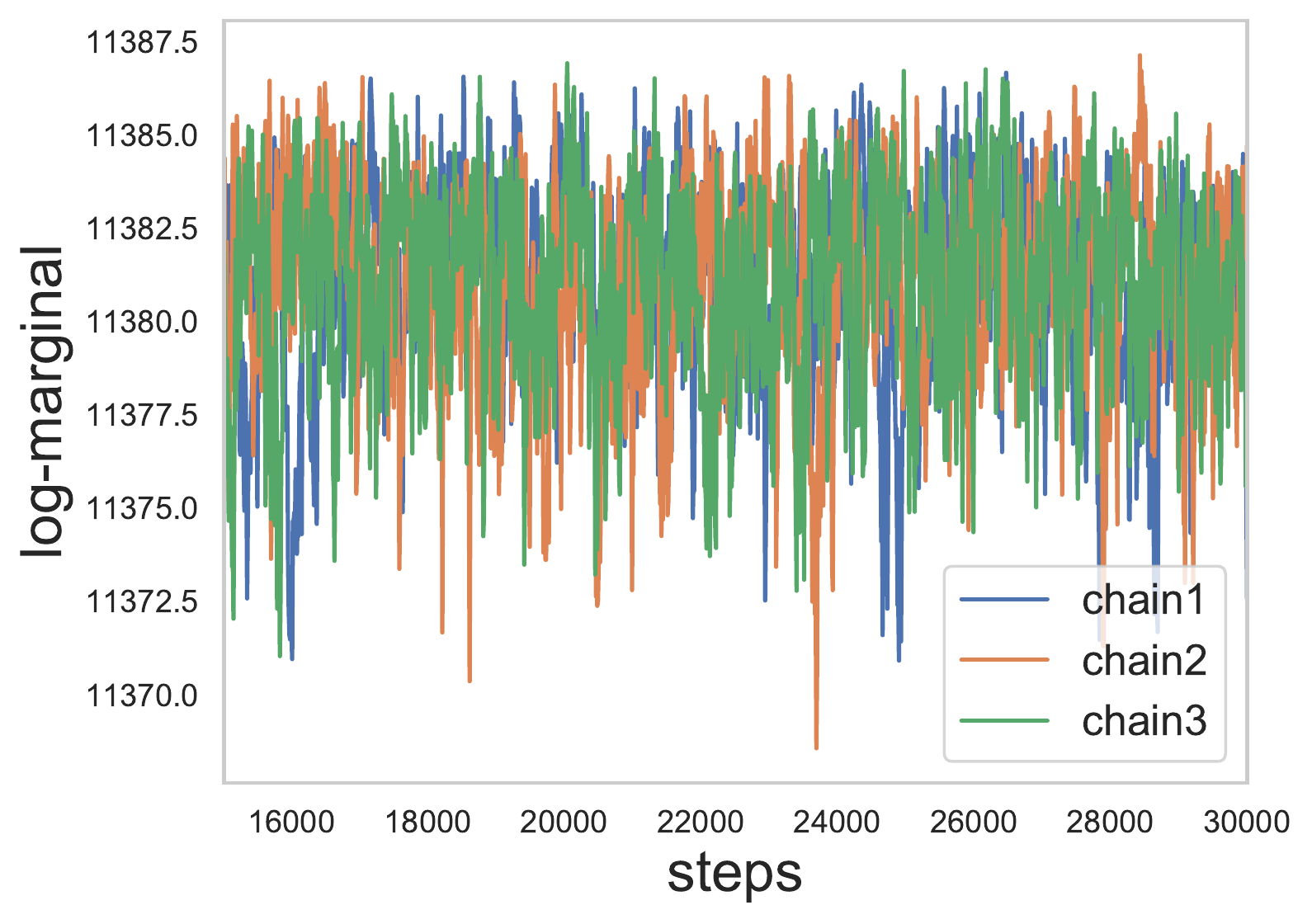}
        \includegraphics[width=0.45\linewidth]{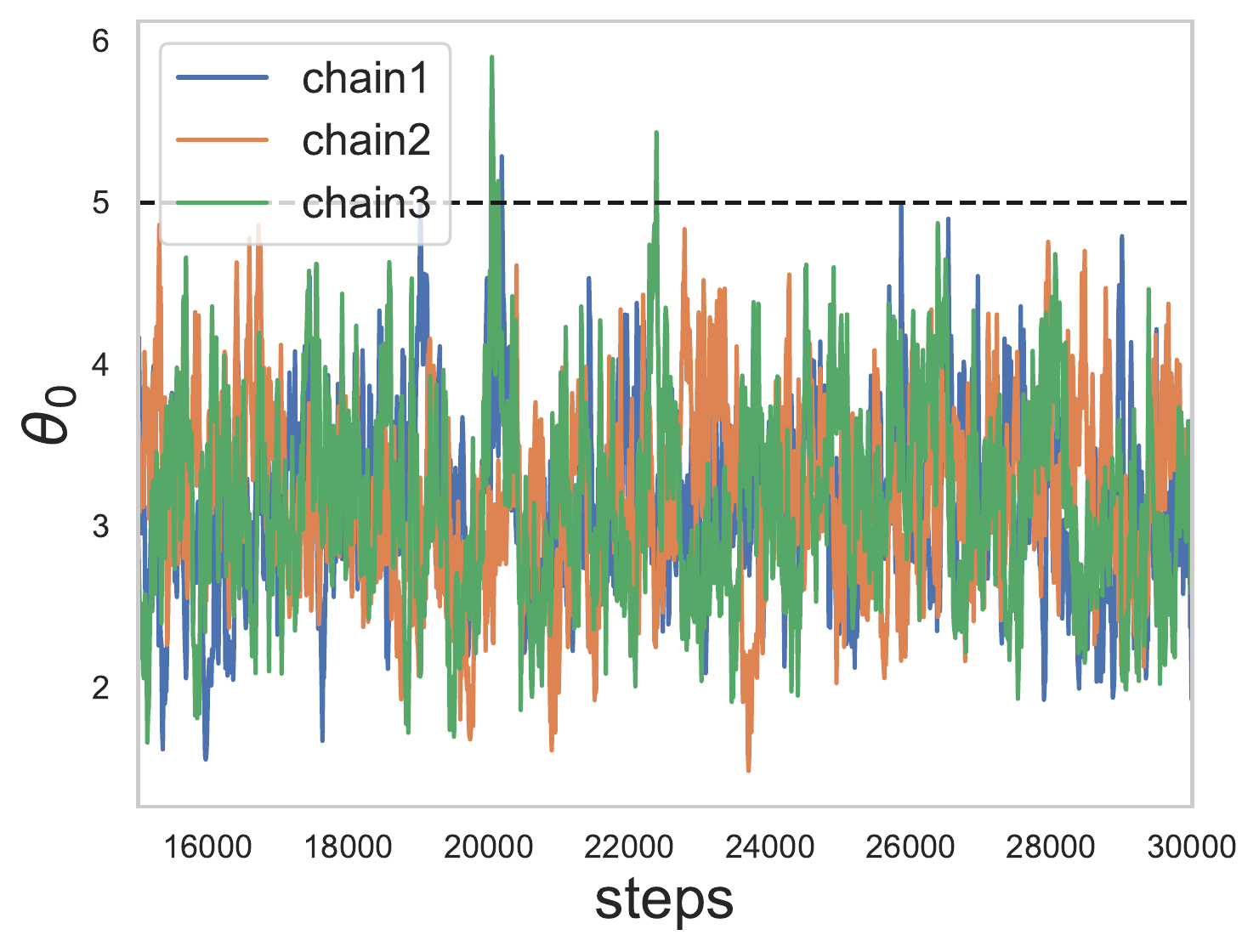}\\
            \includegraphics[width=0.45\linewidth]{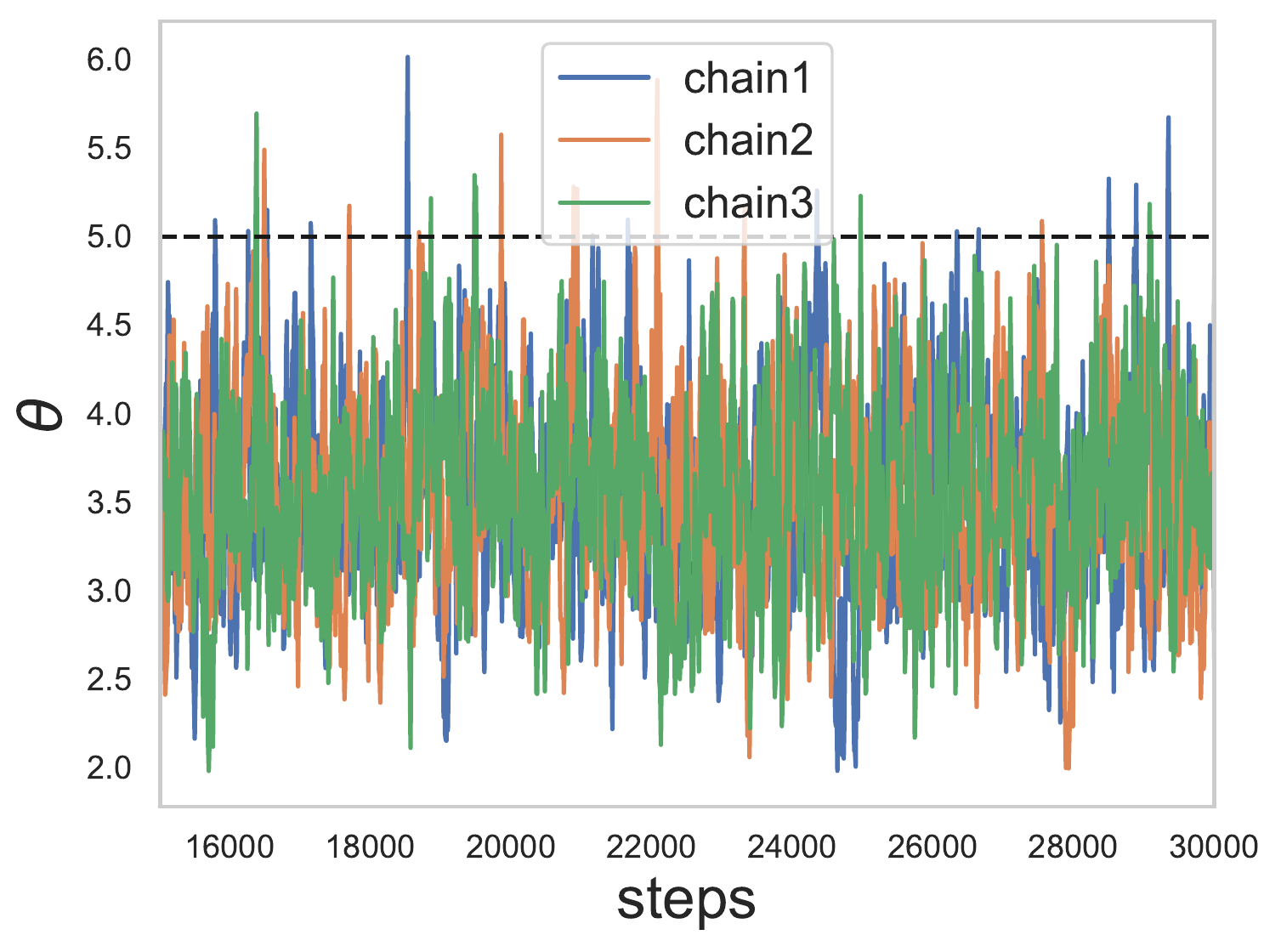}
                \includegraphics[width=0.45\linewidth]{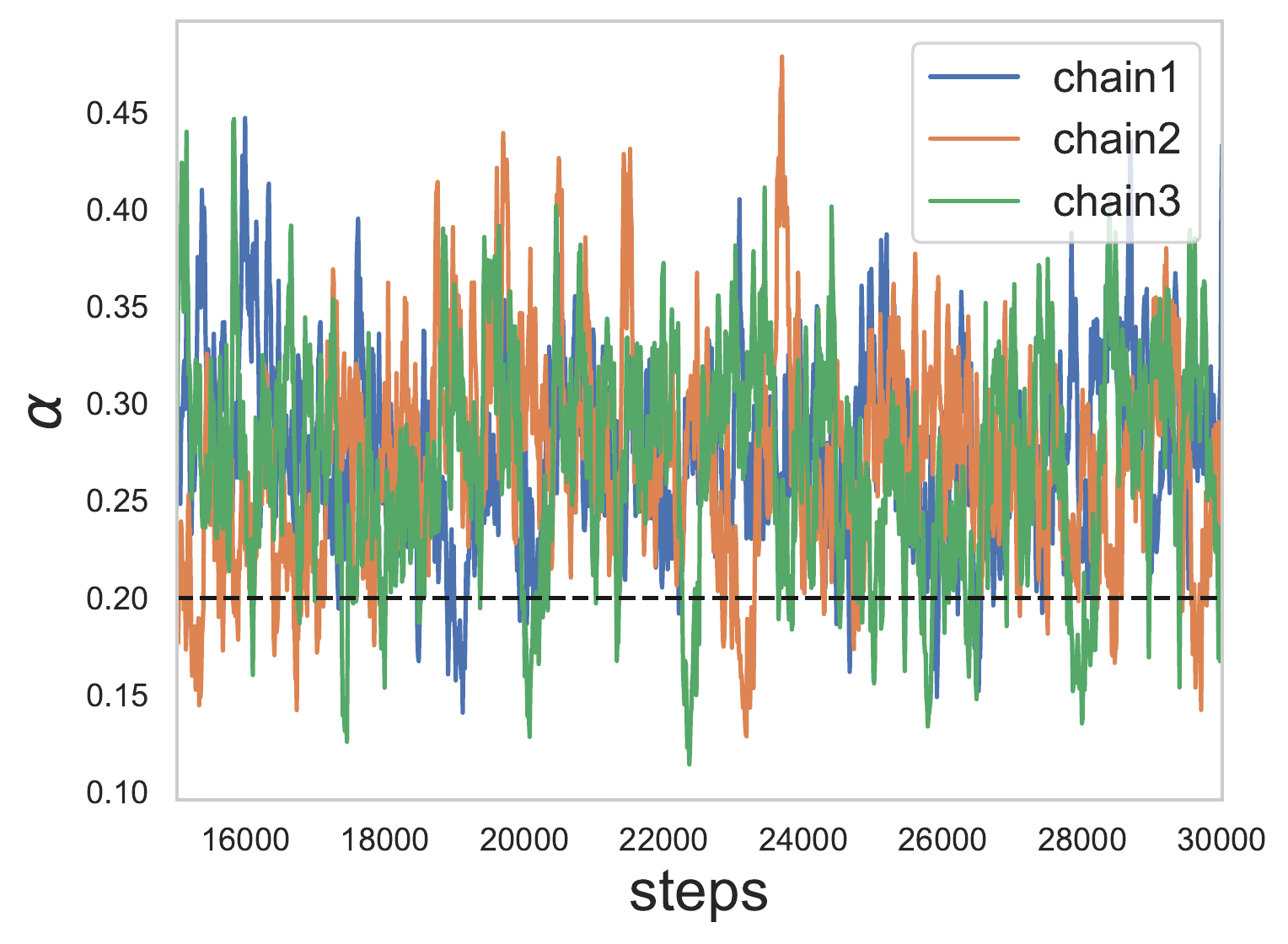}
    \caption{Posterior inference results for the dataset generated with $(\theta_0, \theta, \alpha) = (5.0, 5.0, 0.2)$. Log-marginal trace (upper left), sample traces for $\theta_0$ (upper right), $\theta$ (lower left), and $\alpha$ (lower right). Black dashed lines indicate true values.}
    \label{fig:data2_traces}
\end{figure}

\end{appendix}


\begin{thebibliography}{0}







\bibitem{Argiento}
\textsc{Argiento, R., Cremaschi, A.} and \textsc{Vannucci, M.} (2020). Hierarchical Normalized Completely Random Measures to Cluster Grouped Data. \textit{J. Amer. Statist. Assoc.} \textbf{115} 318-33.

\bibitem{AyedCaron}
\textsc{Ayed, F.} and \textsc{Caron, F}. (2019) Nonnegative Bayesian nonparametric factor models with completely random measures for community detection. arXiv:1902.10693.

\bibitem{Basbug}
\textsc{Basbug, M.} and \textsc{Engelhardt, B.} (2016). Hierarchical compound Poisson factorization. In International Conference on Machine Learning (pp. 1795-1803). PMLR.

\bibitem{BleiLDA}
\textsc{Blei, D.M., Ng, A.Y.} and 
\textsc{Jordan, M.I.} (2003). Latent Dirichlet Allocation. \textit{J. Machine Learning Research}, \textbf{3}, 993-1022.

\bibitem{Broderick1}
\textsc{Broderick, T, Jordan, M.I.,} and \textsc{Pitman, J.} (2013)
Clusters and features from combinatorial stochastic processes.
\emph{Statist. Sci.} \textbf{28}, 289-312.

\bibitem{Broderick2}
\textsc{Broderick, T, Mackey, L, Paisley, J,} and \textsc{Jordan, M.I.}
(2015). 
Combinatorial clustering and the beta negative binomial process. \emph{IEEE Transactions on Pattern Analysis and Machine Intelligence.} \textbf{37},
 290-306.

\bibitem{BroderickWilson}
\textsc{Broderick, T., Wilson, A.} and \textsc{Jordan, M.I.} (2018). Posteriors, conjugacy,
and exponential families for completely random measures. \textit{Bernoulli} \textbf{24} 3181–3221.



\bibitem{CamerLP1}
\textsc{Camerlenghi,F., Lijoi, A., Orbanz,P.} and \textsc {Pr\"unster, I.} (2019). Distribution theory for hierarchical processes. \emph{Ann. Statist.} \textbf{47}, 67-92.

\bibitem{CamerLP2}
\textsc{Camerlenghi,F., Lijoi, A.} and \textsc {Pr\"unster, I.} (2020). Survival analysis via hierarchically dependent mixture hazards. \emph{Ann. Statist.} to appear

\bibitem{CampbellTrait}
\textsc{Campbell, T., Cai, D.} and \textsc{Broderick, T.} (2018). Exchangeable trait allocations. 
\textit{Electronic Journal of Statistics}, \textbf{12} pp. 2290-2322.

\bibitem{Caron2012}
\textsc{Caron, F.} (2012) Bayesian nonparametric models for bipartite graphs. Neural Information Processing Systems (NIPS 2012), Lake Tahoe, USA, 2012.

\bibitem{CaronFox}
\textsc{Caron, F.} and \textsc{Fox, E.B.}(2017). 
Sparse graphs using exchangeable random
measures. \emph{JRSSB} \textbf{79},  1295-1366.

\bibitem{BoChen}
\textsc{Chen, B., Polatkan, G., Sapiro, G., Dunson, D.B.}, and \textsc{Carin, L.}.(2011) The hierarchical beta process for convolutional factor analysis and deep learning. \textit{In Proceedings of the 28th International Conference on International Conference on Machine Learning},  361-368. 2011.

\bibitem{DevroyeTryp}
\textsc{Devroye, L.} (1993). A triptych of discrete distributions related to the stable law. 
\textit{Statistics and Probability Letters}, \textbf{18}, 349-351.

\bibitem{DiBenedettopartition}
\textsc{Di Benedetto, G., Caron, F.} and \textsc{Teh, Y. W.} (2020).
Non-exchangeable random partition models for microclustering.\emph{Ann. Statist.}, to appear.


\bibitem{DykstraLaud}
\textsc{Dykstra, R.L.} and \textsc{Laud, P.W.} (1981). A Bayesian nonparametric approach to
reliability. \textit{Ann. Statist.} \textbf{9} 356–367.


\bibitem{Ferg1973}
\textsc{Ferguson, T.S.} (1973).
 A Bayesian Analysis of Some Nonparametric Problems. \emph{Ann. Statist.} \textbf{1} 209--230. 

\bibitem{FisherPoisson}
\textsc{Fisher, R.A.}(1922).
On the Interpretation of $\chi^{2}$ from Contingency Tables, and the Calculation of P.
\emph{Journal of the Royal Statistical Society} 
\textbf{85}, 87-94.

\bibitem{Fisher}
\textsc{Fisher, R.A., Corbet, S.A.,} and \textsc{Williams, C.B.} (1943) The relation between the number of species and the number of individuals in a random sample of an animal population. 
\textit{The Journal of Animal Ecology} 42-58.




\bibitem{Goldwater}
\textsc{Goldwater, S., Griffiths, T.} and \textsc{Johnson, M.} (2006). Interpolating between types and tokens by estimating power-law generators. \textit{Advances in neural information processing systems}, \textbf{18}, p.459.

\bibitem{GriffithsZ}
\textsc{Griffiths, T.L., and Ghahramani, Z.} (2006)
Infinite Latent Feature Models and the Indian Buffet Process.
In Advances in Neural Information Processing Systems 18 (NIPS-2005).

\bibitem{Griffiths1}
\textsc{Griffiths, T.L., and Ghahramani, Z.} (2011) The Indian buffet process: An introduction and review. \textit{The Journal of Machine Learning Research} 12 (2011): 1185-1224.

\bibitem{GuptaRestrict}
\textsc{Gupta, S.K., Phung, D.,} and \textsc{Venkatesh, S.}~(2012). A Bayesian nonparametric joint factor model for learning shared and individual subspaces from multiple data sources. In Proceedings of the 2012 SIAM International Conference on Data Mining (pp. 200-211). Society for Industrial and Applied Mathematics.

\bibitem{Croy}
\textsc{Heaukulani, C., and Roy, D. M.} (2016). The combinatorial structure of beta negative binomial processes. \textit{Bernoulli} \textbf{22}, 2301-2324.

\bibitem{hjort}
\textsc{Hjort, N. L.} (1990). Nonparametric Bayes estimators based on beta processes in models for life history data. \emph{Ann, Stat } \textbf{18}, 1259-1294.

\bibitem{HJL}
\textsc{Ho, M-W., James, L.F.} and \textsc{Lau, J.W.} (2021). Gibbs Partitions, Riemann-Liouville Fractional Operators, Mittag-Leffler Functions, and Fragmentations derived from stable subordinators. \textit{Journal of Applied Probability, to appear}

\bibitem{IJ2001}
\textsc{Ishwaran, H.} and \textsc{James, L.F.} (2001). 
Gibbs sampling methods for stick-breaking priors.  
\emph{J. Amer. Stat. Assoc.}, \textbf{96}, 161-173.

\bibitem{IJ2003}
\textsc{Ishwaran, H.} and \textsc{James, L.F.} (2003). Generalized weighted Chinese restaurant processes for species sampling mixture models. \textit{Statist. Sinica}, \textbf{13} 1211-1235.

\bibitem{IJ2004}
\textsc{Ishwaran, H.} and \textsc{James, L.F.} (2004). Computational methods for multiplicative intensity models using weighted gamma processes.  \emph{J. Amer. Stat. Assoc.}, \textbf{99}, 175-190.

\bibitem{IR2005}
\textsc{Ishwaran, H., and Rao, J. S.} (2005). Spike and slab variable selection: frequentist and Bayesian strategies. \emph{Ann. Stat. }, 730-773.

\bibitem{James2002}
\textsc{James, L.F.} (2002). Poisson process partition calculus with applications to exchangeable models and Bayesian nonparametrics. Unpublished manuscript. ArXiv math.PR/0205093.


\bibitem{James2005}
\textsc{James, L.F.} (2005)
Bayesian Poisson Process Partition Calculus with an Application to Bayesian L\'evy Moving Averages
\emph{Ann. Stat. }\textbf{33}, 1771-1799.

\bibitem{JamesNTR}
\textsc{James, L.F.} (2006).
Poisson calculus for
spatial neutral to the right processes.
\emph{Ann. Stat. }\textbf{34}, 416-440


\bibitem{James2017}
\textsc{James, L.F.}  (2017).
Bayesian Poisson Calculus for Latent Feature Modeling via Generalized Indian Buffet Process Priors.
\emph{Ann. Stat. }\textbf{45}, 2016-2045.

\bibitem{JamesStick}
\textsc{James, L.~F.} (2019). Stick-breaking Pitman-Yor processes given the species sampling size.
arXiv:1908.07186 [math.ST]

\bibitem{JLP2} \textsc{James, L. F., Lijoi, A.} and
\textsc{Pr\"{u}nster, I.} (2009). Posterior analysis for normalized random measures with independent increments. \emph{Scand. J. Stat.} \textbf{36} 76--97.

\bibitem{Kessler}
\textsc{Kessler, S., Nguyen, V., Zohren, S.} and \textsc{Roberts, S.}(2019) Hierarchical Indian Buffet Neural Networks for Bayesian Continual Learning. arXiv:1912.02290.

\bibitem{KimY}
\textsc{Kim, Y.} (1999). Nonparametric Bayesian estimators for counting processes. \emph{Ann. Stat.}, \textbf{27} 562-588.

\bibitem{KimLee}
\textsc{Kim, Y. and Lee, J.} (2001).
On posterior consistency of survival models.\emph{Ann.
Stat.}, \textbf{29} 666-686.


\bibitem{KnowlesThesis}
\textsc{Knowles, D. A. }(2012). Bayesian non-parametric models and inference for sparse and hierarchical latent structure. Ph.D. Thesis University of Cambridge

\bibitem{KnowlesAAS}
\textsc{Knowles, D.,} and \textsc{Ghahramani, Z.} (2011). Nonparametric Bayesian sparse factor models with application to gene expression modeling. \emph{The Annals of Applied Statistics}, \textbf{5}(2B), 1534-1552.

\bibitem{BuntineHPYtwitter}
\textsc{Lim, K.W., Buntine, W., Chen, C.} and \textsc{Du, L.} (2016). Nonparametric Bayesian topic modelling with the hierarchical Pitman–Yor processes. 
\textit{International Journal of Approximate Reasoning}, \textbf{78} 172-191.


\bibitem{Lo1982}
\textsc{Lo, A.Y.} (1982). Bayesian nonparametric statistical inference for Poisson point processes. \emph{Zeitschrift f{\"u}r Wahrscheinlichkeitstheorie und verwandte Gebiete}, \textbf{59}(1), 55-66.

\bibitem{Lo1984}
\textsc{Lo, A.Y.} (1984). On a Class of Bayesian Nonparametric Estimates: I. Density Estimates.
\textit{Ann. Statist.} \textbf{12}, 351 - 357.

\bibitem{LoWeng89}
\textsc{Lo, A.Y.} and \textsc{Weng, C. S.} (1989). On a class of Bayesian nonparametric estimates: II. Hazard rate estimates. \emph{Annals of the Institute of Statistical Mathematics,} \textbf{41}(2), 227-245.


\bibitem{Masoerotrait}
\textsc{Masoero, L., Camerlenghi, F., Favaro, S.} and \textsc{Broderick, T.} (2018).
Posterior representations of hierarchical completely random measures in trait allocation models. BNP-NeurIPS. 2018.

\bibitem{MasoeroBiometrika}
\textsc{Masoero, L., Camerlenghi, F., Favaro, S.} and \textsc{Broderick, T.} (2021). More for less: Predicting and maximizing genetic variant discovery via Bayesian nonparametrics. To appear\textit{ Biometrika}.









\bibitem{Penrose}
\textsc{Penrose, M. D.},and \textsc{Wade, A. R.} (2004). Random minimal directed spanning trees and Dickman-type distributions. \textit{Advances in Applied Probability}, \textbf{36}, 691-714.

\bibitem{Pit96}
\textsc{Pitman, J.} (1996). \textit{Some developments of the
Blackwell-MacQueen urn scheme.} Statistics, probability and game
theory, 245--267, IMS Lecture Notes Monogr. Ser., 30, Inst. Math.
Statist., Hayward, CA.

\bibitem{Pit97}
\textsc{Pitman, J.} (1997)
Partition structures derived from Brownian motion and stable
subordinators. \emph{Bernoulli} \textbf{3} 79-96


\bibitem{Pit06}
\textsc{Pitman, J.} (2006). \textit{Combinatorial stochastic
processes.} Lectures from the 32nd Summer School on Probability
Theory held in Saint-Flour, July 7--24, 2002. With a foreword by
Jean Picard. Lecture Notes in Mathematics, 1875. Springer-Verlag,
Berlin.

\bibitem{PitmanPoissonMix}
\textsc{Pitman, J.} (2017).
Mixed Poisson and negative binomial models for clustering and species sampling. Manuscript in preparation.

\bibitem{Quenouille}
\textsc{Quenouille, M. H.} (1949). A relation between the logarithmic, Poisson, and negative binomial series. \textit{Biometrics}, \textbf{5}, 162-164.

\bibitem{RagaBlei}
\textsc{Ranganath, R.,} and \textsc{Blei, D.M.} (2018) Correlated Random Measures. 
\textit{Journal of the American Statistical Association.} \textbf{113}  417-430.

\bibitem{Ragadeep}
\textsc{Ranganath, R., Tang, L., Charlin, L.} and \textsc{Blei, D.} (2015) Deep exponential families. In \textit{Artificial Intelligence and Statistics} (pp. 762-771). PMLR.


\bibitem{Scheinthesis}
\textsc{Schein, A.} (2019) \textit{Allocative Poisson Factorization for Computational Social Science.} PhD thesis, University of Massachusetts Amherst, 2019.

\bibitem{Soper}
\textsc{Soper, H.E.}(1922) \textit{Frequency arrays, illustrating the use of logical symbols in the study of statistical and other distributions.} Cambridge University Press.

\bibitem{TehPY}
\textsc{Teh, Y. W.} (2006). A hierarchical Bayesian language model based on Pitman-Yor processes. In \textit{Proceedings of the 21st International
Conference on Computational Linguistics and 44th Annual Meeting of the Association for Computational Linguistics,} 85--992.

\bibitem{TehG}
\textsc{Teh, Y.W., and Gorur, D.} (2009)
Indian Buffet Processes with Power-law Behavior.
NIPS 2009.


\bibitem{HDP}
\textsc{Teh, Y. W.; Jordan, M. I.; Beal, M. J.} and \textsc{ Blei, D. M.} (2006). Hierarchical Dirichlet Processes. \textit{Journal of the American Statistical Association.} \textbf{101}  1566–1581.

\bibitem{Thibaux}
\textsc{Thibaux, R., and Jordan, M. I. }(2007). Hierarchical beta processes and the Indian buffet process. In International conference on artificial intelligence and statistics (pp. 564-571).

\bibitem{Titsias}
\textsc{Titsias, M. K.} (2008). The infinite gamma-Poisson feature model. Advances in Neural Information Processing Systems. 2008.




\bibitem{Wood}
\textsc{Wood, F., Gasthaus, J., Archambeau, C., James, L. F. and Teh, Y. W.} (2011). The Sequence Memoizer. \textit{Communications of the ACM (Research Highlights)} \textbf{54,} 91--98.
 
\bibitem{ZhouCarin2015}
\textsc{Zhou, M.}, and \textsc{Carin, L.} (2015). 
Negative Binomial Process Count and Mixture Modeling,
\textit{IEEE Transactions on Pattern Analysis and Machine Intelligence}, \textbf{37}, 307–320. 

\bibitem{ZhouPGBN}
\textsc{Zhou,M., Cong,Y.}and  \textsc{Chen, B.} (2016). Augmentable gamma belief networks. T\textit{he Journal of Machine Learning Research}, \textbf{17} 5656-5699.

\bibitem{ZhouFoF}
\textsc{Zhou,M., Favaro,S.}and  \textsc{Walker, S.G.} (2017) Frequency of
Frequencies Distributions and Size-Dependent Exchangeable Random Partitions, 
\textit{J. Amer. Statist. Assoc.} \textbf{112}, 1623-1635.

\bibitem{Zhou1}
\textsc{Zhou, M., Hannah, L., Dunson, D.}, and \textsc{Carin, L.} (2012). Beta-negative binomial process and Poisson factor analysis. AISTATS, 1462-1471. PMLR, 2012.

\bibitem{ZhouPadilla2016}
\textsc{Zhou,M., Madrid-Padilla, O.H.} and  \textsc{Scott, J.G.} (2016) Priors
for Random Count Matrices Derived from a Family of Negative Binomial Processes. 
\textit{J. Amer. Statist. Assoc.} \textbf{111}, 1144-1156.

\bibitem{ZhoudHIBP}
\textsc{Zhou, M., Yang, H., Sapiro, G., Dunson, D.} and \textsc{Carin, L.} (2011). Dependent hierarchical beta process for image interpolation and denoising. In Proceedings of the Fourteenth International Conference on Artificial Intelligence and Statistics (pp. 883-891). JMLR Workshop and Conference Proceedings.

\end{thebibliography}
\end{document}